\documentclass{article}

\usepackage[T1]{fontenc}				% codificação da saída
\usepackage{lmodern}
\usepackage[latin1]{inputenc}			% codificação do arquivo fonte
\usepackage{url}
\usepackage{amssymb,amsmath,amsthm}			% vários símbolos e letras matemáticos da AMS
\usepackage{cancel}						% para riscar termos nas equações
\usepackage{mathrsfs}					% p/ usar a fonte\mathscr
\usepackage[final,pagebackref,breaklinks]{hyperref}	% p/ hyperlinks no DVI e no PDF
\usepackage{microtype}

\DeclareFontFamily{U}{matha}{\hyphenchar\font45}
\DeclareFontShape{U}{matha}{m}{n}{
      <5> <6> <7> <8> <9> <10> gen * matha
      <10.95> matha10 <12> <14.4> <17.28> <20.74> <24.88> matha12
      }{}
\DeclareSymbolFont{matha}{U}{matha}{m}{n}
\DeclareMathSymbol{\obot}{2}{matha}{"6B}

\allowdisplaybreaks[2]

\renewcommand{\b}{\beta}
\renewcommand{\d}{\mathrm{d}}								% d de distância ou de derivada
\renewcommand{\H}{\mathbb{H}}
\renewcommand{\l}{\lambda}
\renewcommand{\L}{\mathbb{L}}
\renewcommand{\o}[1]{\mathbb{O}_{k_{#1}}^{n_{#1}}}
\renewcommand{\O}{\mathbb{O}}
\renewcommand{\S}{\mathbb{S}}
\renewcommand{\ss}[1]{\mathbb{S}_{k_{#1}}^{n_{#1}}}
\renewcommand{\t}{\mathrm{t}}

\DeclareMathOperator{\id}{Id}

\DeclareMathOperator{\spa}{span}

\newcommand{\af}[2]{\alpha_{f}\left(#1,#2\right)}			% segunda forma fundamental de f
\newcommand{\aF}[2]{\alpha_{F}\left(#1,#2\right)}			% segunda forma fundamental de F
\newcommand{\ai}[2]{\alpha_{\imath}\left(#1,#2\right)}		% segunda forma fundamental de i
\newcommand{\al}[2]{\alpha\left(#1,#2\right)}				% segunda forma fundamental
\newcommand{\BK}{\mathbf{K}}
\newcommand{\BL}{\mathbf{L}}
\newcommand{\BR}{\mathbf{R}}
\newcommand{\BS}{\mathbf{S}}
\newcommand{\BT}{\mathbf{T}}
\newcommand{\cqd}{{\scriptsize\textbullet}\vspace{1.5ex}}
\newcommand{\cE}{\mathcal{E}}
\newcommand{\cF}{\mathcal{F}}

\newcommand{\e}{\varepsilon}
\newcommand{\E}{\mathbb{E}}

\newcommand{\g}{\gamma}
\newcommand{\G}{\Gamma}
\newcommand{\h}[1]{\mathbb{H}_{k_{#1}}^{n_{#1}}}
\newcommand\hksqrt[2][]{\mathpalette\DHLhksqrtA{{#1}{#2}}}
\def\DHLhksqrtA#1#2{\setbox0=\hbox{$#1\DHLhksqrtB#2$}\dimen0=\ht0
	\advance\dimen0-0.2\ht0
	\setbox2=\hbox{\vrule height\ht0 depth -\dimen0}%
	{\box0\lower0.4pt\box2}%
}
\def\DHLhksqrtB#1#2{\def\a{#1}\def\b{}\ifx\a\b\sqrt{#2\,}\else\sqrt[#1]{#2\,}\fi}
\newcommand{\hO}{\hat{\mathbb{O}}}

\newcommand{\interno}[2]{\left<#1,#2\right>}				% produto interno

\newcommand{\n}{\nabla}

\newcommand{\nbar}{\bar\nabla}
\newcommand{\nbarperp}{{\bar\nabla}^\perp}

\newcommand{\nperp}{\nabla^\perp}
\newcommand{\ntil}{\tilde{\nabla}}

\newcommand{\op}{\obot}

\newcommand{\R}{\mathbb{R}}
\newcommand{\RN}{{\mathbb{R}^N}}
\newcommand{\set}[2]{\left\{\left. #1 \, \right| \, #2\right\}}
\newcommand{\sss}{\Leftrightarrow}
\newcommand{\TK}{\tilde{\mathbf{K}}}
\newcommand{\TL}{\tilde{\mathbf{L}}}
\newcommand{\TR}{\tilde{\mathbf{R}}}
\newcommand{\TS}{\tilde{\mathbf{S}}}
\newcommand{\TT}{\tilde{\mathbf{T}}}
\newcommand{\U}{\mathcal{U}}
\newcommand{\V}{\mathcal{V}}
\newcommand{\x}{\times}
\newcommand{\XO}{\o{1} \x \cdots \x \o{\ell}}
\newcommand{\z}{\zeta}

\newtheoremstyle{normal}% name
	{1}%			Space above
	{1}%			Space below
	{\slshape}%		Body font
	{}%				Indent amount 1
	{\bfseries}%	Theorem head font
	{.}%			Punctuation after theorem head
	{.5em}%			Space after theorem head 2
	{}%				Theorem head spec (can be left empty, meaning ?normal?)
\theoremstyle{definition}
	\newtheorem{df}{Definition}[section]

\theoremstyle{normal}
	\newtheorem{cor}[df]{Corollary}
	\newtheorem{lem}[df]{Lemma}
	\newtheorem{prop}[df]{Proposition}
	\newtheorem{teo}[df]{Theorem}

	\newtheorem{obs}[df]{Remark}
	
\newenvironment{prova}[1][Proof]{\vspace{1ex}\noindent\textit{#1. }}{\qed \vspace{1ex}}

\newcounter{claim}[df]

\newenvironment{afi}[1]{%
	\vspace{1ex}\noindent %
	\refstepcounter{claim} %
	\underline{Claim \theclaim:} #1 \par
}{ %
	\checkmark
	\par\vspace{1ex}
}

\newenvironment{enum}{
	\begin{enumerate}
	\setlength{\topsep}{0pt}
	\setlength{\itemsep}{0pt}
	\setlength{\parskip}{0pt}
	\setlength{\parsep}{0pt}
}{ %
	\end{enumerate}
}

\makeindex

\title{Isometric immersions in products of many space forms: an introductory study}
\author{Bruno Mendonça Rey dos Santos}

\begin{document}
\maketitle

\begin{abstract}
	This article begins the theory of submanifolds in products of 2 or more space forms. The tensors $\BR$, $\BS$ and $\BT$ defined by Lira, Tojeiro and Vitório at \cite{LTV} and the Bonnet theorem proved by them are generalized for the product of many space forms. Besides, some examples given by Mendonça and Tojeiro at \cite{MT} and the reduction of codimension theorem proved by them are also generalized.
\end{abstract}


\tableofcontents
	\section{Introduction}
		A \textbf{space form} is a (semi-)riemannian manifold whose sectional curvature is constant. To us, a space form is also connected and simply connected. Thus, up to isometries, we just have three kinds of space forms:
\begin{enumerate}
	\item the euclidean space $\R^n$, which is the space form whose sectional curvature is $0$;
	\item the sphere $\S_k^n$, which is the space form whose sectional curvature $k$ is positive;
	\item and the hyperbolic space $\H_k^n$, which os the space form whose sectional curvature $k$ is negative.
\end{enumerate}
We will use the following models:
\begin{align*}
	\S_k^n = \set{x \in \R^{n+1}}{\|x\|^2 = \frac{1}{k}}, && \text{if} \ k > 0;\\
	\S_k^n = \set{x \in \L^{n+1}}{\|x\|^2 = \frac{1}{k}}, && \text{if} \ k < 0;\\
	\H_k^n = \set{x \in \L^{n+1}}{\|x\|^2 = \frac{1}{k} \ \text{and} \ x_0 > 0}, && \text{if} \ k < 0.
\end{align*}
Here, $\L^{n+1}$ is the vector space $\R^{n+1}$ with a the following inner product: $\interno{x}{y} := -x_0y_0 + \sum\limits_{i=1}^n x_iy_i$. To make things easier, we will denote the space form with dimension $n$ and sectional curvature $k$ by $\o{}$

The authors of \cite{LTV} defined the tensors $\BR$, $\BS$ and $\BT$ associated to an isometric immersion $f \colon M^m \to \o{i}\x\o{2}$. They proved some results involving this tensors and they found the fundamental equations for isometric immersions in products of two space forms. Using these fundamental equations, they proved a Bonnet theorem for isometric immersions in products of two space forms.

The authors of \cite{MT} constructed some examples of isometric immersions in $\o{1}\x\o{2}$, they proved some reduction of codimension results, they classified all parallel immersions in $\o{1}\x\o{2}$ and they reduced the classification of umbilical immersions $f \colon M^m \to \o{1}\x\o{2}$ (with $m \geq 3$ and $k_1+k_2 \ne 0$) to the classification of isometric immersions in $\S^n\x\R$ and $\H^n\x\R$. The umbilical immersions in $\S^n\x\R$ were classified in \cite{MT2}.

The present article is a beginning of studying isometric immersions in product of many space forms, 2 or more. So here the tensors $\BR$, $\BS$ and $\BT$ must be defined for isometric immersions $f \colon M^m \to \XO$. Section 2 takes care of these generalized definitions and of the equations involving the new tensors.

Section 3 generalizes some examples given at \cite{MT} and Section 4 generalizes the reduction of codimension theorem of \cite{MT}.

Finally, the last section of this article generalizes the Bonnet theorem of \cite{LTV} for isometric immersion in $\XO$. 

	\section{Isometric immersions in $\XO$}
		\subsection{The inclusion function $\imath \colon \XO \hookrightarrow \RN$}

Let $k_1$, $\cdots$, $k_\ell$ be real numbers and let $\upsilon, \tau \colon \R \to \{0,1\}$ be the functions given by
\[\upsilon(x) :=
\begin{cases}
	1, \ \text{if} \ x \ne 0;\\
	0, \ \text{if} \ x = 0;
\end{cases} \quad \text{and} \quad
\tau(x) := \begin{cases}
	1, \ \text{if} \ x < 0;\\
	0, \ \text{if} \ x \geq 0.
\end{cases}\]
Let $\R^{N_i} = \R_{\tau(k_i)}^{n_i+\upsilon(k_i)}$, for each $i \in \{1, \cdots, \ell\}$, and let $\RN = \R^{N_1} \x \cdots \x \R^{N_\ell}$. Here, $\R^n_0 = \R^n$ and $\R^n_1 = \L^n$. So there exists a canonical inclusion $\imath \colon \XO \hookrightarrow \RN$ given by $\imath(x_1, \cdots, x_\ell) = (x_1, \cdots, x_\ell)$, where $x_i \in \o{i}\subset \R^{N_i}$.

Lets denote $\hO = \XO$. Then the codimension of $\imath \colon \hO \to \RN$ is the number of elements of the set $J := \set{i \in \{1, \cdots, \ell\}}{k_i \ne 0}$.

\begin{lem}\label{tangentei}
	If $J := \set{i \in \{1, \cdots, \ell\}}{k_i \ne 0}$ and $x = (x_1, \cdots, x_\ell) \in \hO$, with $x_i \in \o{i}$, then
	\[T_x \hO =	\set{X \in \RN}{X_i \perp x_i, \forall i \in J},\]
	where $X = (X_1, \cdots, X_\ell)$, with $X_i \in \R^{N_i}$.
\end{lem}

\begin{proof}
	Let $x \in \hO$ and $X \in T_x \hO$. Thus there is a differentiable curve $\b \colon I \to \O$ such that $\b(0) = x$ and $\b'(0) = X$.

	But $\b(t) = \big( \b_1(t), \cdots, \b_\ell(t) \big)$ and $\|\b_i(t)\|^2 = \frac{1}{k_i}$, for all $i \in J$. Then $\interno{\b_i(t)}{\b_i'(t)} = 0$, for all $i \in J$. It follows that $\interno{x_i}{X_i} = \interno{\b_i(0)}{\b_i'(0)} = 0$, for each $i \in J$. Thus $T_x \hO \subset \set{X \in \RN}{X_i \perp x_i, \forall i \in \{1, \cdots, \ell\}}$.

	Lets consider
	\[V_i = \begin{cases}
		\{x_i\}^\perp, & \text{if} \ i \in J; \\
		\R^{N_i}, & \text{if} \ i \not\in J;
	\end{cases}\]
	where $\{x_i\}^\perp$ is the orthogonal complement of $\spa\{x_i\}$ in $\R^{N_i}$. Thus
	\[\set{X \in \RN}{X_i \perp x_i, \forall i \in J} = V_1 \x \cdots \x V_\ell.\]
	Since $T_x \hO$ and $\set{X \in \RN}{X_i \perp x_i, \forall i \in J}$ have the same dimension, they are the same subspace of $\RN$.
\end{proof}

\begin{lem} \label{lem: pi_j}
	Let $\RN = \R^{N_1} \times \cdots \times \R^{N_\ell}$ and $\pi_i \colon \RN \to \RN$ be the orthogonal projections given by $\pi_i(x_1, \cdots, x_\ell) = (0, \cdots, 0, x_i, 0, \cdots, 0)$. If $k_i \ne 0$, then the field $\left.\pi_i\right|_{\hO} = \pi_i \circ \imath$ is normal to $\hO$.
\end{lem}

\begin{proof}
	Let $X \in T_x \hO$ a tangent vector, thus $X_i \perp x_i$ (by Lemma \ref{tangentei}). It follows that $\interno{(\pi_i \circ \imath)(x)}{X} = \interno{x_i}{X_i} = 0$. Therefore $(\pi_i \circ \imath)$ is normal to $\hO$.
\end{proof}

Since the codimension of $\hO$ in $\RN$ is equal to  the number of elements of the set $J := \set{i \in \{1, \cdots, \ell\}}{k_i \ne 0}$, then
\begin{equation}\label{Tiperp}
	T_x^\perp \hO = \spa\set{\pi_i(x)}{i \in J} = \spa\{-k_1 \pi_1(x), \cdots, -k_\ell \pi_\ell(x)\}.
\end{equation}

Now, let $\nbar$ and $\ntil$ be the Levi-Civita connections in $\hO$ and $\RN$ respectively, and let $\alpha_{\imath}$, $\bar{\mathcal{R}}$ and $A_\eta^\imath$ be (respectively) the second fundamental form of $\imath$,  the curvature tensor of $\hO$ and the shape operator in the normal direction $\eta$ given by $\interno{A_\eta^\imath X}{Y} = \interno{\ai{X}{Y}}{\eta}$.

\begin{lem}\label{lem: ai}
	For all $X, Y \in \G \left(T \hO \right)$, it holds that
	\begin{equation}\label{eq: ai}
		\ai{X}{Y} = -\sum_{i=1}^\ell \interno{\pi_i X}{Y}k_i(\pi_i \circ \imath),
	\end{equation}
	 Besides, if $k_i \ne 0$, then $A^{\imath}_{\pi_i \circ \imath} = -\pi_i$.
\end{lem}

\begin{proof}
	Let $x \in \hO$ and $X \in T_x \hO$. If $k_i \ne 0$, then $X_i \perp x_i$ and $\pi_i X = (0, \cdots, X_i, \cdots, 0) \in T_x \hO$.
	On the other side,
	\[\ntil_X (\pi_i \circ \imath)(x) = \d {\pi_i}_{\imath(x)}\left( \d \imath_x X \right) = \pi_i \imath_*X = \imath_* \pi_i X \ \Rightarrow \ A^{\imath}_{\pi_i \circ \imath} = -\pi_i.\]

	Let $J = \set{i \in \{1, \cdots, \ell\}}{k_i \ne 0}$. Then
	\begin{align*}
		& \ai{X}{Y} = \sum_{i\in J} \interno{\ntil_X \imath_* Y}{\pi_i\circ \imath}\frac{\pi_i\circ \imath}{\|\pi_i\circ \imath\|^2} = \sum_{i\in J} \interno{\ai{X}{Y}}{\pi_i\circ \imath}k_i(\pi_i\circ \imath) = \\
		& = -\sum_{i \in J} \interno{\pi_i X}{Y} k_i(\pi_i\circ \imath) = -\sum_{i=1}^\ell \interno{\pi_i X}{Y} k_i(\pi_i\circ \imath). \qedhere
	\end{align*}
\end{proof}

\begin{lem} \label{lem: curv1}
	$\displaystyle \bar{\mathcal R}(X,Y)Z = \sum\limits_{i=1}^\ell k_i \left(\pi_i X \wedge \pi_i Y \right) Z$, where $(A\wedge B) C := \interno{B}{C}A - \interno{A}{C}B$.
\end{lem}

\begin{proof}
	Using Gauss equation,
	\begin{align*}
		& \bar{\mathcal R}(X,Y)Z = A_{\ai{Y}{Z}}^{\imath} X - A_{\ai{X}{Z}}^{\imath} Y = \\
		& \stackrel{\eqref{eq: ai}}{=} -\sum_{i=1}^\ell \interno{\pi_i Y}{Z} A^{\imath}_{k_i(\pi_i \circ \imath)} X + \sum_{i=1}^\ell \interno{\pi_i X}{Z} A^{\imath}_{k_i(\pi_i \circ \imath)} Y = \\
		& = \sum_{i=1}^\ell \interno{\pi_i Y}{Z} k_i \pi_i X - \sum_{i=1}^\ell \interno{\pi_i X}{Z} k_i \pi_i Y = \\
		&= \sum_{i=1}^\ell k_i\left[\interno{\pi_i Y}{Z} \pi_i X - \interno{\pi_i X}{Z} \pi_i Y \right]. \qedhere
	\end{align*}
\end{proof}

\begin{lem}\label{pi_i parallel}
	The tensor $\pi_i \colon T \hO \to T \hO$ is parallel, that is, $(\n_X \pi_i) Y = 0, \ \forall X, Y \in \G\left(T \hO \right)$.
\end{lem}

\begin{proof}
	Here we use $\pi_i$ to denote the orthogonal projection $\pi_i \colon \RN \to \RN$ and also the restrictions $\left.\pi_i\right|_{\hO} \colon \hO \to \hO$ and $\left.\pi_i\right|_{T_x \hO} \colon T_x \hO \to T_x \hO$.
	
	Let $X,Y \in \G\left(T \hO \right)$, thus
	\begin{align*}
		& \imath_* \pi_i\nbar_X Y = \pi_i \left[\imath_* \nbar_X Y \right] = \pi_i \left[ \ntil_X \imath_*Y \right] - \pi_i \left[\ai{X}{Y} \right] = \\
		& = \pi_i \left[ \ntil_X \imath_*Y \right] - \pi_i\left[ -\sum_{j=1}^\ell \interno{\pi_j X}{Y}k_j(\pi_j \circ \imath) \right] = \\
		&= \pi_i \left[ \ntil_X \imath_*Y \right] + \interno{\pi_i X}{Y}k_i(\pi_i \circ \imath).
	\end{align*}
	
	But $Y \colon \hO \to \RN$ and $\imath_* Y = Y$. Thus
	\begin{align*}
		& \pi_i \left[ \ntil_X \imath_*Y \right] = \pi_i\left[\d (\imath_* Y) X \right] = \pi_i\left[\d Y \cdot X \right] = \d \left(\pi_i \circ Y \right)\cdot X = \ntil_X \pi_i Y = \\
		& = \ntil_X \imath_* \pi_i Y = \imath_* \nbar_X \pi_i Y + \ai{X}{\pi_i Y} \stackrel{\eqref{eq: ai}}{=} \imath_* \nbar_X \pi_i Y - \interno{\pi_i X}{\pi_i Y}k_i (\pi_i \circ \imath) = \\
		& = \imath_* \nbar_X \pi_i Y - \interno{\pi_i X}{Y}k_i (\pi_i \circ \imath).
	\end{align*}
	Therefore $\imath_* \pi_i\nbar_X Y = \imath_* \nbar_X \pi_i Y$,  and hence $\left( \n_X \pi_i \right) Y = 0$.
\end{proof}

		\subsection{The $\BR_i$, $\BS_i$ and $\BT_i$ tensors}

Let $M^m$ be a riemannian manifold and $f\colon M^m \to \hO$ be an isometric immersion. We will call $\mathcal R$ and $\mathcal{R}^\perp$ the curvature tensors on the tangent bundle ($TM$) and on the normal bundle ($T^\perp M$) of $f$, respectively. Let also $\alpha := \alpha_f \colon (TM \x TM) \to T_f^\perp M$ be the second fundamental form of $f$ and $A_\eta := A^f_\eta \colon TM \to TM$ be the shape operator in the normal direction $\eta$, given by $\interno{A_\eta X}{Y} = \interno{\al{X}{Y}}{\eta}$, for all $X, Y \in TM$.

\begin{df}
	Let \index{$\BL_i$} $\BL_i \colon TM \to T \hO$, $\BK_i\index{$\BK_i$} \colon T^\perp M \to T \hO$, \index{$\BR_i$} $\BR_i \colon TM \to TM$, \index{$\BS_i$} $\BS_i \colon TM \to T^\perp M$ and \index{$\BT_i$} $\BT_i \colon T^\perp M \to T^\perp M$ be given by
	\begin{center}
	\begin{tabular}{ll}
		$\BL_i := \BL_i^f := \pi_i\circ {f}_*$; & $\BK_i := \BK_i^{f} := \left.\pi_i\right|_{T^\perp M}$; \\
		$\BR_i := \BR_i^f := \BL_i^\t\BL_i$ ; & $\BS_i := \BS_i^f := \BK_i^\t\BL_i$; \\
		$\BT_i := \BT_i^f := \BK_i^\t \BK_i$.
	\end{tabular}
	\end{center}
\end{df}

Lets remark that $\BR_i$ and $\BT_i$ are self-adjoint tensors. Besides, given $X \in T_xM$, $\xi \in T_x^\perp M$ and $Z \in T_{f(x)} \hO$, then
\[\begin{cases}
	\interno{\BL_i^\t Z}{X} = \interno{Z}{\BL_i X} = \interno{Z}{\pi_if_*X} = \interno{\pi_i Z}{f_*X} = \interno{\left(\pi_i Z \right)^T}{f_* X}, \\
	\interno{\BK_i^\t Z}{\xi} = \interno{Z}{\BK_i \xi} = \interno{Z}{\pi_i\xi} = \interno{\pi_i Z}{\xi} = \interno{\left(\pi_i Z \right)^\perp}{\xi},
\end{cases}\]
where $\left( \pi_i Z \right)^T$ and $\left( \pi_i Z \right)^\perp$ are the orthogonal projections of $\pi_i Z$ on $TM$ and $T^\perp M$, respectively. Thus $f_* \BL_i^\t Z = \left(\pi_i Z \right)^T$ and $\BK_i^\t Z = \left(\pi_i Z \right)^\perp$. Consequently
\begin{align*}
	& f_* \BR_i X = f_* \BL_i^\t\BL_i X = \left( \pi_i \BL_i X \right)^T = (\BL_i X)^T, \\
	& \BS_i X = \BK_i^\t\BL_i X = \left( \pi_i \BL_i X \right)^\perp = (\BL_i X)^\perp, \\
	& f_*\BS_i^\t \xi = f_* \BL_i^\t\BK_i \xi = \left( \pi_i \BK_i \xi \right)^T = (\BK_i \xi)^T, \\
	& \BT_i \xi = \BK_i^\t\BK_i \xi = \left( \pi_i \BK_i \xi \right)^\perp = (\BK_i \xi)^\perp.
\end{align*}
Therefore
\begin{equation}\label{eq: K L}
	\pi_if_* X = \BL_i X = f_*\BR_i X + \BS_i X \quad \text{and} \quad \pi_i \xi = \BK_i \xi = f_*\BS_i^\t \xi + \BT_i \xi.
\end{equation}

%\begin{figure}[!ht]
%	\centering
%	\psfrag{TO1}{$T\o{1}$}
%	\psfrag{TO2}{$T\o{2}$}
%	\psfrag{TM}{$f_*TM$}
%	\psfrag{Tperp}{$T^\perp M$}
%	\psfrag{LX}{$\BL_j X$}
%	\psfrag{RX}{$f_*\BR_j X$}
%	\psfrag{SX}{$\BS_j X$}
%	\psfrag{X}{$f_*X$}
%	\psfrag{KN}{$\BK_j \eta$}
%	\psfrag{TN}{$\BT_j \eta$}
%	\psfrag{StN}{$f_*\BS_j^\t \eta$}
%	\psfrag{N}{$\eta$}
%	\includegraphics[bb=0 0 319 260]{figuras/tensores.eps}
%	\caption{Tensores $\BK_j$, $\BL_j$, $\BR_j$, $\BS_j$ e $\BT_j$.}
%	\label{figura: tensores}
%\end{figure}

\begin{obs} If $X \in T_x M$ and $\xi \in T_x^\perp M$, then
	\begin{align*}
		& f_* \BR_i X + \BS_i X = \pi_i f_* X = f_* X  - \sum_{\substack{j=1\\ j\ne i}}^\ell \pi_j f_* X = f_* X - \sum_{\substack{j=1\\ j\ne i}}^\ell \left( f_* \BR_j X + \BS_j X \right). \\
		& f_* \BS_i^\t \xi + \BT_i \xi = \pi_i \xi = \xi - \sum_{\substack{j=1\\ j\ne i}}^\ell \pi_j \xi = \xi - \sum_{\substack{j=1\\ j\ne i}}^\ell \left( f_* \BS_j^\t \xi + \BT_j \xi \right).
	\end{align*}
	Thus
	\begin{align*}
		\BR_i &= \id - \sum_{\substack{j=1\\ j\ne i}}^\ell \BR_j, & \BS_i &= - \sum_{\substack{j=1\\ j\ne i}}^\ell \BS_j, & \BT_i &= \id - \sum_{\substack{j=1\\ j\ne i}}^\ell \BT_j,
	\end{align*}
	that is,
	\begin{align}\label{somas}
		\sum_{i=1}^\ell \BR_i &= \id|_{T_x M}, & \sum_{i=1}^\ell \BS_i &= 0, & \sum_{i=1}^\ell \BT_i &= \id|_{T_x^\perp M}.
	\end{align}
\end{obs}

\begin{lem}\label{lem: R S T}
	The following equations hold:
	\begin{align}
		 & \BS_i^\t\BS_i = \BR_i(\id - \BR_i), && \BT_i\BS_i = \BS_i( \id - \BR_i),  && \BS_i\BS_i^\t = \BT_i(\id - \BT_i), \label{eq: RST}\\
		 & \BS_i^\t\BS_j \stackrel{i \ne j}{=} -\BR_i\BR_j, && \BT_i \BS_j \stackrel{i \ne j}{=} -\BS_i\BR_j, && \BS_i \BS_j^\t \stackrel{i\ne j}{=} -\BT_i\BT_j. \label{eq: RST2}
	\end{align}
\end{lem}

\begin{proof}
	Let $X, Y \in TM$ and $\xi \in T^\perp M$. Hence
	\begin{multline*}
	f_* \BR_i X + \BS_i X = \pi_if_* X = {\pi_i}^2 f_* X = \pi_i (f_*\BR_i X + \BS_i X) = \\
	= f_* \BR_i^2 X + \BS_i\BR_i X + f_*\BS_i^\t\BS_i X + \BT_i\BS_i X.
	\end{multline*}
	
	Consequently, $\BR_i X = \BR_i^2 X + \BS_i^\t\BS_i X$ and $\BS_i X = \BS_i\BR_i X + \BT_i\BS_i X$. Therefore $\BS_i^\t\BS_i = \BR_i(\id - \BR_i)$ and $\BT_i\BS_i = \BS_i(\id - \BR_i)$.
	
	Analogously,
	\begin{multline*}
	f_* \BS_i^\t \xi + \BT_i \xi = \pi_i \xi = {\pi_i}^2 \xi = \pi_i \left(f_* \BS_i^\t\xi + \BT_i\xi \right) = \\
	= f_* \BR_i\BS_i^\t \xi + \BS_i\BS_i^\t \xi + f_*\BS_i^\t\BT_i \xi + \BT_i^2 \xi.
	\end{multline*}
	Therefore, $\BT_i \xi = \BS_i\BS_i^\t \xi + \BT_i^2 \xi$ and $\BS_i\BS_i^\t = \BT_i(\id - \BT_i)$.
	
	On the other side,
	\begin{multline*}
	0 = \pi_i \pi_j f_* X = \pi_i \left( f_* \BR_j X + \BS_j X \right) = \\
	= f_* \BR_i\BR_j X + \BS_i\BR_j X + f_* \BS_i^\t\BS_j X + \BT_i\BS_j X.
	\end{multline*}
	Thus $\BS_i^\t\BS_j = -\BR_i\BR_j$ and $\BT_i\BS_j = -\BS_i\BR_j$.
	
	Analogously,
	\[0  = \pi_i \pi_j \xi = \pi_i \left( f_* \BS_j^\t \xi + \BT_j \xi \right) = f_* \BR_i\BS_j^\t \xi + \BS_i\BS_j^\t \xi + f_* \BS_i^\t\BT_j \xi + \BT_i\BT_j \xi.\]
	Hence $\BS_i^\t\BT_j = -\BR_i\BS_j^\t$ and $\BT_i\BT_j = -\BS_i\BS_j^\t$.
\end{proof}

\begin{obs}\label{observação: autovalor R}
	If $\l$ is an eigenvalue of $\BR_i$ and $X$ is an eigenvector associated to $\l$, then
	\[ \l (1 - \l) \|X\|^2 = \interno{\BR_i (\id - \BR_i) X}{X}\stackrel{\eqref{eq: RST}}{=} \interno{\BS_i^\t\BS_i X}{X} = \| \BS_i X\|^2 \geq 0.\]
	Therefore $\l \in [0,1]$. Analogously, using the third equation of \eqref{eq: RST}, the eigenvectors of $\BT_i$ are also in $[0,1]$.
\end{obs}

\begin{lem}
	The following equations hold:
	\begin{align}
		& (\n_X \BR_i)Y = A_{\BS_i Y}X + \BS_i^\t \al{X}{Y}, \label{eq: derivada R} \\
		& \left. \begin{aligned}
			& (\n_X \BS_i)Y = \BT_i \al{X}{Y} - \al{X}{\BR_i Y}, \\
			& \big(\n_X \BS_i^\t \big) \xi = A_{\BT_i \xi} X - \BR_i A_\xi X,
		\end{aligned}\right\} \label{eq: derivada S} \\
		& (\n_X \BT_i)\xi = -\BS_i A_\xi X - \al{X}{\BS_i^\t\xi}. \label{eq: derivada T}
	\end{align}
\end{lem}

\begin{proof}
	Let $X,Y \in \G(TM)$ and $\xi \in \G\left( T^\perp M \right)$. Because of Lemma \ref{pi_i parallel}, we have that
	\begin{align*}
		&0 = \left( \n_{f_*X} \pi_i \right) f_* Y = \nbar_{f_* X} \pi_if_* Y - \pi_i \nbar_{f_* X} f_* Y = \\
		& = \nbar_X (f_* \BR_i Y + \BS_i Y) - \pi_i \left[ f_* \n_X Y + \al{X}{Y} \right] = \\
		& = \nbar_X f_* \BR_i Y + \nbar_X \BS_i Y - f_*\BR_i \n_X Y - \BS_i \n_X Y - \\
		& \quad -  f_*\BS_i^\t\al{X}{Y} - \BT_i \al{X}{Y} = \\
		& = f_* \n_X \BR_i Y + \al{X}{\BR_i Y} - f_*A_{\BS_i Y} X + \nperp_X \BS_i Y - \\
		& \quad - f_*\left[\BR_i \n_X Y + \BS_i^\t\al{X}{Y} \right] - \BS_i \n_X Y - \BT_i \al{X}{Y} = \\
		& = f_* \left[ \left(\n_X \BR_i \right) Y - A_{\BS_i Y} X - \BS_i^\t \al{X}{Y} \right] + \al{X}{\BR_i Y}  + \\
		& \quad + \left( \n_X \BS_i \right) Y - \BT_i \al{X}{Y}.
	\end{align*}
	Therefore
	\[\left(\n_X \BR_i \right) Y = A_{\BS_i Y} X + \BS_i^\t \al{X}{Y} \ \ \text{and} \ \ \left( \n_X \BS_i \right) Y = \BT_i \al{X}{Y} - \al{X}{\BR_i Y}.\]

	Analogously,
	\begin{align*}
		& 0 = \left( \n_{f_*X} \pi_i \right) \xi = \\
		&= \nbar_{f_* X} \pi_i\xi - \pi_i \nbar_{f_* X} \xi = \nbar_{f_* X} \left( f_* \BS_i^\t\xi + \BT_i \xi \right) - \pi_i \left( - f_* A_\xi X + \nperp_X \xi \right) = \\
		& = \nbar_X f_* \BS_i^\t\xi + \nbar_X \BT_i\xi + f_*\BR_i A_{\xi} X + \BS_i A_{\xi} X - f_* \BS_i^\t\nperp_X \xi - \BT_i \nperp_X \xi  = \\
		& = f_*\n_X \BS_i^\t\xi + \al{X}{\BS_i^\t \xi} - f_*A_{\BT_i\xi} X + \nperp_X \BT_i\xi + f_*\left( \BR_i A_\xi X - \BS_i^\t\nperp_X \xi \right) + \\
		& \quad + \BS_i A_{\xi} X - \BT_i \nperp_X \xi  = \\
		& = f_* \left[ \left( \n_X \BS_i^\t \right)\xi - A_{\BT_i \xi} X + \BR_i A_\xi X \right] + \al{X}{\BS_i^\t \xi} + \left(\n_X \BT_i\right) \xi + \BS_i A_\xi X.
	\end{align*}
	Therefore
	\[\left(\n_X \BS_i^\t\right)\xi = A_{\BT_i \xi} X - \BR_i A_\xi X \quad \text{and} \quad \left(\n_X \BT_i\right) \xi = - \BS_i A_\xi X - \al{X}{\BS_i^\t \xi}. \qedhere\]
\end{proof}

\begin{obs}
	Making the inner product of both sides of the first equation in \eqref{eq: derivada S} by $\xi$, and comparing to the inner product of both sides of the second equation in \eqref{eq: derivada S} by $Y$, we can conclude that both equations in \eqref{eq: derivada S} are equivalent. We decided to write down both in order to use the most convenient form wen we need it.
\end{obs}

\subsubsection{Gaus equation}

	\begin{align*}
		& f_* \mathcal{R}(X,Y)Z = \left( \bar{\mathcal R}(f_*X,f_*Y)f_*Z \right)^T + f_*A_{\al{Y}{Z}}X - f_*A_{\al{X}{Z}}Y \\
		& \stackrel{\text{Lemma \ref{lem: curv1}}}{=} \left( \sum_{i=1}^\ell k_i \left[ \interno{\pi_if_*Y}{f_*Z}\pi_if_*X - \interno{\pi_if_*X}{f_*Z}\pi_if_*Y \right] \right)^T + \\
		& \quad + f_*A_{\al{Y}{Z}}X - f_*A_{\al{X}{Z}}Y = \\
		& = \sum_{i=1}^\ell k_i \left[ \interno{\BR_i Y}{Z}f_*\BR_i X - \interno{\BR_i X}{Z}f_*\BR_i Y \right] + f_*A_{\al{Y}{Z}}X - f_*A_{\al{X}{Z}}Y.
	\end{align*}
	Therefore
	\begin{equation}\label{eq: Gauss}
		\mathcal{R}(X,Y)Z = \sum_{i=1}^\ell k_i \left( \BR_i X \wedge \BR_i Y\right)Z + A_{\al{Y}{Z}}X - A_{\al{X}{Z}}Y,
	\end{equation}
	where $(A \wedge B) C := \interno{B}{C} A - \interno{A}{C} B$.

\subsubsection{Codazzi equation}
	\begin{align*}
		& \left(\nabla^\perp_X \alpha \right)(Y,Z) - \left(\nabla^\perp_Y \alpha \right)(X,Z) = \left(\bar{\mathcal R}(f_*X,f_*Y)f_*Z \right)^\perp = \\
		& \stackrel{\text{Lemma \ref{lem: curv1}}}{=} \left( \sum_{i=1}^\ell k_i \left[ \interno{\pi_if_*Y}{f_*Z}\pi_if_*X - \interno{\pi_if_*X}{f_*Z}\pi_if_*Y \right] \right)^\perp = \\
		& = \sum_{i=1}^\ell k_i \left[ \interno{\BR_i Y}{Z}\BS_i X - \interno{\BR_i X}{Z}\BS_i Y \right].
	\end{align*}

	\begin{align*}
		& f_* (\n_Y A)(X,\xi) - f_*(\n_X A)(Y, \xi) = \left(\bar{\mathcal R}(f_*X,f_*Y)\xi \right)^\mathrm{T} = \\
		& \stackrel{\text{Lemma \ref{lem: curv1}}}{=} \left( \sum_{i=1}^\ell k_i \left[ \interno{\pi_if_*Y}{\xi}\pi_if_*X - \interno{\pi_if_*X}{\xi}\pi_if_*Y \right] \right)^\mathrm{T} = \\
		& = \sum_{i=1}^\ell k_i \left[ \interno{\BS_i Y}{\xi}f_*\BR_i X - \interno{\BS_i X}{\xi}f_*\BR_i Y \right].
	\end{align*}

	Therefore,
	\begin{equation}\label{eq: Codazzi}
		\left. \begin{aligned}
			\left(\nabla^\perp_X \alpha \right)(Y,Z) - \left(\nabla^\perp_Y \alpha \right)(X,Z) = \sum_{i=1}^\ell k_i \left[ \interno{\BR_i Y}{Z}\BS_i X - \interno{\BR_i X}{Z}\BS_i Y \right] \\
			(\n_Y A)(X,\xi) - (\n_X A)(Y, \xi) = \sum_{i=1}^\ell k_i \left[ \interno{\BS_i Y}{\xi}\BR_i X - \interno{\BS_i X}{\xi}\BR_i Y \right]
		\end{aligned} \right\}
	\end{equation}

\subsubsection{Ricci equation}
	\begin{align*}
		& \mathcal{R}^\perp(X,Y)\xi = \left(\bar{\mathcal R}(f_*X,f_*Y)\xi \right)^\perp + \al{X}{A_\xi Y} - \al{A_\xi X}{Y} = \\
		& \stackrel{\text{Lemma \ref{lem: curv1}}}{=} \left( \sum_{i=1}^\ell k_i \left[ \interno{\pi_if_*Y}{\xi}\pi_if_*X - \interno{\pi_if_*X}{\xi}\pi_if_*Y \right] \right)^\perp + \\
		& \quad + \al{X}{A_\xi Y} - \al{A_\xi X}{Y} = \\
		& = \sum_{i=1}^\ell k_i \left[ \interno{\BS_i Y}{\xi}\BS_i X - \interno{\BS_i X}{\xi}\BS_i Y \right] + \al{X}{A_\xi Y} - \al{A_\xi X}{Y}.
	\end{align*}

	Therefore,
	\begin{equation}\label{eq: Ricci}
		\left. \begin{aligned}
			\mathcal{R}^\perp(X,Y)\xi = \al{X}{A_\xi Y} - \al{A_\xi X}{Y} + \sum_{i=1}^\ell k_i \left( \BS_i X \wedge \BS_i Y \right) \xi \\
			\interno{\mathcal{R}^\perp(X,Y)\xi}{\z} = \interno{\left[A_\xi, A_\z \right]X}{Y} + \interno{\sum_{i=1}^\ell k_i \left( \BS_i X \wedge \BS_i Y \right) \xi}{\z}
		\end{aligned} \right\}
	\end{equation}
		\subsection{The immersion $F = \imath \circ f$}

Let $F := \imath \circ f \colon M \to \RN$ and, for each $i \in \{1, \cdots, \ell\}$, let $\nu_i = -k_i(\pi_i\circ F)$. Let also $\nbar$, $\ntil$ and $\nbarperp$ be the connexions on $\hO$, $\RN$ and $T_F^\perp M$, respectively.

\begin{lem}
	$T_F^\perp M = T_f^\perp M \op \spa\{\nu_1, \cdots, \nu_\ell\}$. Besides,
	\begin{align}
		& \nbarperp_X \nu_i = -k_i\imath_* \BS_i X, && A^F_{\nu_i} = k_i \BR_i, \label{eq: nui} \\ 
		& A_{\imath_*\xi}^F = A_\xi^f, && \nbarperp_X \imath_* \xi = \imath_*\nperp_X \xi + \sum_{i=1}^\ell \interno{\BS_i X}{\xi}\nu_i, \label{nbarperpxi}
	\end{align}
	for any $X \in \Gamma(TM)$ and every $\xi \in \Gamma\left( T^\perp M\right)$.
	 
\end{lem}

\begin{proof}
	By Lemma \ref{lem: pi_j}, the vector field $\nu_i := -k_i(\pi_i \circ F)$ is normal to $F$. Hence
	\begin{align*}
		& T_{F(p)}^\perp M = T_{f(p)}^\perp M \op T_{\imath(F(p))}^\perp \hO \stackrel{\eqref{Tiperp}}{=} \\
		& = T_{f(p)}^\perp M \op \spa\left\{-k_1(\pi_1\circ F)(p), \cdots, -k_\ell(\pi_\ell\circ F)(p)\right\} = \\
		& = T_{f(p)}^\perp M \op \spa\left\{\nu_1(p), \cdots, \nu_\ell(p)\right\}
	\end{align*}

	Let $X \in \Gamma(TM)$, thus
	\[\ntil_X\nu_i = -k_i \pi_i F_* X = -k_i\imath_*\pi_i f_* X = -k_i\imath_* \left[ f_*\BR_i X + \BS_i X\right].\]
	Therefore, $\nbarperp_X \nu_i = -k_i\imath_* \BS_i X$ and $A^F_{\nu_i} = k_i \BR_i$.	

	Last, if $\xi \in \Gamma\left(T^\perp M \right)$, then
	\begin{align*}
		&\ntil_X \imath_*\xi (p) = \imath_* \nbar_X \xi(p) + \ai{f_*X}{\xi}(F(p)) = \\
		& \stackrel{\eqref{eq: ai}}{=} \imath_*\left(-f_*A^f_\xi X + \nperp_X \xi(p) \right) - \sum_{i=1}^\ell \interno{\pi_if_* X}{\xi}k_i \pi_i(F(p)) = \\
%		& = -F_*A^f_\xi X + \imath_* \nperp_X \xi(p) - \sum_{i=1}^\ell \interno{f_* X}{\BS_i^\t \xi + \BT_i \xi}k_i(\pi\circ F)(p) = \\
		& = -F_*A^f_\xi X + \imath_* \nperp_X \xi(p) + \sum_{i=1}^\ell \interno{\BS_i X}{\xi}\nu_i(p).
	\end{align*}
		Therefore $A_{\imath_*\xi}^F = A_\xi^f$ and $\nbarperp_X \imath_* \xi = \imath_*\nperp_X \xi + \sum\limits_{i=1}^\ell \interno{\BS_i X}{\xi}\nu_i$.
\end{proof}

\begin{lem}\label{lem: aF}
	$\aF{X}{Y} = \imath_* \af{X}{Y} + \sum\limits_{i=1}^\ell \interno{\BR_i X}{Y} \nu_i$, for all $X,Y \in TM$.
\end{lem}

\begin{proof}
	If $X, Y \in TM$, then
	\begin{align*}
		& \aF{X}{Y} = \imath_*\af{X}{Y} + \ai{f_*X}{f_*Y} = \\
		& \stackrel{\text{Lemma \ref{lem: ai}}}{=} \imath_* \af{X}{Y} - \sum_{i=1}^\ell \interno{\pi_i f_* X}{f_* Y} k_i(\pi_i \circ F) = \\
		& = \imath_* \af{X}{Y} + \sum_{i=1}^\ell \interno{\BR_i X}{Y} \nu_i. \qedhere
	\end{align*}
\end{proof}

%	\begin{lem}\label{lem: A e Phi}
%		Seja \index{$\vartheta$}$\vartheta := \nu_1 - \nu_2$, logo $A_\vartheta^F = \Phi$ e $\nbarperp_X \vartheta = (k_1+k_2)\imath_* \BS_i X$, para todo $X \in TM$, em que $\Phi = k_1 \id - (k_1+k_2)\BR_i$.
%	\end{lem}

%	\begin{proof}
%		Segue diretamente das fórmulas \eqref{eq: nu1} e \eqref{eq: nu2}.
%	\end{prof}
	\section{Examples of isometric immersions in $\XO$}
		\subsection{Products of isometric immersions}\label{sec: R S T nulos}

Let $i \in \{1, \cdots, \ell\}$ and $z \in \prod\limits_{\substack{j=1\\ j \ne i}}^\ell \o{j}$ be a fixed point, $\bar f \colon M \to \o{i}$ be an isometric immersion and $\imath_i^z \colon \o{i} \to \hO$ be the totally geodesic embedding given by $\imath_i^z(x) := (z,x)$. Simple examples of isometric immersions in $\hO$ can be made with compositions like $\imath_i^z \circ \bar f$:
\[\begin{matrix}
	M & \stackrel{\bar f}{\longrightarrow} & \o{i} & \stackrel{\imath_i^z}{\longrightarrow} & \left( \prod\limits_{\substack{j=1\\ j \ne i}}^\ell \o{j}\right) \x \o{i} \\
	x & \longmapsto & \bar f(x) & \longmapsto & \left(z, \bar f(x) \right)
\end{matrix}\]

\begin{lem}\label{lem: R=0}
	Let $f \colon M^m \to \hO$ be an isometric immersion. Then
	\begin{enum}
		\item $f(M) \subset \{z\} \x \o{i}$ for some $z \in \prod\limits_{\substack{j=1\\ j \ne i}}^\ell \o{j}$ if, and only if, $\BR_i = \id$.
		\item $f(M) \subset \left( \prod\limits_{\substack{j=1\\ j \ne i}}^\ell \o{j} \right)\x \{z\}$, for some $z\in \o{i}$ if, and only if, $\BR_i = 0$.
	\end{enum}
\end{lem}

\begin{proof}

	\noindent We know that $f(M) \subset \{z\}\x\o{i}$ if, and only if, $\pi_if_* X = f_* X$, for any $X \in TM$. But $\pi_i f_* = \BL_i$ and $\BR_i = \BL_i^\t \BL_i$, thus
	\begin{multline*}
		\pi_i f_* X = f_* X, \ \forall X \in TM \sss \\
		\sss \interno{\BR_i X}{Y} = \interno{\BL_i X}{\BL_i Y} = \interno{X}{Y}, \ \forall X,Y \in TM \sss \BR_i = \id.
\end{multline*}
	
	On the other side, $f(M) \subset \left( \prod\limits_{\substack{j=1\\ j \ne i}}^\ell \o{j}\right) \x \{z\}$ if, and only if, $\pi_i f_* X =0$, for all $X \in TM$. But
	\begin{multline*}
		\pi_i f_* X = 0, \ \forall X \in TM \sss \\
		\sss \interno{\BR_i X}{Y} = \interno{\BL_i X}{\BL_i Y} = 0, \ \forall X, Y \in TM \sss \BR_i = 0. \qedhere
	\end{multline*}
\end{proof}

Let $I = \{i_1, \cdots, i_l\} \subsetneq \{1, \cdots, \ell\}$, with $i_1 < i_2 < \cdots < i_l$. Now, lets consider the totally geodesic embedding
\[\jmath \colon \prod_{\substack{i=1\\i \notin I}}^\ell \o{i} \to \left( \prod\limits_{\substack{i=1\\i \notin I}}^\ell \o{i} \right) \x \left( \prod_{i \in I} \o{i} \right) = \hO,\]
given by $\jmath(x) = (x,z)$, where each $z \in \prod\limits_{j=1}^l \o{i_j}$ is a fixed point.

As a consequence of Lemma \ref{lem: R=0}, we have the following Corollary.

\begin{cor}
	Let $I = \{i_1, \cdots, i_l\} \subsetneq \{1, \cdots, \ell\}$, with $i_1 < i_2 < \cdots < i_l$, and let $f \colon M \to \hO$ be an isometric immersion. Then the following sentences are equivalent:
	\begin{enum}
		\item $f = \jmath \circ \bar f$, where $\bar f \colon M \to \prod\limits_{\substack{i=1\\i \notin I}}^\ell \o{i}$
		is an isometric immersion and $\jmath$ is the totally geodesic embedding given above.
		
		\item $\BR_{i_1}= \cdots = \BR_{i_l} = 0$.
	\end{enum}
\end{cor}

We can build other examples of isometric immersions $f \colon M_1\x M_2 \to \hO$, by $f(x,y) :=  (f_1(x), f_2(y))$, where
\[\begin{matrix}
	f_1 : & M_1 & \to & \prod\limits_{\substack{j=1\\ j \ne i}}^\ell \o{j}\\
	& x & \mapsto & f_1(x)
\end{matrix}\quad \text{and} \quad
\begin{matrix}
	f_2 : & M_2 & \to & \o{i} \\
	& y & \mapsto & f_2(y)
\end{matrix}\]
are isometric immersions. In order to study these examples, we need some new results.

\begin{lem}\label{lem: ker S}
	$\ker \BS_i = \ker \BR_i \op \ker (\id - \BR_i)$ and $\BR_i(\ker \BS_i) = \ker(\id - \BR_i)$.
\end{lem}

\begin{proof}
	Since $\ker \BS_i = \ker \BS_i^\t\BS_i$ and $\BS_i^\t\BS_i \stackrel{\eqref{eq: RST}}{=} \BR_i(\id - \BR_i)$, then
	\[\ker \BS_i = \set{ X \in T_xM}{\BR_i X = \BR_i^2 X}.\]
	Hence $\ker \BR_i \subset \ker \BS_i$ and $\ker(\id - \BR_i) \subset \ker \BS_i$.

	\begin{afi}{$\BR_i(\ker \BS_i) \subset \ker \BS_i$.}
		If $X \in \ker \BS_i$, then $\BR_i X = \BR_i^2 X$. So $\BR_i(\BR_i X) = \BR_i \left( \BR_i^2 X\right) = \BR_i^2 (\BR_i X)$, that is, $\BR_i X \in \ker \BS_i$.
	\end{afi}

	So $\BR_i|_{\ker \BS_i} = \BR_i^2|_{\ker \BS_i}$ and we know that $\BR_i$ is self-adjoint, then $\BR_i|_{\ker \BS_i}$ is an orthogonal projection. Therefore $\ker\BS_i = \ker \BR_i|_{\ker \BS_i} \op \BR_i(\ker \BS_i)$. Now we have to show that $\BR_i(\ker \BS_i) = \ker (\id - \BR_i)$.

	Indeed, if $Y \in \BR_i(\ker \BS_i)$, then $Y = \BR_i X$, for some $X \in \ker \BS_i$. Hence $Y = \BR_i X = \BR_i^2 X = \BR_i Y$, that is, $Y \in \ker (\id - \BR_i)$.
	
	On the other side, if $Y \in \ker(\id - \BR_i)$, then $Y = \BR_i Y$. Thus $\BR_i Y = \BR_i^2 Y$. Therefore $Y \in \ker \BS_i$ and $Y = \BR_i Y$, so $Y \in \BR_i( \ker \BS_i)$.
\end{proof}

\begin{lem}\label{lem: S=0}
	Let $f\colon M^m \to \hO$ be an isometric immersion with $M$ connected. If $\BS_i = 0$, then $\ker \BR_i$ and $\ker (\id - \BR_i)$ have constant dimension on $M$.
\end{lem}

\begin{proof}
	If $\BS_i = 0$ then, by Lemma \ref{lem: ker S}, $TM = \ker \BS_i = \ker \BR_i \op \ker (\id - \BR_i) = \ker \BR_i \op \BR_i(TM)$.
	
	Now, for each $j \in \{0, \cdots, m\}$, let $A_j := \set{x \in M}{\dim \left(\ker \left.\BR_i\right|_{T_x M} \right) = j}$.

	\begin{afi}{Each $A_j$ is an open set.}
		If $A_j = \varnothing$, then $A_j$ is open. So lets suppose that $\dim \left( \ker \left.\BR_i\right|_{T_p M} \right) = j$, for some $p \in M$.

		If $j = 0$, then there are vector fields $X_1, \cdots, X_m$, defined in a neighborhood $U$ of $p$, such that $\BR_i X_1, \cdots, \BR_i X_m$ are LI in $U$. Hence $\BR_i\left(T_x M \right) = T_x M$, for every $x \in U$, that is, $\dim \left(\ker \left.\BR_i\right|_{T_x M} \right) = 0$, for every $x \in U$.

		If $j = m$, then $\dim \left[ \ker\left.(\id - \BR_i)\right|_{T_p M} \right] = 0$. Hence there are $X_1, \cdots, X_m$, defined in a neighborhood $U$ of $p$, such that $(\id - \BR_i)X_1, \cdots, (\id - \BR_i)X_m$ are LI in $U$. Thus $\dim \left[ \ker \left.(\id - \BR_i)\right|_{T_x M} \right] = 0$, for every $x \in U$, that is, $m = \dim \left( \ker\left.\BR_i\right|_{T_x M} \right)$, for all $x \in U$.

		Lets suppose now that $0 < j < m$. In this case there are vector fields $X_1, \cdots, X_m$, defined in a neighborhood $U$ of $p$, such that $(\id - \BR_i)X_1$, $\cdots$, $(\id - \BR_i)X_j$ are LI in $U$ and $\BR_i X_{j+1}$, $\cdots$, $\BR_i X_m$ are also LI in $U$. Thus, $\dim \left(\ker \BR_i|_{T_x M} \right) \geq j$ and $\dim \left[ \ker (\id - \BR_i)|_{T_x M} \right] \geq m-j$, for any $x \in U$.
		
		But $m = \dim (\ker\BR_i) + \dim [\ker (\id - \BR_i)]$, then $\dim \left( \left.\ker \BR_i\right|_{T_x M} \right) = j$, for all $x \in U$. Therefore $A_j$ is open, for any $j \in \{0, \cdots, m\}$.
	\end{afi}

	We know that $M = \bigcup\limits_{j=0}^m A_j$ and $A_j \cap A_o = \varnothing$, if $j \ne o$. Since $M$ is connected, $M = A_j$, for some $j \in \{0, \cdots, m\}$, that is, $\dim (\ker \BR_i)$ is constant in $M$.
\end{proof}

\begin{prop}\label{prop: S=0}
	Let $f \colon M^m \to \hO$ be an isometric immersion with, $M$ connected. Thus the following claims are equivalent:
	\begin{enum}
		\item $M$ is locally (isometric to) a product manifold $M_1^{m_1}\x M_2^{m_2}$, with $0 < m_1 < m$, and $f$ is locally a product immersion
		\[\begin{matrix}
			f : & M_1 \x M_2 & \longrightarrow & \left(\prod\limits_{\substack{j=1 \\ j \ne i}}^\ell \o{j} \right) \x \o{i} \\
			& (x,y) & \longmapsto & \big(f_1(x), f_2(y)\big).
		\end{matrix}\]
		
		\item $\BS_i = 0$ and $\dim (\ker \BR_i) = m_1$, with $0 < m_1 < m$.
	\end{enum}

	Besides that, if $M$ is complete and simply connected and the second claim is true, then $M$ is globally isometric to a product manifold $M_1^{m_1} \x M_2^{m_2}$ and $f$ is globally a product immersion like in \textsl{(I)}.
\end{prop}

\begin{proof}
	\textbf{\textsl{(I)} $\Rightarrow$ \textsl{(II)}:} Lets suppose that $M^m$ like in \textsl{(I)}. Hence, $T_{(x,y)} M = {\imath_1^y}_* T_x M_1 \op {\imath_2^x}_* T_y M_2$, where $\imath_1^y \colon M_1 \to M$ and $\imath_2^x \colon M_2 \to M$ are given by $\imath_1^y(z) = (z,y)$ and $\imath_2^x (z) = (x,z)$.

	If $X \in {\imath_1^y}_* T_x M_1$, then $X = {\imath_1^y}_* \tilde X = \left(\tilde X, 0 \right) \in T_{(x,y)} \left(\prod\limits_{\substack{j=1 \\ j \ne i}}^\ell \o{j} \right) \x \o{i}$ and
	$\pi_i f_* X = \pi_if_* \left( \tilde X, 0 \right) = \pi_i\left({f_1}_* \tilde X, {f_2}_* 0 \right) = 0$, thus $\BR_i X = 0$.
	
	Analogously, if $Y \in {\imath_2^x}_* T_y M_2$ then $(\id - \BR_i)Y = 0$. Thus
	\[T_{(x,y)} M = {\imath_1^y}_* T_x M_1 \op {\imath_2^x}_* T_y M_2 = \ker \BR_i|_{T_{(x,y)} M} \op \ker (\id - \BR_i)|_{T_{(x,y)} M}.\]
	Therefore $\dim (\ker \BR_i) = m_1$ and from Lemma \ref{lem: S=0} we conclude that $\BS_i = 0$. \vspace{1ex}

	\noindent \textbf{\textsl{(II)} $\Rightarrow$ \textsl{(I)}:} Lets suppose now that $\BS_i = 0$ and $\dim (\ker \BR_i) = m_1$, with $0 < m_1 < m$. Hence, by Lemma \ref{lem: ker S}, we know that $TM = \ker \BR_i \op \ker(\id - \BR_i)$, and it follows from Lemma \ref{lem: S=0} that $\ker \BR_i$ and $\ker (\id-\BR_i)$ are distributions on $M$.

	\begin{afi}{$\ker \BR_i$ and $\ker (\id - \BR_i)$ are parallel distributions.}	
		Let $Y \in \G(\ker \BR_i)$ and $X \in \G(TM)$. Thus $\BR_i \n_X Y \stackrel{\eqref{eq: derivada R}}{=} \n_X \BR_i Y = 0$. Therefore $\ker \BR_i$ and $(\ker \BR_i)^\perp = \ker (\id - \BR_i)$ are parallel distributions.
	\end{afi}

	Now, for each $x \in M$, let $L_1^{m_1}(x)$ and $L_2^{m_2}(x)$ be integral submanifolds of $\ker \BR_i$ and $(\ker \BR_i)^\perp$, respectively, at the point $x$. The De Rham's Theorem (see \cite{dR} and \cite{RS}) assure us that for each $x \in M$, there is a neighborhood $U$ of $x$ and there are open sets $M_1^{m_1}\subset L_1^{m_1}(x)$ and $M_2^{m_2} \subset L_2^{m_2}(x)$ and an isometry $\psi \colon M_1\x M_2 \to U$ such that $x \in U \cap M_1 \cap M_2$ and, for each $(x_1,x_2) \in M_1\x M_2$, $\psi \left( M_1 \x\{x_2\} \right)$ is a leaf of $\ker \BR_i$ and $\psi \left(\{x_1\}\x M_2 \right)$ is a leaf of $(\ker \BR_i)^\perp$.

	Thus we can identify  $U$ with $M_1^{m_1} \x M_2^{m_2}$, $\ker \BR_i$ with $T M_1$ and $(\ker \BR_i)^\perp$ with $T M_2$ and we can consider the applications $f \colon M_1^{m_1} \x M_2^{m_2} \to \hO$ and $F = \imath \circ f \colon M_1^{m_1}\x M_2^{m_2} \to \RN$, where $\imath \colon \hO \hookrightarrow \RN$ is the canonical inclusion.

	\begin{afi}{If $X \in \ker \BR_i$ and $Y \in (\ker \BR_i)^\perp$, then $\aF{X}{Y} = 0$.}
		If $X \in \ker \BR_i$ and $Y \in \BR_i(TM)$, then $Y = \BR_i Z$, for some $Z \in TM$. Hence $\af{X}{Y} = \af{X}{\BR_i Z} \stackrel{\eqref{eq: derivada S}}{=} \BT_i \af{X}{Z} \stackrel{\eqref{eq: derivada S}}{=} \af{\BR_i X}{Z} = 0$. Thus
		\begin{multline*}
			\aF{X}{Y} = \imath_* \af{X}{Y} + \sum_{j=1}^\ell \interno{\BR_j X}{Y}\nu_j = \sum_{\substack{j=1\\ j\ne i}}^\ell \interno{\BR_j X}{\BR_i Z}\nu_j = \\
			= \sum_{\substack{j=1\\ j \ne i}}^\ell\interno{X}{\BR_j \BR_i Z}\nu_j \stackrel{\eqref{eq: RST2}}{=} 0.
		\end{multline*}

		Therefore the claim holds.
	\end{afi}

	Now we can apply Moore's Lemma (see \cite{Detc}) and conclude that there is an orthogonal decomposition $\RN = V_0 \op V_1 \op V_2$ and a vector $v_0 \in V_0$ such that $F \colon M_1 \x M_2 \to \RN$ is given by $F(x,y) = (v_0, F_1(x), F_2(y))$. Besides,
	\begin{align*}
		& V_1 = \spa\set{F_*(p) X}{p \in M_1 \x M_2 \ \text{e} \ X \in \ker \BR_i|_{T_p M}}, \\
		& V_2 = \spa\set{F_*(p) Y}{p \in M_1 \x M_2 \ \text{e} \ Y \in \BR_i\left(T_p M \right)}
	\end{align*}
	and $F_i(M_i) \subset V_i$.

	On the other side, $\pi_if_*(\ker \BR_i) = \{0\}$ and $\left.\pi_i\right|_{f_*\ker (\id - \BR_i)} = \id$, hence
	\begin{align*}
		& \spa\set{F_*(p) X}{p \in M_1 \x M_2 \ \text{e} \ X \in \ker \BR_i|_{T_p M}} \subset \prod_{\substack{j=1\\j \ne i}}^\ell \R^{N_j},\\
		& \spa\set{F_*(p) Y}{p \in M_1 \x M_2 \ \text{e} \ Y \in \BR_i\left(T_p M \right)} \subset \R^{N_i},
	\end{align*}
	therefore $F_1(M_1)\subset \prod\limits_{\substack{j=1\\j \ne i}}^\ell \R^{N_j}$ and $F_2(M_2) \subset \R^{N_i}$.
	
	Lets define $\tilde f_1 \colon M_1 \to \prod\limits_{\substack{j=1\\j \ne i}}^\ell \R^{N_j}$ and $\tilde f_2 \colon M_2 \to \R^{N_i}$ by $\tilde f_1(x) := \Pi(v_0) + F_1(x)$ and $\tilde f_2(y) = \pi_i(v_0) + F_2(y)$, where $\Pi \colon \RN \to \prod\limits_{\substack{j=1\\j \ne i}}^\ell \R^{N_j}$ is the orthogonal projection. Hence
	\[F(x,y) = \big(\tilde f_1(x), \tilde f_2(y)\big) \in \left( \prod_{\substack{j=1\\ j \ne i}}^\ell \o{j}\right) \x \o{i}.\]
	Now, if $f_1 \colon M_1 \to \prod\limits_{\substack{j=1\\ j \ne i}}^\ell \o{j}$ and $f_2 \colon M_2 \to \o{i}$ are given by $f_1(x) = \tilde f_1(x)$ and $f_2(y) = \tilde f_2(y)$, then $f = f_1 \x f_2$.
	
	If $M$ is complete and simply connected, the De Rham's Lemma assure us that $M$ is (globally) isometric to $L_1 \x L_2$, where $L_1$ and $L_2$ are the leafs of $\ker \BR_i$ and $(\ker \BR_i)^\perp$ (respectively) at the same point. In this case, considering $f \colon L_1 \x L_2 \to \hO$, the calculations made above show us that $f$ is globally a product immersion.
\end{proof}

\begin{cor}
	Let $f \colon M^m \to \hO$ be an isometric immersion. The following claims are equivalent:
	\begin{enum}
		\item $M$ is locally (isometric to) a product manifold $M_1^{m_1}\x \cdots \x M_\ell^{m_\ell}$, with $0 < m_i < m$, and $f$ is locally a product immersion $f|_{M_1\x \cdots \x M_\ell} = f_1 \x \cdots \x f_\ell$, where each $f_i \colon M_i \to \o{i}$ is an isometric immersion.
		\item For each $i \in \{1, \cdots, \ell\}$, $\BS_i = 0$ and $\dim (\ker \BR_i) = m_i$, with $0 < m_i < m$.
	\end{enum}

	Besides, if $M$ is complete and simply connected and the second claim holds, then $M$ is globally an isometric product $M_1^{m_1} \x \cdots \x M_\ell^{m_\ell}$ and $f$ is globally a product immersion.
\end{cor}
		\subsection{Other products of isometric immersions}

Let $I = \{i_1, \cdots, i_{\ell_1}\} \subsetneq \{1, \cdots, \ell\}$, with $i_1 < i_2 < \cdots < i_{\ell_1}$. Other products of isometric immersions are the following kind of immersions
\[\begin{matrix}
	f_1 \x f_2 : & M_1 \x M_2 & \longrightarrow & \left(\prod\limits_{i\in I} \o{i} \right) \x \left(\prod\limits_{\substack{i=1 \\ i \notin I}}^\ell \o{i} \right) \\
	& (x,y) & \longmapsto & \big(f_1(x), f_2(y)\big),
\end{matrix}\]
where $f_1 \colon M_1 \to \prod\limits_{i\in I} \o{i}$ and $f_2 \colon M_2 \to \prod\limits_{\substack{i=1 \\ i \notin I}}^\ell \o{i}$ are isometric immersions.

Let $\Pi_1$ and $\Pi_2$ be the orthogonal projections given by
\begin{align*}
	\begin{matrix}
		\Pi_1 : & \RN & \longrightarrow & \left(\prod\limits_{i \in I} \R^{N_i} \right) \x \left(\prod\limits_{\substack{i=1 \\ i \notin I}}^\ell \R^{N_i}\right) \\
		& (x,y) & \longmapsto & (x,0)
	\end{matrix} \\
	\begin{matrix}
		\Pi_2 : & \RN & \longrightarrow & \left(\prod\limits_{i \in I} \R^{N_i} \right) \x \left(\prod\limits_{\substack{i=1 \\ i \notin I}}^\ell \R^{N_i}\right) \\
		& (x,y) & \longmapsto & (0,y);
	\end{matrix}
\end{align*}
Hence, $\Pi_1 = \sum\limits_{i \in I} \pi_i$ and $\Pi_2 = \sum\limits_{\substack{i=1 \\ i \notin I}}^\ell \pi_i$.

If $X \in T_x M$ and $\xi \in T_x^\perp M$, then
\[\begin{aligned}
	& \Pi_1 f_* X = \sum_{i \in I} \pi_i f_* X = \sum_{i \in I} \left[ f_* \BR_i X + \BS_i X \right]= f_* \sum_{i \in I} \BR_i X + \sum_{i \in I} \BS_i X.\\
	& \Pi_1 \xi = \sum_{i \in I} \pi_i \xi = \sum_{i \in I} \left[ f_* \BS^\t_i X + \BT_i X \right]= f_* \sum_{i \in I} \BS^\t_i X + \sum_{i \in I} \BT_i X.
\end{aligned}\]
Analogously $\Pi_2 f_* X = f_* \sum\limits_{\substack{i=1 \\ i \notin I }}^\ell \BR_i X + \sum\limits_{\substack{i=1 \\ i \notin I }}^\ell \BS_i X$, and $\Pi_2 \xi = f_* \sum\limits_{\substack{i=1 \\ i \notin I }}^\ell \BS^\t_i X + \sum\limits_{\substack{i=1 \\ i \notin I }}^\ell \BT_i X$. Consequently, the following equations hold
\begin{align*}
	\left(\Pi_1 f_* X \right)^T &= f_* \sum_{i \in I} \BR_i X, & \left( \Pi_1 f_* X\right)^\perp &= \sum_{i \in I} \BS_i X,\\
	\left(\Pi_2 f_* X \right)^T &= f_* \sum_{\substack{i=1 \\ i \notin I }}^\ell \BR_i X, & \left( \Pi_2 f_* X\right)^\perp &= \sum_{\substack{i=1 \\ i \notin I }}^\ell \BS_i X, \\
	\left(\Pi_1 \xi \right)^T &= \sum_{i \in I} \BS^\t_i X, & \left(\Pi_1 \xi \right)^\perp &= \sum_{i=1}^{\ell_1} \BT_i \xi, \\
	\left(\Pi_2 \xi \right)^T &= \sum_{\substack{i=1 \\ i \notin I }}^\ell \BS^\t_i X, & \left(\Pi_2 \xi \right)^\perp &= \sum_{\substack{i=1 \\ i \notin I }}^\ell \BT_i \xi.
\end{align*}

So, for each $i\in \{1, 2\}$, lets consider the tensors $\TL_i \colon TM \to T\hO$ and $\TK_i \colon T^\perp M \to T\hO$ given by $\TL_i X = \Pi_i f_*X$ and $\TK_i \xi = \Pi_i \xi$. Let also $\TR_i := \TL^\t_i \TL_i$, $\TS_i := \TK_i^\t \TL_i$ and $\TT_i := \TK_i^\t\TK_i$. So, from calculations analogous to those made for the tensors $\BR_i$, $\BS_i$ and $\BT_i$, it follows that:
\begin{align*}
	&\TL_i X = f_*\TR_i X + \TS_i X, && \TK_i \xi = f_*\TS_i^\t \xi + \TT_i \xi, \\
	& \TR_1 + \TR_2 = \id|_{TM}, && \TS_1 = -\TS_2, && \TT_1 + \TT_2 = \id|_{T^\perp M}.	
\end{align*}
Hence
\begin{align*}
	& \TR_1 = \sum_{i \in I} \BR_i, && \TS_1 = \sum_{i \in I} \BS_i, && \TT_1 = \sum_{i \in I} \BT_i,\\
	& \TR_2 = \sum_{\substack{i=1 \\ i \notin I }}^\ell \BR_i, && \TS_2 = \sum_{\substack{i=1 \\ i \notin I }}^\ell \BS_i, && \TT_2 = \sum_{\substack{i=1 \\ i \notin I }}^\ell \BT_i.
\end{align*}

Also from calculations analogous to those made for $\BR_i$, $\BS_i$ and $\BT_i$, we have the following equations:
\begin{align}
	& \TS_i^\t\TS_i = \TR_i(\id - \TR_i), && \TT_i\TS_i = \TS_i( \id - \TR_i), && \TS_i\TS_i^\t = \TT_i(\id - \TT_i), \label{equação: TR TS TT} \\
	 & \TS_i^\t\TS_j \stackrel{i \ne j}{=} -\TR_i\TR_j, && \TT_i \TS_j \stackrel{i \ne j}{=} -\TS_i\TR_j, && \TS_i \TS_j^\t \stackrel{i\ne j}{=} -\TT_i\TT_j.
\end{align}

\begin{lem}\label{lem: TR=0}
	Let $f \colon M^m \to \hO$ be an isometric immersion. Then
	\begin{enum}
		\item $f(M) \subset \{z\} \x \left(\prod\limits_{\substack{i=1 \\ i \notin I}}^\ell \o{i} \right)$ for some $z \in \prod\limits_{i \in I} \o{i}$ if, and only if, $\sum\limits_{i \in I} \BR_i = 0$.
		\item $f(M) \subset \left(\prod\limits_{i \in I} \o{i} \right) \x \{z\}$, for some $z \in \prod\limits_{\substack{i=1 \\ i \notin I}}^\ell \o{i}$ if, and only if, $\sum\limits_{i \in I} \BR_i = \id$.
	\end{enum}
\end{lem}

\begin{proof}
	The proof is analogous to the proof of Lemma \ref{lem: R=0}.
\end{proof}

\begin{lem}\label{lem: ker TS}
	For each $i \in \{1,2\}$,
	\[\ker \TS_i = \ker \TR_i \op \ker \left(\id - \TR_i\right) \quad \text{and} \quad \TR_i\left(\ker \TS_i\right) = \ker\left(\id - \TR_i\right).\]
\end{lem}

\begin{proof}
	The proof is analogous to the proof of Lemma \ref{lem: ker S}
\end{proof}

%\begin{cor}\label{corolário: ker soma S}
%	$\ker\sum\limits_{i \in I} \BS_i = \left(\ker\sum\limits_{i \in I} \BR_i \right) \op \left(\ker \sum\limits_{\substack{i=1 \\ i \notin I }}^\ell \BR_i \right) = \ker \sum\limits_{\substack{i=1 \\ i \notin I }}^\ell \BS_i$.
%\end{cor}

\begin{lem}\label{lem: TS=0}
	Let $f\colon M^m \to \hO$ be a isometric immersion. If $M$ is connected and $\sum\limits_{i \in I}\BS_i = 0$, then $\ker \sum\limits_{i \in I} \BR_i$ and $\ker \sum\limits_{\substack{i=1 \\ i \notin I}}^\ell \BR_i$ have constant dimension in $M$.
\end{lem}

\begin{proof}
	If $\sum\limits_{i \in I}\BS_i = 0$ , then $\TS_1 = \sum\limits_{i \in I} \BS_i = 0$, and, by Lemma \ref{lem: ker TS}, $TM = \ker \TR_1 \op \ker \left(\id - \TR_1\right)$.
	
	For each $j \in \{0, \cdots, m\}$, let $A_j := \set{x \in M}{\dim \left(\ker \left.\TR_1\right|_{T_x M} \right) = j}$.

	\begin{afi}{$A_j$ is open.}
		Analogous to the proof of Claim 1 of Lemma \ref{lem: S=0}.
	\end{afi}

	But $M = \bigcup\limits_{j=0}^m A_j$ and $A_j \cap A_o = \varnothing$, if $j \ne o$. Then, since $M$ is connected, $M = A_j$, for some $j \in \{0, \cdots, m\}$, that is, $\ker \TR_1$ have constant dimension in $M$. Therefore $\ker \sum\limits_{i \in I} \BR_i$ and $\ker \sum\limits_{\substack{i=1 \\ i \notin I}}^\ell \BR_i$ have constant dimension in $M$
\end{proof}

\begin{prop}
	Let $f \colon M^m \to \hO$ be an isometric immersion, with $M$ connected. Thus the following claims are equivalent
	\begin{enum}
		\item $M$ is locally (isometric to) a product manifold $M_1^{m_1}\x M_2^{m_2}$, with $0 < m_2 < m$, and $f$ is locally a product immersion $f|_{M_1\x M_2} = f_1 \x f_2$, given by $f(x,y) = (f_1(x), f_2(y))$, where
		\[f_1 \colon M_1 \to \prod_{i \in I} \o{i} \quad \text{and} \quad f_2 \colon M_2 \to \prod_{\substack{i=1 \\ i \notin I}}^\ell \o{i}\]
		are isometric immersions.
%		\[\begin{matrix}
%			\displaystyle f|_{M_1\x M_2} = f_1 \x f_2 : & M_1 \x M_2 & \longrightarrow & \left(\displaystyle\prod_{i \in I} \o{i} \right) \x \left(\displaystyle\prod_{\substack{i=1 \\ i \notin I}}^\ell \o{i} \right) \\
%			& (x,y) & \longmapsto & \big(f_1(x), f_2(y)\big),
%		\end{matrix}\]
		
		\item $\sum\limits_{i \in I} \BS_i = 0$, $\dim \left[\ker \left(\sum\limits_{i \in I} \BR_i \right)\right] = m_2$ and $0 < m_2 < m$.
	\end{enum}

	Besides, if \textsl{(II)} holds and $M$ is complete and simply connected, then $M$ is globally isometric to product manifold $M_1^{m_1} \x M_2^{m_2}$ and $f$ is globally an product immersion like in \textsl{(I)}.
\end{prop}

\begin{obs}\label{obsevação: soma Si}
	$\sum\limits_{i \in I} \BS_i = 0 \stackrel{\eqref{somas}}{\sss} \sum\limits_{\substack{i=1 \\ i \notin I}}^\ell \BS_i =0$.
\end{obs}

\begin{proof}
	\textbf{\textsl{(I)} $\Rightarrow$ \textsl{(II)}:} Lets suppose that $M$ is locally a product manifold $M_1^{m_1}\x M_2^{m_2}$ and that
	\[\begin{matrix}
		f|_{M_1\x M_2} = f_1 \x f_2 : & M_1 \x M_2 & \longrightarrow & \left(\prod\limits_{i \in I} \o{i} \right) \x \left(\prod\limits_{\substack{i=1 \\ i \notin I}}^\ell \o{i} \right) \\
		& (x,y) & \longmapsto & \big(f_1(x), f_2(y)\big).
	\end{matrix}\]
	
	Let $\imath_1^y \colon M_1 \to M$ and $\imath_2^x \colon M_2 \to M$ be given by $\imath_1^y(z) = (z,y)$ and $\imath_2^x (z) = (x,z)$. Hence, $T_{(x,y)} M = {\imath_1^y}_* T_x M_1 \op {\imath_2^x}_* T_y M_2$.
	
	If $X \in {\imath_1^y}_* T_x M_1$, then $X = {\imath_1^y}_* \tilde X = \left(\tilde X, 0 \right) \in T_{(x,y)} \left(\prod\limits_{i \in I} \o{j} \right) \x \left(\prod\limits_{\substack{i=1 \\ i \notin I}}^\ell \o{i} \right)$ and
	$\Pi_2 f_*X = \Pi_2 f_* \left( \tilde X, 0 \right) = \Pi_2\left({f_1}_* \tilde X, {f_2}_* 0 \right) = 0$, thus $\TR_2 X = 0$, $\TS_2 X = 0$ and $\TS_1X= 0$.
	
	Analogously, if $Y \in {\imath_2^x}_* T_y M_2$ then $\TR_1 Y = 0$, $\TS_1 Y = 0$ and $\TS_2 Y = 0$. Thus $\sum\limits_{i \in I} \BS_i = 0$, ${\imath_1^y}_* T_x M_1 \subset \ker \TR_2$ and ${\imath_2^x}_* T_y M_2 \subset \ker \TR_1$.
	
	So $TM = \ker \TS_1 = \ker \TR_1 \op \ker\left( \id - \TR_1 \right) = \ker \TR_1 \op \ker \TR_2$. Therefore ${\imath_1^y}_* T_x M_1 = \ker\TR_2$, ${\imath_2^x}_* T_x M_2 = \ker\TR_1$ and $\dim \ker\left( \sum\limits_{i \in I} \BR_i \right) = m_2$.
	\vspace{1ex}

	\noindent \textbf{\textsl{(II)} $\Rightarrow$ \textsl{(I)}:} By Lemma \ref{lem: ker TS}, $TM = \left( \ker\TR_2 \right) \op \left( \ker \TR_1\right)$ and, by Lemma \ref{lem: TS=0}, $\ker \TR_2  = \ker \sum\limits_{\substack{i=1 \\ i \notin I }}^{\ell_1} \BR_i$ and $\ker \TR_1= \ker \sum\limits_{i \in I} \BR_i$ are distributions on $M$.

	\begin{afi}{$\ker \TR_2$ and $\ker \TR_1$ are parallel distributions.}
		Let $Y \in \G(\ker \TR_1)$ and $X \in \G(TM)$. Hence,
		\begin{align*}
			& \n_X \TR_1 Y = \n_X \sum_{i \in I} \BR_i Y \stackrel{\eqref{eq: derivada R}}{=} \sum_{i \in I} \left( \BR_i \n_X Y + A_{\BS_i Y} X + \BS^\t_i \al{X}{Y} \right) = \\
			&= \left(\sum_{i \in I} \BR_i \right)\n_X Y + A_{\sum\limits_{i \in I} \BS_i Y} X + \left(\sum_{i \in I} \BS_i \right)^\t \al{X}{Y} =\\
			& = \TR_1 \n_X Y + A_{\TS_1 Y} X +\TS_1^\t \al{X}{Y} = \TR_1 \n_X Y.
		\end{align*}
		Therefore $\ker \TR_1$ is parallel and the same holds for $\left(\ker \TR_1\right)^\perp = \ker \TR_2$.
	\end{afi}

	For each $x \in M$, let $L_1^{m_1}(x)$ and $L_2^{m_2}(x)$ be integral submanifolds of $\ker\TR_2$ and $\ker\TR_1$, respectively, at $x$. Thus, by De Rham's Lemma (see \cite{dR} and \cite{RS}), for each $x \in M$, there is a neighborhood $U$ of $x$, and open sets $M_1^{m_1}$ and $M_2^{m_2}$ of $L_1^{m_1}(x)$ and $L_2^{m_2}(x)$ (respectively) and there is an isometry $\psi \colon M_1\x M_2 \to U$ such that $x \in U \cap M_1 \cap M_2$ and, for each $(x_1,x_2) \in M_1\x M_2$, $\psi \left( M_1 \x\{x_2\} \right)$ is a leaf of $\ker\TR_2$ and $\psi \left(\{x_1\}\x M_2 \right)$ is a leaf of $\ker\TR_1$.

	So we can identify $U$ with $M_1^{m_1} \x M_2^{m_2}$, $\ker\TR_2$ with $T M_1$ and $\ker\TR_1$ with $T M_2$ and we can consider the applications $f \colon M_1^{m_1} \x M_2^{m_2} \to \hO$ and $F = \imath \circ f \colon M_1^{m_1}\x M_2^{m_2} \to \RN$, where $\imath \colon \hO \hookrightarrow \RN$ is the canonical inclusion.

	\begin{afi}{If $X \in \ker\TR_2$ and $Y \in \ker\TR_1$, then $\aF{X}{Y} = 0$.}
		Let $X \in \ker\left.\TR_2\right|_{T_xM}$ and $Y \in \ker\left.\TR_1\right|_{T_xM}$. We know that $\TR_1$ is a orthogonal projection in $T_x M$, because it is self-adjoint and $\TR_1^2 = \TR_1$, thus $\ker\TR_2 = \ker\left( \id - \TR_1 \right) = \TR_1(TM)$. Hence, $X = \TR_1 Z$, for some $Z \in TM$, and
		\begin{multline*}
			\al{X}{Y} = \al{\TR_1 Z}{Y} = \al{\sum_{i \in I} \BR_i Z}{Y} \stackrel{\eqref{eq: derivada S}}{=} \sum_{i \in I}\BT_i \al{Z}{Y} = \\
			\stackrel{\eqref{eq: derivada S}}{=} \sum_{i \in I}\al{Z}{\BR_i Y} = \al{Z}{\TR_1 Y} = 0.
		\end{multline*}

		On the other side, $\TS_1 = 0 = \TS_2$, thus $\Pi_1 f_* X = f_* \TR_1 X + \TS_1 X = f_* X$ and $\Pi_2 f_* Y = f_* \TR_2 Y + \TS_2 Y = f_* Y$,
		that is, $f_* X \in \prod\limits_{i \in I} \R^{N_i}$ and $f_* Y \in \prod\limits_{\substack{i=1 \\ i \notin I}}^\ell \R^{N_i}$. It follows that $\pi_i f_* X = 0$, $\forall i \notin I$, and $\pi_i f_* Y = 0$, $\forall i \in I$. Thus
		\begin{multline*}
			\aF{X}{Y} = \imath_* \af{X}{Y} + \sum_{i=1}^\ell \interno{\pi_i f_* X}{f_*Y}\nu_i = \\
			= \sum_{i \in I} \interno{f_* X}{\pi_i f_*Y}\nu_i +  \sum_{\substack{i=1 \\ i \notin I} }^\ell \interno{\pi_i f_*X}{f_*Y}\nu_i =  0.
		\end{multline*}

		Therefore the claim holds.
	\end{afi}

	Now, by Moore's Lemma (see \cite{Detc}), there is a orthogonal decomposition $\RN = V_0 \op V_1 \op V_2$ and a vector $v_0 \in V_0$ such that $F \colon M_1 \x M_2 \to \RN$ is given by $F(x,y) = (v_0, F_1(x), F_2(y))$. Besides
	\begin{align*}
		& V_1 = \spa\set{F_*(p) X}{p \in M_1 \x M_2 \ \text{and} \ X \in \left.\ker\TR_2\right|_{T_p M}}, \\
		& V_2 = \spa\set{F_*(p) Y}{p \in M_1 \x M_2 \ \text{and} \ Y \in \left.\ker\TR_1\right|_{T_p M}},
	\end{align*}
	and $F_i(M_i) \subset V_i$.

	But, $\Pi_if_*(\ker \TR_i) = \{0\}$ and $\left.\Pi_i\right|_{f_*\ker (\id - \TR_i)} = \id$, thus
	\begin{align*}
		& \spa\set{F_*(p) X}{p \in M_1 \x M_2 \ \text{and} \ X \in \ker \left.\TR_2\right|_{T_p M}} \subset \prod_{i \in I} \R^{N_i} \quad \text{and}\\
		& \spa\set{F_*(p) Y}{p \in M_1 \x M_2 \ \text{and} \ Y \in \ker \left.\TR_1\right|_{T_p M}} \subset \prod_{\substack{i=1 \\ i \notin I}}^\ell \R^{N_i}.
	\end{align*}
	Therefore $F_1(M_1)\subset\prod\limits_{i \in I} \R^{N_i}$ and $F_2(M_2) \subset \prod\limits_{\substack{i=1 \\ i \notin I }}^\ell \R^{N_i}$. Hence we can define $f_1 \colon M_1 \to \prod\limits_{i \in I} \R^{N_i}$ and $f_2 \colon M_2 \to \prod\limits_{\substack{i=1 \\ i \notin I}}^\ell \R^{N_i}$ by $f_1(x) := \Pi_1(v_0) + F_1(x)$ and $f_2(y) = \Pi_2(v_0) + F_2(y)$.	Consequently
	\[f(x,y) = \big(f_1(x), f_2(y)\big) \in \left( \prod_{i \in I} \o{i}\right) \x \left( \prod_{\substack{i=1 \\ i \notin I }}^\ell\o{i}\right).\]

	If $M$ is complete and simply connected, De Rham's Lemma assure us that $M$ is (globally) isometric to $L_1 \x L_2$, where $L_1$ and $L_2$ are leafs of $\ker \TR_2$ and $(\ker \BT_2)^\perp$ (respectively) at the same point. In this case, considering $f \colon L_1 \x L_2 \to \hO$, the calculations made above show that $f$ is globally a product immersion.
\end{proof}

\begin{cor}
	Let $f \colon M \to \hO$ be an isometric immersion and let $\ell_1$, $\cdots$, $\ell_n$ be positive natural numbers such that $\sum\limits_{j=1}^n \ell_j = \ell$. Let also $I_1 := \{i_1^1, \cdots, i_{\ell_1}^1\}$, $\cdots$, $I_n := \{i_1^n, \cdots, i_{\ell_n}^n\}$ be disjoint sets such that $\bigcup\limits_{j=1}^n I_j = \{1, \cdots, \ell\}$. Then the following claims are equivalent:
	\begin{enum}
	\item $M$ is locally (isometric to) a product manifold $M_1^{m_1}\x \cdots \x M_n^{m_n}$, with $0 < m_j < m$, and $f$ is locally a product immersion $f|_{M_1\x \cdots \x M_n} = f_1 \x \cdots \x f_n$, given by $f(x_1,\cdots, x_n) = (f_1(x_1), \cdots, f_n(x_n))$, where $f_j \colon M_j \to \prod\limits_{i \in I_j} \o{i}$ is a isometric immersion, for each $j \in \{1, \cdots, n\}$.

	\item For each $j \in \{1, \cdots, n\}$, $\sum\limits_{i \in I_j} \BS_i = 0$ and $0< \dim \ker \left(\sum\limits_{i \in I_j} \BR_i \right) < m$.
\end{enum}

Besides, if \textsl{(II)} holds and $M$ is complete and simply connected, then $M$ is globally isometric to product manifold $M_1^{m_1} \x \cdots \x M_n^{m_n}$ and $f$ is globally an product immersion like in \textsl{(I)}.
\end{cor}
		\subsection{Weighted sums}

Let $k_1, \cdots, k_\ell \in \R$ and $a_1, \cdots, a_\ell\in \R^*$ be such that $\sum\limits_{i=1}^\ell a_i^2 = 1$. For each $i$, let $\tilde{k}_i := a_i^2 k_i$ and $f_i \colon M^m \to \O_{\tilde{k}_i}^{n_i}$ be a isometric immersion. We can define $f \colon M^m \to \hO$ by $f(x) := \left(a_1 f_1(x), \cdots, a_\ell f_\ell(x)\right)$. Thus $f$ is a isometric immersion, because
\[\interno{f_* X}{f_* Y} = \sum_{i=1}^\ell a_i^2 \interno{{f_i}_*X}{{f_i}_*Y} = \interno{X}{Y}.\]
This immersion $f$ is called \textbf{weighted sum} of the immersions $f_1$, $\cdots$, $f_\ell$ and the numbers $a_1$, $ \cdots$, $a_\ell$ are called the \textbf{weights} of $f$.

Now, lets consider $F = \imath \circ f$. We know that the fields $\nu_i = -k_i(\pi_i \circ F)$ are normal to $F$ and that $\RN = F_*T M \op \imath_*T^\perp M \op \spa\{\nu_1, \cdots, \nu_\ell\}$, where $T_x^\perp M \subset T_{f(x)}\hO$ is the normal space of $f$ at $x$. But since $\sum\limits_{i=1}^\ell a_i^2 = 1$ and $F_*T_xM = \set{\left(a_1 {f_1}_* X, \cdots, a_\ell{f_\ell}_*X \right)}{X \in T_x M}$, then $\left( a_1{f_1}_* X, \cdots, \frac{a_i^2-1}{a_i}{f_i}_* X, \cdots, a_\ell {f_\ell}_* X\right) \in \imath_* T_x^\perp M \perp F_* T_xM \op \spa\{\nu_1, \cdots, \nu_1\ell\}$.

Now we will study the tensors $\alpha$, $\BR_i$, $\BS_i$ and $\BT_i$ of $f$.

\begin{lem}\label{lem:RS weighted}
	If $f$ is a weighted sum and $a_1, \cdots, a_\ell$ are its weights, then $\BR_i = a_i^2 \id$ and $\BS_i X = -a_i^2\left(a_1{f_1}_*X, \cdots, \frac{a_i^2-1}{a_i}{f_i}_*X, \cdots, a_\ell{f_\ell}_*X\right)$.
\end{lem}

\begin{proof}
	\begin{align*}
		& \pi_if_*X = \left(0, \cdots, a_i{f_i}_*X, \cdots, 0\right) = \left(0, \cdots, \left(\sum_{j=1}^\ell a_j^2 \right)a_i{f_i}_*X, \cdots, 0 \right) =  \\
		& = a_i^2 \left( 0, \cdots, a_i {f_i}_*X, \cdots, 0 \right) + \left(0, \cdots, a_i \left(\sum_{\substack{j=1\\ j\ne i}}^\ell a_j^2 \right){f_i}_*X, \cdots 0\right) = \\
		& = a_i^2 \left( a_1{f_1}_* X, \cdots, a_i {f_i}_*X, \cdots, a_\ell {f_\ell}_* X \right) - \\
		& \quad - a_i^2 \left( a_1{f_1}_* X, \cdots, 0, \cdots, a_\ell {f_\ell}_* X \right) + a_i\left(0, \cdots, \left(1-a_i^2\right){f_i}_*X, \cdots, 0 \right) = \\
		& = a_i^2 f_* X - a_i^2\left(a_1{f_1}_*X, \cdots, \frac{a_i^2-1}{a_i}{f_i}_*X, \cdots, a_\ell{f_\ell}_*X\right). \qedhere
	\end{align*}
\end{proof}

\begin{obs}
	From Lemma \ref{lem:RS weighted}
	\[\BS_i\left(T_x M \right) = \set{a_i^2\left(a_1{f_1}_*X, \cdots, \frac{a_i^2-1}{a_i}{f_i}_*X, \cdots, a_\ell{f_\ell}_*X\right)}{X \in T_x M},\]
	and  $\dim \BS_i\left(T_x M \right) = m$. Besides, $\BT_i \BS_i = \BS_i (\id - \BR_i) = \left(1-a_i^2 \right) \BS_i = \sum\limits_{\substack{j=1\\ j\ne i}}^\ell a_j^2 \BS_i$.
\end{obs}

\begin{lem}\label{lem: segunda forma}
	If $f$ is a weighted sum then
	\[\af{X}{Y} = \big(a_1\alpha_{f_1}(X,Y), \cdots, a_\ell\alpha_{f_\ell}(X,Y) \big).\]
\end{lem}

\begin{proof}
	For each $i$, let $\tilde{\imath}_i \colon \O_{\tilde{k}_i}^{n_i} \to \R^{N_i}$ be the canonical inclusion and let $\mathcal{J}_i \colon \R^{N_i} \to \RN$ be totally geodesic immersion given by $\mathcal{J}_i(x) := (0, \cdots, x, \cdots, 0)$, with $x$ in the $i$th position. Thus, each $\tilde{\imath}_i$ is the restriction of the identity $I_i \colon \R^{N_i} \to \R^{N_i}$ and $T_x \O_{\tilde{k}_i}^{n_i} = T_{a_i x} \o{i}$. So, if $X \in T_x \O_{\tilde{k}_i}^{n_i}$, then $\left.\tilde{\imath}_i\right._* X = {I_i}_* X = {\imath_i}_* X$, where $\imath_i \colon \o{i} \to \R^{N_i}$ is the canonical inclusion.

	Besides, $\imath = (\imath_1 \x \cdots \x \imath_\ell)$ and
	\begin{multline*}
		F_* X = \imath_* \big(a_1{f_1}_* X, \cdots, a_\ell {f_\ell}_*X \big) = \big({\imath_1}_* a_1 {f_1}_* X, \cdots, {\imath_\ell}_* a_\ell {f_2}_*X \big) = \\
		= \big( a_1 \left. \tilde{\imath}_1 \right._* {f_1}_* X, \cdots, a_\ell \left.\tilde{\imath}_2\right._* {f_2}_* X \big).
	\end{multline*}

	Let $\ntil^i$ be the Levi-Civita connexion in $\R^{N_i}$. Hence
	\begin{align*}
		& \ntil_X F_* Y = \ntil_X \big( a_1 \left. \tilde{\imath}_1 \right._* {f_1}_* Y, \cdots, a_\ell \left.\tilde{\imath}_2\right._* {f_\ell}_* Y \big) = \\
		&= \ntil_X \left[ a_1 {\mathcal{J}_1}_* \left.\tilde{\imath}_1\right._* {f_1}_* Y + \cdots + a_\ell {\mathcal{J}_\ell}_* \left.\tilde{\imath}_2\right._* {f_\ell}_* Y \right] = \\
		& = a_1 {\mathcal{J}_1}_* \ntil^1_X \left( \tilde{\imath}_1 \circ f_1\right)_* Y + \cdots + a_\ell {\mathcal{J}_\ell}_* \ntil^\ell_X \left( \tilde{\imath}_2 \circ f_\ell\right)_* Y = \\
		& = a_1 {\mathcal{J}_1}_* \left.\tilde{\imath}_1\right._* \left[ {f_1}_* \n_X Y  + \alpha_{f_1}(X,Y)  \right] + a_1 {\mathcal{J}_1}_* \alpha_{\tilde{\imath}_1}\left({f_1}_*X, {f_1}_* Y\right) + \cdots + \\
			& \quad + a_\ell {\mathcal{J}_\ell}_* \left.\tilde{\imath}_\ell\right._* \left[ {f_\ell}_* \n_X Y  + \alpha_{f_\ell}(X,Y)  \right] + a_\ell {\mathcal{J}_\ell}_* \alpha_{\tilde{\imath}_\ell}\left({f_\ell}_*X, {f_\ell}_* Y\right)  = \\
		& = \big( a_1 \left.\tilde{\imath}_1\right._* \left[{f_1}_* \n_X Y + \alpha_{f_1}(X, Y) \right], 0, \cdots, 0 \big) - a_1 \interno{X}{Y} \left( \tilde{k}_1 f_1, 0, \cdots, 0 \right) + \\
			& \quad + \cdots + \\
			& \quad + \big( 0, \cdots, 0, a_\ell \left.\tilde{\imath}_\ell\right._* \left[{f_\ell}_* \n_X Y  + \alpha_{f_\ell}(X,Y) \right]\big) - a_\ell \interno{X}{Y} \left( 0, \cdots, 0, \tilde{k}_\ell f_\ell \right)= \\
		& = \big( {\imath_1}_* a_1 {f_1}_* \n_X Y , \cdots, {\imath_\ell}_* a_\ell {f_\ell}_* \n_X Y \big) + \\
			& \quad + \big( {\imath_1}_* a_1 \alpha_{f_1}(X,Y), \cdots, {\imath_\ell}_* a_\ell \alpha_{f_\ell}(X,Y) \big) - \interno{X}{Y} \left( a_1 \tilde{k}_1 f_1, \cdots, a_\ell \tilde{k}_\ell f_\ell  \right).
	\end{align*}

	But, $a_i \tilde{k_i} = a_i \cdot a_i^2 k_i$ and $F = (a_1f_1, \cdots, a_\ell f_\ell)$, thus
	\[\left( a_1 \tilde{k}_1 f_1, \cdots, a_\ell \tilde{k}_\ell f_\ell \right) = \left(a_1^2 \cdot a_1 k_1 f_1, \cdots, a_\ell^2 \cdot a_\ell k_\ell f_\ell\right) = \sum_{i=1}^\ell a_i^2 k_i(\pi_i \circ F).\]
	Therefore
	\begin{multline*}
		\ntil_X F_* Y = F_* \n_X Y + \imath_* \big( a_1 \alpha_{f_1}(X,Y), \cdots, a_\ell \alpha_{f_\ell}(X,Y) \big) + \interno{X}{Y} \sum_{i=1}^\ell a_i^2 \nu_i.
	\end{multline*}

	On the oder side, $\ntil_X F_* Y = F_* \n_X Y + \imath_* \af{X}{Y} + \sum\limits_{i=1}^\ell \interno{\BR_i X}{Y} \nu_i$. Since $\BR_i X = a_i^2 X$, it follows that $\af{X}{Y} = \big(a_1\alpha_{f_1}(X,Y), \cdots, a_\ell\alpha_{f_\ell}(X,Y) \big)$.
\end{proof}

\begin{cor}\label{cor: f_i umbílica}
	Let $f = (a_1f_1, \cdots, a_\ell f_\ell)$ be a weighted sum of isometric immersions. The following claims are equivalent:
	\begin{enum}
		\item $f$ is umbilical.
		\item $F = \imath \circ f$ is umbilical.
		\item Each $f_i$ é umbilical.
	\end{enum}
\end{cor}

\begin{proof}
	It follows directly from Lemma \ref{lem:RS weighted}.
\end{proof}

The follow proposition characterizes weighted sums.

\begin{prop}\label{prop: weighted}
	Let $f \colon M^m \to \hO$ be an isometric immersion, then $f$ is a weighted sum of isometric immersions with positive if, and only if, there are $a_1, \cdots, a_\ell \in (0,1)$ such that $\BR_i = a_i^2 \id$ and $\sum\limits_{i=1}^\ell a_i^2 = 1$.
\end{prop}

\begin{proof}
	If $f \colon M \to \hO$ is a wighted sum with weights $a_1, \cdots, a_\ell \in \R_+^*$, then $\sum\limits_{i=1}^\ell a_i^2 = 1$, $a_i \in (0,1)$ and we already know that $\BR_i = a_i^2 \id$, for each $i \in \{1, \cdots, \ell\}$.

	Lets suppose that there are $a_1, \cdots, a_\ell \in (0,1)$ such that $\BR_i = a_i^2 \id$ and $\sum\limits_{i=1}^\ell a_i^2 =1$. Thus $a_i^2 \interno{X}{Y} = \interno{\BR_i X}{Y} = \interno{\BL_i^\t \BL_i X}{Y} = \interno{\pi_i f_* X}{\pi_i f_* Y}$, as a consequence each $\pi_i\circ f$ is a similarity with ratio $a_i$ and $f_i := a_i^{-1}(\pi_i \circ f)$ is an isometric immersion in $\R^{N_i}$.

	Let $\tilde{k}_i := a_i^2 k_i$. If $k_i = 0$ then $\tilde{k}_i = 0$, $\O_{\tilde{k}_i}^{n_i} = \E^{N_i}$ and $f_i(x) \in \O_{\tilde{k}_i}^{n_i}$. If $k_i \ne 0$, then $\|f_i(x)\|^2 = \frac{1}{a_i^2 k_i}$, thus $f_i(x) \in \O_{\tilde{k}_i}^{n_i} \subset \R^{N_i}$.

	But $f = (a_1 f_1, \cdots, a_\ell f_\ell)$, therefore $f$ is a weighted sum.
\end{proof}
		\subsection{A particular weighted sum}\label{seção: exemplo}

Let $k_1, \cdots, k_\ell$ be real numbers such that $k_ik_j > 0$, for every $i, j \in \{1, \cdots, \ell\}$. Lets denote $\e := \frac{k_1}{|k_1|}$, $M_i := \prod\limits_{ \substack{j=1\\ j\ne i} }^\ell k_j$, $\l := \sum\limits_{i=1}^\ell M_i$, $a_i := \left(\frac{M_i}{\l}\right)^{\frac{1}{2}}$ and $\tilde{k}_i := a_i^2 k_i$.
Thus, $\e = \frac{k_i}{|k_i|}$, $\sum\limits_{i=1}^\ell a_i^2 = 1$ and $\tilde{k}_i = a_i^2 k_i = \frac{\prod\limits_{\substack{j=1\\ j\ne i}}^\ell k_j}{\l}k_i = \frac{\prod\limits_{j=1}^\ell k_j}{\l} = \tilde k_1$.

Lets denote $k := \tilde{k}_1$ and let $T_1, \cdots, T_\ell \in \mathrm{O}_{\tau(k)}(n+1)$ be such that $T_i \left(\o{i}\right) = \o{i}$. We know that the restrictions $\left.T_i\right|_{\o{}} \colon \o{} \to \o{}$ are isometries, hence the function 
\[\begin{matrix}
	g : & \o{} & \longrightarrow & \XO \\
	& x & \longmapsto & \big(a_1 T_1(x), \cdots, a_\ell T_\ell(x) \big),
\end{matrix}\]
is a weighted sum of $T_1$, $\cdots$, $T_\ell$ with weights $a_1$, $\cdots$, $a_\ell$.

By Lemma \ref{lem: segunda forma}, $g$ is a totally geodesic immersion. Besides, the following equations hold
\begin{align*}
	& \BR_i = a_i^2 \id, \quad \BS_i X =  -a_i^2\left(a_1T_1 X, \cdots, \frac{a_i^2-1}{a_i}T_i X, \cdots, a_\ell T_\ell X\right),\\
	& k_i \BR_i = k_i a_i^2 \id = k \id, \quad \BT_i\BS_i = \left(1-a_i^2 \right) \BS_i \quad \text{and} \\
	& \BS_i\left(T_x \o{} \right) = \set{a_i^2\left(a_1T_1 X, \cdots, \frac{a_i^2-1}{a_i}T_i X, \cdots, a_\ell T_\ell X\right)}{X \in T_x \o{}}.
\end{align*}

%Como a codimensão de $g$ em $\hO$ é $n(\ell - 1)$, então $T^\perp \o{} = \BS(TM)$. Dessa forma, se $\xi \in T^\perp M$, então $\xi = \BS X$, para algum $X$ tangente, logo $\BT \xi = \BT \BS X = a^2 \BS X = a^2 \xi$, ou seja, $\BT = a^2 \id$.

%Let $G \colon \R_{\tau(k)}^{n+1} \to \R_{\ell \cdot\tau(k)}^{\ell\cdot n+\ell}$ be given by $G(x) := \big( a_1 T_1(x), \cdots, a_\ell T_\ell (x) \big)$. It is easy to show that $G$ is a linear isometric immersion, thus $V := G\left(\R_{\tau(k)}^{n+1}\right)$ is a vector subspace of $\R_{\ell\tau(k)}^{\ell n+\ell}$. Besides, $V$ is isomorphic to $\L^{n+1}$ when $k_i < 0$.
%Thus, by Lemma \ref{lem: V} below, $V \cap \S\left(0, \e\hksqrt{|k|^{-1}}\right) \subset \O_{k_1}^{n} \x \cdots \x \O_{k_2}^n$. Therefore, $g\left(\o{} \right) = G\left(\o{}\right) \subset V \cap \S\left(0, \e\hksqrt{|k|^{-1}}\right) \subset \O_{k_1}^{n} \x \cdots \x \O_{k_\ell}^n$.

The weighted sum $g$ plays an important hole at the theory of isometric immersions in $\hO$. This hole is a kind of reduction of codimension theorem (Theorem \ref{teo: Phi=0}). But first we need some results.

\begin{lem}\label{lem: V}
	Let $k_1, \cdots, k_\ell \in \R^*$ be such that $k_ik_j > 0$, for all $i,j \in \{1, \cdots, \ell\}$, and lets denote $M_i := \prod\limits_{\substack{j=1\\ j\ne i}}^\ell k_j$, $\l := \sum\limits_{i=1}^\ell M_i$, $a_i := \left(\frac{M_i}{\l}\right)^{\frac{1}{2}}$ and $\quad k := \tilde{k}_i := a_i^2 k_i$. Let also $V^{n+1} \subset \prod\limits_{i=1}^\ell \R_{\tau(k)}^{n+1}$ be a vector subspace isomorphic to $\R^{n+1}_{\tau(k)}$. With the assumptions above, the following claims are equivalent:
	\begin{enum}
		\item $V \cap \S_k^{n\ell + \ell-1} \subset \prod\limits_{i=1}^\ell \S_{k_i}^n \subset \prod\limits_{i=1}^\ell \R_{\tau(k)}^{n+1}$.
		\item $\left.\pi_i\right|_V$ is a similarity with ratio $a_i$.
		\item There are linear isometries $T_1, \cdots, T_\ell \in \mathrm{O}_{\tau(k)}(n+1)$ such that
		\[V = \set{\left(a_1 T_1(x), \cdots, a_\ell T_\ell(x)\right)}{x \in \R_{\tau(k)}^{n+1}}.\]
	\end{enum}
\end{lem}

\begin{proof}
	Since $k_ik_j > 0$, then $k_1, \cdots, k_\ell$ are all positive or all negative. Thus $k$ and $k_i$ have the same sign.

	\vspace{1ex}\noindent \underline{\textsl{(I)} $\Rightarrow$ \textsl{(II)}:} We have two cases: $k >0$ or $k < 0$.
	
	\noindent \vspace{1ex} \textbf{First case: $k > 0$.}
	
	In this case $\R_{\ell\tau(k)}^{\ell n+\ell} = \R^{\ell n+\ell}$ is an euclidean space. Thus,
	\begin{align*}
	& V \cap \S_k^{n\ell + \ell-1} \subset \prod_{i=1}^\ell \S_{k_i}^n \sss \frac{x}{\hksqrt{k} \|x\|} \in \prod_{i=1}^\ell \S_{k_i}^n, \ \forall x \in V\setminus\{0\} \sss \\
	& \sss \frac{\|\pi_i x\|^2}{k \|x\|^2} = \frac{1}{k_i}, \ \forall x \in V\setminus\{0\} \sss \ \|\pi_i x\|^2 = \frac{k \|x\|^2}{k_i} = a_i^2 \|x\|^2, \ \forall x \in V.
	\end{align*}
	
	Therefore \textsl{(I) $\Rightarrow$ (II)}.
	
	\vspace{1ex} \noindent \textbf{Second case: $k < 0$.}
	
	In this case $\R_{\ell \tau(k)}^{\ell n+\ell} = \R_\ell^{\ell n+\ell}$ and
	$V \cap \S_k^{n\ell + \ell -1} = \set{\frac{x}{\hksqrt{|k|} \|x\|}}{x \in V \ \text{and} \ \|x\|^2 < 0}$.
	Hence,
	\begin{align*}
	& V \cap \S_k^{n\ell + \ell -1} \subset \prod_{i=1}^\ell \S_{k_i}^n \sss \frac{x}{\hksqrt{|k|} \|x\|} \in \prod_{i=1}^\ell \S_{k_i}^n, \ \forall x \in V \ \text{with} \ \|x\|^2 < 0 \sss \\
	& \sss \frac{\|\pi_i x\|^2}{k \|x\|^2} = \frac{1}{k_i}, \ \forall x \in V \ \text{with} \ \|x\|^2 < 0 \sss \\
	& \sss \|\pi_i x\|^2 = \frac{k \|x\|^2}{k_i} = a_i^2\|x\|^2, \ \forall x \in V \ \text{with} \ \|x\|^2 < 0.
	\end{align*}
	
	%	If each $\left.\pi_i\right|_V$ is a similarity of ratio $a_i$, then the equivalences above assure us that $V \cap \S\left(0, \e\hksqrt{|k|^{-1}}\right) \subset \ss{1}\x\cdots \x\ss{\ell}$.
	
	Lets suppose that $V \cap \S_k^{n\ell + \ell -1} \subset \prod\limits_{i=1}^\ell \S_{k_i}^n$. We will show that each $\left.\pi_i\right|_V$ is a similarity of ratio $a_i$.
	
	Let $\{e_0, \cdots, e_n\}$ be a orthonormal base of $V$ with $e_0$ timelike. Then $\|\pi_ie_0\|^2 = -a_i^2$. Let $\alpha, \beta \in \R$ be such that $-\alpha^2 + \beta^2 < 0$, and let $j \in \{1, \cdots, n\}$. Thus $\alpha e_0 + \beta e_j$ is timelike and, by the equivalences above, $\left\|\pi_i\left(\alpha e_0 + \beta e_j \right) \right\|^2 = a_i^2 \left(-\alpha^2 + \beta^2 \right)$.
	
	On the other side,
	\begin{align*}
	& \left\|\pi_i\left(\alpha e_0 + \beta e_j \right) \right\|^2 = \alpha^2 \|\pi_ie_0\|^2 + 2\alpha\beta \interno{\pi_ie_0}{\pi_ie_j} + \beta^2 \|\pi_ie_j\|^2 = \\
	& =-a_i^2\alpha^2 + 2\alpha\beta \interno{\pi_ie_0}{\pi_ie_j} + \beta^2 \|\pi_ie_j\|^2.
	\end{align*}
	Hence
	\begin{align*}
	& -a_i^2\alpha^2 + 2\alpha\beta \interno{\pi_ie_0}{\pi_ie_j} + \beta^2 \|\pi_ie_j\|^2 = a_i^2 \left(-\alpha^2 + \beta^2 \right) \ \Rightarrow \\
	& \Rightarrow \ 2\alpha\beta \interno{\pi_ie_0}{\pi_ie_j} + \beta^2 \|\pi_ie_j\|^2 = a_i^2 \beta^2 \ \Rightarrow \\
	& \Rightarrow \interno{\pi_ie_0}{\pi_ie_j} = \frac{\beta}{2\alpha}\left(a_i^2 - \|\pi_ie_j\|^2 \right), \ \text{if} \ |\alpha|>|\beta|.
	\end{align*}
	
	Thus, if we first take $\alpha=2$ and $\beta =1$, and then we take $\alpha=2$ and $\beta = -1$, we will conclude that
	\[\interno{\pi_ie_0}{\pi_ie_j} = \frac{1}{4}\left(a_i^2 - \|\pi_ie_j\|^2 \right) \quad \text{and} \quad \interno{\pi_ie_0}{\pi_ie_j} = \frac{-1}{4}\left(a_i^2 - \|\pi_ie_j\|^2 \right).\]
	Therefore $\interno{\pi_ie_0}{\pi_i e_j} = 0$ and $\|\pi_ie_j\|^2 = a_i^2$.
	
	Now, let $\alpha$, $\beta$, $\g \in \R$ such that $\alpha^2 > \beta^2 + \g^2$, and let $j_1, j_2 \in \{1, \cdots, n\}$ with $j_1\ne j_2$. Thus $\alpha e_0 + \beta e_{j_1} + \gamma e_{j_2}$ is timelike and
	\begin{multline*}
	a_i^2\left(-\alpha^2 + \beta^2 + \g^2 \right) = \left\|\pi_i\left(\alpha e_0 + \beta e_{j_1} + \g e_{j_2}\right) \right\|^2 =\\
	= a_i^2\left(-\alpha^2 + \beta^2 + \g^2 \right) + 2 \beta \g \interno{\pi_ie_i}{\pi_ie_j}.
	\end{multline*}
	
	Therefore, $\interno{\pi_ie_{j_1}}{\pi_ie_{j_2}} = 0$ and we conclude that $\left.\pi_i\right|_V$ is a similarity with ratio $a_i$. \cqd

	\noindent \underline{\textsl{(II)} $\Rightarrow$ \textsl{(III)}:} Lets suppose that each $\left.\pi_i\right|_V$ be a similarity of ratio $a_i$, thus each $a_i^{-1}\cdot \left.\pi_i\right|_V$ is a linear isometry on its image. So, given an orthonormal base $\{v_0, \cdots, v_n\}$ of $V$, let $G \colon \R_{\tau(k)}^{n+1} \to V$ be the linear isometry such that $G(e_i) = v_i$, where $\{e_0, \cdots, e_n\}$ is the canonical base of $\R_{\tau(k)}^{n+1}$ (with $e_0$ timelike, if $k < 0$). Now let $T_i$ be given by $T_i := a_i^{-1} \cdot \left(P_i \circ G \right)$, where $P_i$ is the projection of $\R_{\ell \tau(k)}^{\ell n+\ell} = \prod\limits_{i=1}^\ell \R_{\tau(k)}^{n+1}$ on its $i$th factor. Hence each $T_i$ is linear and
	\[V = \set{G(x)}{x \in \R_{\tau(k)}^{n+1}} = \set{\big(a_1 T_1(x), \cdots, a_\ell T_\ell(x) \big)}{x \in \R_{\tau(k)}^{n+1}}.\]
	
	On the other side, $\interno{T_i(x)}{T_i(y)} = \frac{1}{a_i^2}\interno{\pi_i(G(x))}{\pi_i(G(y))} = \interno{x}{y}$, therefore each $T_i$ is an isometry. \cqd

	\noindent \underline{\textsl{(III)} $\Rightarrow$ \textsl{(I)}:}
	Lets suppose that there are $T_1, \cdots, T_\ell \in \mathrm{O}_{\tau(k)}(n+1)$ such that
	\[V = \set{ \big(a_1 T_1(x), \cdots, a_\ell T_\ell(x) \big)}{x \in \R_{\tau(k)}^{n+1}}.\]
	
	If $y \in V \cap \S(0,\e\hksqrt{|k|^{-1}})$, then $y = \big(a_1 T_1(x), \cdots, a_\ell T_\ell(x) \big)$, for some $x \in \R_{\tau(k)}^{n+1}$ and $\|y\|^2 = \frac{1}{k}$. Since $\sum\limits_{i=1}^\ell a_i^2 = 1$, then $\|x\|^2 = \|T_i(x)\|^2 = \|y\|^2 = \frac{1}{k}$. Thus $\|\pi_i(y)\|^2 = a_i^2 \|T_i(x)\|^2 = \frac{a_i^2}{k} = \frac{1}{k_i}$. Therefore $y \in \S_{k_1}^n \x \cdots \x \S_{k_\ell}^n$.
\end{proof}

\begin{cor}\label{cor: V}
	Let $k_1, \cdots, k_\ell \in \R_-^*$ and let $a_i$ and $k$ be like in Lemma \ref{lem: V}. Let also $V^{n+1} \subset \prod\limits_{i=1}^\ell \L^{n+1}$ be a vector subspace isomorphic to $\L^{n+1}$ and $C$ one of the two connected components of $V \cap \S_k^{n\ell + \ell -1}$. With the assumptions above, if $C \subset \prod\limits_{i=1}^\ell \O_{k_i}^n$, then there are linear isometries $T_1, \cdots, T_\ell \in \mathrm{O}_{\tau(k)}(n+1)$ such that $V = \set{\left(a_1 T_1(x), \cdots, a_\ell T_\ell(x)\right)}{x \in \R_{\tau(k)}^{n+1}}$ and	$T_i\left(\O_{k_i}^n \right) = \O_{k_i}^n$.
\end{cor}

\begin{proof}
	If $C \subset \prod\limits_{i=1}^\ell \O_{k_i}^n$, then %$P_i(C) \subset \O_{k_i}^n$, for all $i \in \{1, \cdots, \ell\}$, where $P_i$ is the projection in the $i$th component. Hence
	\[V \cap \S_k^{n\ell + \ell -1} = C \cup (-C) \subset \left( \prod_{i=1}^\ell \O_{k_i}^n \right) \cup \left(-\prod_{i=1}^\ell \O_{k_i}^n \right) = \prod_{i=1}^\ell \S_{k_i}^\ell.\]
	Thus, by Lemma \ref{lem: V}, $V = \set{\left(a_1 T_1(x), \cdots, a_\ell T_\ell(x)\right)}{x \in \R_{\tau(k)}^{n+1}}$, for some linear isometries $T_1, \cdots, T_\ell \in \mathrm{O}_{\tau(k)}(n+1)$.
	
	So, if $y \in C$, there is an $x \in \R_{\tau(k)}^{n+1}$ such that $ y = \left(a_1 T_1(x), \cdots, a_\ell T_\ell(x)\right)$. But $P_i(y) = a_i T_i(x) \in \O_{k_i}^n$, for all $i \in \{1, \cdots, \ell\}$, where $P_i$ is the projection on the $i$th coordinate. Hence $T_i(x) \in \O_k^n$, for all $i \in \{1, \cdots, n \}$, and $x \in \pm \O_k^n$.
	
	If $x \in \O_k^n$, then $T_i\left(\O_k^n \right) = \O_k^n$ and $T_i \left( \O_{k_i}^n \right) = \O_{k_i}^n$, for all $i$, otherwise, taking $\tilde T_i = -T_i$, we conclude that $V = \set{\left(a_1 \tilde T_1(x), \cdots, a_\ell \tilde T_\ell(x)\right)}{x \in \R_{\tau(k)}^{n+1}}$ and $\tilde T_i \left(\O_{k_i}^n\right) = \O_{k_i}^n$.
\end{proof}

\begin{df}
	Let $f \colon M \to N$ be an isometric immersion. The \textbf{first normal space} of $f$ at $x$, is the space $\mathcal{N}_1(x) := \spa\set{\al{X}{Y}}{X, Y \in T_x M}$.
\end{df}

\begin{lem}\label{lem: sinal}
	Let $a_1, \cdots a_n \in \R^*$ and $A_i := \prod\limits_{\substack{j=1\\ j\ne i}}^n a_j$. If $A_iA_j > 0$, for all $i,j \in \{1, \cdots, n\}$, then $a_ia_j >0$, for all $i,j \in \{1, \cdots, n\}$.
\end{lem}

\begin{proof}
	Lets suppose, by absurd, that $a_1, \cdots, a_n$ do not have all the same sign. Hence, reordering if necessary, we can suppose that $a_1, \cdots, a_m$ are negative and that $a_{m+1}, \cdots, a_n$ are positive. Thus
	\begin{align*}
	A_1 &= \prod_{j=2}^n a_j = \prod_{j=2}^m a_j \cdot \prod_{j=m+1}^n a_j = (-1)^{m-1}\prod_{j=2}^n|a_j|, \\
	A_n &= \prod_{j=1}^{n-1} a_j = \prod_{j=1}^m a_j \cdot \prod_{j=m+1}^{n-1} a_j = (-1)^m\prod_{j=1}^{n-1}|a_j|.
	\end{align*}
	Therefore $A_1A_n < 0$, which is a contradiction.
\end{proof}

\begin{prop}\label{prop: Phi=0}
	Let $f \colon M \to \hO$ be an isometric immersion with $k_1$, $\cdots$, $k_\ell \in \R^*$. Thus, if $k_i \BR_i = k_j \BR_j$, for every $i, j$, then the following claims hold:
	\begin{enum}
		\item $\l := \sum\limits_{i=1}^\ell M_i \ne 0$, where $M_i := \prod\limits_{\substack{j=1\\ j \ne i}}^\ell k_j$.
		\item $\BR_i = \l_i \id$, where $\l_i := \frac{M_i}{\l}$.
		\item $k_ik_j > 0$.
		\item $\BL_i$ is a similarity of ratio $a_i := \hksqrt{\l_i}$.
		\item $\|\pi_i F\|^2 = a_i^2 \|F\|^2$.
		\item $\mathcal{N}_1 \perp \BS(TM)$.
		\item If $\mathcal{N}_1$ is parallel in the normal connexion of $f$, then each $\left.\pi_i\right|_{\mathcal{N}_1}$ is a similarity of ratio $a_i$.
	\end{enum}
\end{prop}

\begin{prova} From $k_i \BR_i = k_j \BR_j$, we know that $\BR_j = \frac{k_i}{k_j} \BR_i$.

	\noindent \textbf{\textsl{(I)} and \textsl{(II)}:} Using the first equation in \eqref{somas},
	
	\begin{multline*}
		\id = \sum_{j=1}^n \BR_j = \BR_i + \sum_{\substack{j=1,\\ j\ne i}}^\ell \BR_j = \left[1 + \sum_{\substack{j=1\\ j\ne i}}^\ell \frac{k_i}{k_j}\right]\BR_i = \left[1 + k_i \frac{\sum\limits_{\substack{j=1\\j \ne i}}^\ell \prod\limits_{\substack{l=1\\ l \notin\{i, j\}}}^\ell k_l}{ \prod\limits_{\substack{j=1\\ j\ne i}}^\ell k_j}\right] \BR_i = \\
		= \frac{\prod\limits_{\substack{j=1\\ j\ne i}}^\ell k_j + \sum\limits_{\substack{j=1\\ j\ne i}}^\ell \prod\limits_{\substack{l=1\\ l\ne j}}^\ell k_l}{\prod\limits_{\substack{j=1\\ j\ne i}}^\ell k_j} \BR_i = \frac{M_i + \sum\limits_{\substack{j=1\\ j\ne i}}^\ell M_j}{M_i} \BR_i = \frac{\l}{M_i} \BR_i.
	\end{multline*}
	Therefore $\l \ne 0$ and $\BR_i = \frac{M_i}{\l} \id$. \cqd

	\noindent \textbf{\textsl{(III)}:} Let $X \in T M$ with $X \ne 0$. Hence
	\[\interno{\BS_i X}{\BS_i X} = \interno{\BS_i^\t\BS_i X}{X} \stackrel{\eqref{eq: RST}}{=} \interno{\BR_i(\id - \BR_i)X}{X} \stackrel{\text{\textsl{(II)}}}{=} \l_i(1-\l_i) \|X\|^2  \geq 0.\]

	Since $\BR_i \ne 0$, then $\l_i \in (0,1]$, and it follows that $\l$ and $M_i$ have the same sign. Therefore, $M_1$, $\cdots$, $M_\ell$ have all the same sign and, by Lemma \ref{lem: sinal}, $k_ik_j >0$, for all $i,j \in \{1, \cdots, \ell\}$. \cqd

	\noindent \textbf{\textsl{(IV)}:} Since $\BL_i^\t\BL_i = \BR_i = \l_i \id$, then $\interno{\BL_i X}{\BL_i Y} = \l_i\interno{X}{Y}$. Therefore $\BL_i$ is a similarity of ratio $a_i = \hksqrt{\l_i}$. \cqd

	\noindent \textbf{\textsl{(V)}:} $\|\pi_i F\|^2 = \frac{1}{k_i} = \frac{\prod\limits_{\substack{j=1\\ j\ne i}}^\ell k_j}{\prod\limits_{j=1}^\ell k_j} = \frac{\prod\limits_{\substack{j=1\\ j\ne i}}^\ell k_j}{\l}\cdot \frac{\l}{\prod\limits_{j=1}^\ell k_j} = \l_i \sum\limits_{i=1}^\ell \frac{1}{k_i} = \l_i \|F\|^2$. \cqd

	\noindent \textbf{\textsl{(VI)}:} Lets consider the trilinear application $\beta \colon T_xM\x T_xM \x T_xM \to \R$ given by $\beta(X,Y,Z) = \interno{\al{X}{Y}}{\BS_i Z}$.
	
	Using equation \eqref{eq: derivada R}, we conclude that $A_{\BS_i Y} X = - \BS_i^\t \al{X}{Y}$. Thus,
	\[\interno{\al{X}{Z}}{\BS_i Y} = \interno{A_{\BS_i Y} X}{Z} = - \interno{\BS_i^\t \al{X}{Y}}{Z} = -\interno{\al{X}{Y}}{\BS_i Z}.\]
	Hence $\beta$ is symmetric in the first 2 variables and antisymmetric in the last two, thus $\beta = 0$ and $\al{X}{Y} \perp \BS_i(T_x M)$. \cqd

	\noindent \textbf{\textsl{(VII)}:} Lets suppose that $\mathcal{N}_1$ is parallel in the normal connexion. Hence, by equation \eqref{eq: derivada S},
	\begin{align*}
		& \interno{(\n_X \BS_i)Y}{\al{W}{Z}} = \interno{\BT_i \al{X}{Y} - \al{X}{\BR_i Y}}{\al{W}{Z}} \ \Rightarrow \\
		& \Rightarrow \ \interno{\nperp_X \BS_i Y}{\al{W}{Z}} - \interno{\BS_i \n_X Y}{\al{W}{Z}} = \interno{\BT_i \al{X}{Y}}{\al{W}{Z}} - \\
			& \quad - \interno{\l_i \al{X}{Y}}{\al{W}{Z}} \ \Rightarrow \\
		& \stackrel{\textsl{(VI)}}{\Rightarrow} \ \interno{\nperp_X \BS_i Y}{\al{W}{Z}} = \interno{\BT_i \al{X}{Y}}{\al{W}{Z}} - \\
			& \quad \l_i \interno{\al{X}{Y}}{\al{W}{Z}} \ \Rightarrow \\
		& \Rightarrow \ - \cancel{\interno{\BS_i Y}{\nperp_X \al{W}{Z}}} = \interno{\BT_i \al{X}{Y}}{\al{W}{Z}} - \\
			& \quad - \l_i \interno{\al{X}{Y}}{\al{W}{Z}} \ \Rightarrow \\
		& \Rightarrow \ \interno{\BT_i \al{X}{Y}}{\al{W}{Z}} = \l_i \interno{\al{X}{Y}}{\al{W}{Z}} \ \Rightarrow \\
		& \Rightarrow \ \interno{\pi_i\al{X}{Y}}{\pi_i\al{W}{Z}} = \l_i \interno{\al{X}{Y}}{\al{W}{Z}}.
	\end{align*}

	Therefore $\left.\pi_1\right|_{\mathcal{N}_1}$ is a similarity with ratio $a_i = \hksqrt{\l_i}$.
\end{prova}

\begin{teo}\label{teo: Phi=0}
	Let $f \colon M^m \to \hO$ be an isometric immersion with $k_1$, $\cdots$, $k_\ell \in \R^*$ and lets assume that $\mathcal{N}_1$ is a parallel vector sub-bundle of $T^\perp M$ with dimension $\bar n$. If $k_i \BR_i = k_j \BR_j$, for every $i, j \in \{1, \cdots, \ell\}$, then $k_ik_j > 0$ for all pairs $i, j$, $m+ \bar n \leq \min\{n_1, \cdots, n_\ell\}$ and $f = (\jmath_1\x\cdots \x \jmath_\ell) \circ g \circ \bar f$, where each $\jmath_i \colon \O_{k_i}^{m+\bar n} \hookrightarrow \o{i}$ is a totally geodesic inclusion, $g \colon \O_k^{m+\bar n} \to \prod\limits_{i=1}^\ell \O_{k_1}^{m+\bar n}$ is a totally geodesic immersion like the one given at the beginning of this section and $\bar f \colon M^m \to \O_k^{m+\bar n}$ is a isometric immersion.
	
\end{teo}

\begin{proof}
	With our assumptions, all conclusions of Proposition \ref{prop: Phi=0} are true, hence $k_ik_j > 0$, for every $i$ and $j$.
	
	Let $\imath \colon \hO \hookrightarrow \RN$ be the canonical inclusion and $F = \imath \circ f$. Thus $V := F_*TM \op \imath_* \mathcal{N}_1^f \op \spa{F}$ is a vector bundle with dimension $m+\bar n + 1$.
	
	\begin{afi}{$V$ is a constant vector subspace of $\RN$.}
		Let $X,Y,Z \in \G(TM)$. Thus
		\begin{align*}
			& \ntil_X F_* Y \stackrel{\text{Lemma \ref{lem: aF}}}{=} F_* \n_X Y + \imath_* \af{X}{Y} + \sum_{i=1}^\ell \interno{\BR_i X}{Y} \nu_i = \\
			& = F_* \n_X Y + \imath_* \af{X}{Y} + \sum_{i=1}^\ell \l_i \interno{X}{Y} k_i (\pi_i \circ F) = \\
			& = F_* \n_X Y + \imath_* \af{X}{Y} + \sum_{i=1}^\ell \frac{\prod\limits_{\substack{j=1\\ j\ne i}}^\ell k_j}{\l} \interno{X}{Y} k_i (\pi_i \circ F) = \\
			& = F_* \n_X Y + \imath_* \af{X}{Y} + \frac{\prod\limits_{j=1}^\ell k_j}{\l} \interno{X}{Y} F \in \G(V); \\
			& \ntil_X \imath_* \af{Y}{Z} = \imath_* \nbar_X \af{Y}{Z} + \ai{f_* X}{\af{Y}{Z}} = \\
			&= -F_* A^f_{\af{Y}{Z}} X + \imath_* \nperp_X \af{Y}{Z} + \sum_{i=1}^\ell \interno{\pi_i f_*X}{\af{Y}{Z}}\nu_i = \\
			&= -F_* A^f_{\af{Y}{Z}} X + \imath_* \nperp_X \af{Y}{Z} + \sum_{i=1}^\ell \cancel{\interno{\BS_i f_*X}{\af{Y}{Z}}}\nu_i = \\
			&= -F_* A^f_{\af{Y}{Z}} X + \imath_* \nperp_X \af{Y}{Z}.
		\end{align*}
	
		Since $\mathcal{N}_1$ is parallel in the normal connexion of $f$, $\ntil_X \imath_* \af{Y}{Z} \in \Gamma(V)$.
	
		Last, $\ntil_X F = F_* X \in V$. Thus $V$ is a parallel sub-bundle of $\RN$, therefore $V$ is a constant subspace on $\RN$.
	\end{afi}
	
	Because $V$ is a constant subspace of $\R_y^N$ and $F_*(TM) \subset V$, then $F(M) \subset F(x_0) + V$, for any fixed $x_0 \in M$, because $X \interno{F - F(x_0)}{\xi} = \interno{F_* X}{\xi} = 0$, for all constant $\xi \in V^\perp$. But $\spa\{F(x_0)\} \subset V$, hence $F(M) \subset V$.
	
	Let $x_0 \in M$ be a fixed point, then $V = \spa\left\{\frac{F(x_0)}{\|F(x_0)\|}\right\}\op F_*T_{x_0}M \op \imath_* \mathcal{N}_1(x_0)$.
	
	Since $k_1, \cdots, k_\ell$ have the same sign, $\|f(x)\|^2 = \sum\limits_{i=1}^\ell \frac{1}{k_i} \ne 0$, for all $x \in M$. Thus $V$ is isomorphic to $\R_{\tau(k_1)}^{m+\bar n +1}$.

	\begin{afi}{Each $\left.\pi_i\right|_V$ is a similarity of ratio $a_i = \hksqrt{\l_i}$.}
		Let $X,Y,Z \in T_{x_0}M$. Hence
		\begin{align*}
		& \interno{\pi_iF(x_0)}{\pi_iF_*X} = \interno{\nu_i(x_0)}{F_*X} = 0; \\
		& \interno{\pi_iF(x_0)}{\pi_i\imath_* \af{X}{Y}} = \interno{\nu_i(x_0)}{\imath_* \af{X}{Y}} = 0; \\
		& \interno{\pi_iF_*X}{\pi_i\imath_*\af{Y}{Z}} = \interno{-\BS_i X}{\af{Y}{Z}} = 0.
		\end{align*}
		Thus $\left.\pi_i\right|_V = \left.\pi_i\right|_{\spa\{F(x_0)\}}\x\left.\pi_i\right|_{F_*T_{x_0}M}\x\left.\pi_i\right|_{\mathcal{N}_1^f(x_0)}$, and,  by items \textsl{(IV)}, \textsl{(V)} and \textsl{(VII)} of Proposition \ref{prop: Phi=0}, we know that $\left.\pi_i\right|_{\spa\{F(x_0)\}}$, $\left.\pi_i\right|_{F_* T_{x_0} M}$ and $\left.\pi_i\right|_{\mathcal{N}_1^f(x_0)}$ are similarities of ratio $a_i$. Therefore $\left.\pi_i\right|_V$ is a similarity of ratio $a_i$.
	\end{afi}
	
	Given that each $\left.\pi_i\right|_V$ is a similarity of ratio $a_i$, then $\pi_i(V) = \{(0, \cdots, 0)\} \x \R_{\tau(k_1)}^{m+\bar n+1}\x\{(0, \cdots, 0)\}$. Hence $V \subset \prod\limits_{i=1}^\ell\R_{\tau(k_1)}^{m+\bar n+1}$. But $f(M) \subset V \cap \left( \XO \right)$, hence, $f(M) \subset \prod\limits_{i=1}^\ell \O_{k_i}^{m+\bar n}$. Therefore, $f = (\jmath_1\x \cdots, \x \jmath_\ell) \circ \tilde{f}$, where each $\jmath_i \colon \O_{k_1}^{m+\bar n} \to \o{i}$ is a inclusion and $\tilde{f} \colon M \to \O_{k_1}^{m+\bar n}\x \cdots \x \O_{k_\ell}^{m+\bar n}$ is a isometric immersion with $\tilde{f}(M) \subset V$.
	
	Since $f(M) \subset V \cap \prod\limits_{i=1}^\ell \O_{k_i}^{m+\bar n}$, then there is a connected component $C$ of $V \cap \S_k^{(m+\bar n)\ell + \ell -1}$ such that $\pi(C) \subset \O_{k_i}^{m+\bar n}$. Thus by Lemma \ref{lem: V} (and Corollary \ref{cor: V},if $k < 0$) there are linear isometries $T_1, \cdots, T_\ell$ such that $T_i\o{i} = \o{i}$ and
	\[V = \set{G(x)}{x \in \R_{\tau(k)}^{n+1}} = \set{\big(a_1 T_1(x), \cdots, a_\ell T_\ell(x) \big)}{x \in \R_{\tau(k)}^{n+1}}.\]
	
	Let $g \colon \O_k^{m+\bar n} \to \O_{k_1}^{m+\bar n}\x \cdots \x \O_{k_\ell}^{m+\bar n}$ be the totally geodesic isometric immersion given at the beginning of this section.
	
	\begin{afi}{$\tilde f(M) \subset g\left( \O_k^{m+\bar n} \right)$.}	
		\begin{align*}
		& y \in \tilde f(M) \Rightarrow y \in V \cap \prod_{i=1}^\ell \O_{k_i}^{m+\bar n} \Rightarrow \\
		& \Rightarrow y = (a_1T_1(x), \cdots, a_\ell T_\ell(x)), \ \text{for some} \ x \in \R^{m+\bar n+1}_{\tau(k)} \ \text{and} \ a_iT_i(x) \in \O_{k_i}^{m+\bar n} \Rightarrow \\
		& \Rightarrow \pi_i(y) = a_i T_i(x) \ \text{and} \ x \in \O_{\tilde k_i}^{m + \bar n}.
		\end{align*}
		Therefore $y \in g\left( \O_k^{m+\bar n} \right)$.
	\end{afi}
	
	Let $\bar{f} := g^{-1} \circ \tilde{f} \colon M \to \O_k^{m+\bar n}$. Thus $\bar{f}$ is a isometric immersion and $f = (\jmath_1\x \cdots \x \jmath_\ell) \circ g \circ \bar{f}$.
\end{proof}
	\section{Reduction of codimension}
		We say that the codimension of $f \colon M^m \to \hO$ \textbf{is reduced by $\bar n$ at the $i$th coordinate}, if there is a totally geodesic submanifold $\O_{k_i}^{m_i} \subset \o{i}$ such that $n_i - m_i = \bar n$ and $f(M) \subset \o{1}\x \cdots \x \O_{k_i}^{m_i} \x \cdots \x \o{\ell}$.

For each $i \in \{1, \cdots, \ell\}$, let $\O_{k_i}^{m_i} \subset \o{i}$ be a totally geodesic submanifold. Let also $\R_{\tau(k_i)}^{m_i+\upsilon(k_i)} \subset \R^{N_i}$ be such that $\O_{k_i}^{m_i} = \R_{\tau(k_i)}^{m_i+\upsilon(k_i)} \cap \o{i}$. Lets consider the orthogonal projections
\[\begin{matrix}
	\Pi_i : & \R_{\tau(k_1)}^{m_1+\upsilon(k_1)}\x \cdots \x \R_{\tau(k_\ell)}^{m_\ell+\upsilon(k_\ell)} & \longrightarrow & \R_{\tau(k_1)}^{m_1+\upsilon(k_1)}\x \cdots \x \R_{\tau(k_\ell)}^{m_\ell+\upsilon(k_\ell)} \\
	& (x_1, \cdots, x_n) & \longmapsto & (0, \cdots, x_i, \cdots, 0).
\end{matrix}\]
Thus each $\Pi_i$ is the restriction of $\pi_i$ to $\R_{\tau(k_1)}^{m_1+\upsilon(k_1)}\x \cdots \x \R_{\tau(k_\ell)}^{m_\ell+\upsilon(k_\ell)}$.

Now, let $\tilde{f} \colon M^m \to \O_{k_1}^{m_1}\x\cdots\x\O_{k_\ell}^{m_\ell}$ be a isometric immersion and $f \colon M^m \to \hO$ be given by $f := (\jmath_1 \x \cdots \x \jmath_\ell) \circ \tilde{f}$, where each $\jmath_i \colon \O_{k_i}^{m_i} \hookrightarrow \o{i}$ is a totally geodesic inclusion. We will also use $\jmath_i$ denote the inclusion $\jmath_i \colon \R_{\tau(k_i)}^{m_i} \to \R_{\tau(k_i)}^{n_i}$, hence, $(\jmath_1 \x \cdots \x \jmath_\ell) \circ \Pi_i = \pi_i \circ (\jmath_1 \x \cdots \x \jmath_\ell)$ and
\begin{align*}
	& \pi_i f_* X = \left(\jmath_1\x \cdots \x\jmath_\ell\right)_*\Pi_i \tilde{f}_* X \Rightarrow \\
	& \quad \Rightarrow f_* \BR_i^f X + \BS_i^f X = f_* \BR_i^{\tilde f} X + (\jmath_1 \x \cdots \x \jmath_\ell)_* \BS_i^{\tilde f} X. \\
	& \pi_i (\jmath_1 \x \cdots \x \jmath_\ell)_* \xi = (\jmath_1 \x \cdots \x \jmath_\ell)_*\Pi_i \xi \ \Rightarrow \\
	& \Rightarrow \ f_*\left(\BS_i^f\right)^\t (\jmath_1 \x \cdots \x \jmath_\ell)_* \xi + \BT_i^f (\jmath_1 \x \cdots \x \jmath_\ell)_* \xi = \\
	& \quad = f_* \big(\BS_i^{\tilde f}\big)^\t\xi + (\jmath_1 \x \cdots \x \jmath_\ell)_* \BT_i^{\tilde f} \xi.
\end{align*}
Therefore
\begin{equation}\label{eq: tensores composta}
	\begin{aligned}
		& \BR_i^f = \BR_i^{\tilde f}, \quad \BS_i^f = (\jmath_1 \x \cdots \x \jmath_\ell)_*\BS_i^{\tilde f}, \\
		& \left(\BS_i^f\right)^\t \left.(\jmath_1 \x \cdots \x \jmath_\ell)_*\right|_{T_{\tilde f}^\perp M} = \big(\BS_i^{\tilde f}\big)^\t,\\
		& \BT_i^f \left.(\jmath_1 \x \cdots \x \jmath_\ell)_*\right|_{T_{\tilde f}^\perp M} = (\jmath_1 \x \cdots \x \jmath_\ell)_* \BT_i^{\tilde f}.
	\end{aligned}
\end{equation}

In the rest of this section we are going to prove some results about reduction of codimension of $f$ at the $i$th coordinate.

\begin{lem}\label{lem: reducao}
	If $f\colon M^m \to \hO$ is a isometric immersion, then:
	\begin{enum}
		\item $\BS_i(TM)^\perp$ is invariant under $\BT_i$ and $\BS_i(TM)^\perp = \mathcal{U}_i \op \mathcal{V}_i$, where $\U_i := \ker \BT_i$ and $\V_i := \ker(\id - \BT_i)$.
		\item $\nperp \left( \U_i \cap \mathcal{N}_1^\perp \right)\subset \U_i$ and $\nperp \left(\V_i \cap \mathcal{N}_1^\perp \right)\subset \V_i$.
		\item $\pi_i$ fixes the points of $\V_i$ and $\id - \pi_i$ fixes the points of $\U_i$.
	\end{enum}
\end{lem}

\begin{prova}
	\par \noindent \textbf{\textsl{(I)}:} It follows from the second equation in \eqref{eq: RST} that $\BT_i \left[ \BS_i(TM) \right] \subset \BS_i(TM)$. Hence, if $\xi \in \BS_i(TM)$ and $\z \in \BS_i(TM)^\perp$, then $\interno{\xi}{\BT_i \z} = \interno{\BT_i \xi}{\z} = 0$. Therefore $\BT_i$ leaves $\BS_i(TM)^\perp$ invariant.
	
	From the third equation in \eqref{eq: RST}, it follows that $\left.\BT_i\right|_{\BS_i(TM)^\perp}^2 = \left.\BT_i\right|_{\BS_i(TM)^\perp}$, thus $\BT_i|_{\BS_i(TM)^\perp}$ is a orthogonal projection and $\BS_i(TM)^\perp = {\U_i \op \V_i}$, where $\U_i := \ker \BT_i|_{\BS_i(TM)^\perp}$ and $\V_i := \left.\ker(\id -\BT_i)\right|_{\BS_i(TM)^\perp}$.
	
	By another side, using the third equation \eqref{eq: RST},
	\[\ker[\BT_i(\id-\BT_i)] = \ker \BS_i\BS_i^\t = \ker\BS_i^\t = [\BS_i(TM)]^\perp.\]
	Besides, $\ker \BT_i$ and $\ker(\id-\BT_i)$ are subsets of $\ker[\BT_i(\id-\BT_i)] = \BS_i(TM)^\perp$, therefore $\ker \BT_i = \ker\BT_i|_{\BS_i(TM)^\perp}$ and $\ker(\id-\BT_i) = \ker(\id-\BT_i)|_{\BS_i(TM)^\perp}$. \cqd

	\noindent \textbf{\textsl{(II)}:} If $\xi \in \G \left(\V_i \cap \mathcal{N}_1^\perp \right)$, then
	\[\nperp_X \xi - \BT_i \nperp_X \xi = \nperp_X \BT_i \xi - \BT_i \nperp_X \xi \stackrel{\eqref{eq: derivada T}}{=} 0, \ \forall X \in \G(TM).\]
	Hence $\nperp_X \xi \in \ker(\id - \BT_i) = \V_i$.
	
	If $\xi \in \U_i \cap \mathcal{N}_1^\perp$, then
	\[-\BT_i \nperp_X \xi = \nperp_X \BT_i \xi - \BT_i \nperp_X \xi \stackrel{\eqref{eq: derivada T}}{=} 0, \ \forall X \in \G(TM).\]
	Therefore $\nperp_X \xi \in \ker \BT_i = \U_i$. \cqd

	\noindent \textbf{\textsl{(III)}:} Let $\xi \in \U_i = \ker \BT_i$ and $\z \in \V_i = \ker(\id - \BT_i)$ . From $\BS_i(TM)^\perp = \mathcal{U}_i \op \mathcal{V}_i$, it follows that
	\begin{align*}
		& (\id - \pi_i) \xi = \xi - \BS_i^\t \xi - \BT_i \xi = \xi; \\
		& \pi_i \z = \BS_i^\t \z + \BT_i \z = \z.
	\end{align*}
	
	Therefore $\pi_i$ fixes the points of $\V_i$, and $\id-\pi_i$ fixes the points of $\U_i$.
\end{prova}

\begin{teo}\label{teo: reducao}
	If $f\colon M^m\to \hO$ is a isometric immersion, then the following claims are equivalent:
	\begin{enum}
		\item The codimension of $f$ is reduced by $\bar n$ at the $i$th coordinate.
		\item There is a vector sub-bundle $L^{\bar n} \subset T^\perp M$ such that $L^{\bar n} \subset \V_i \cap \mathcal{N}_1^\perp$ and $L^{\bar n}$ is parallel in the normal connection of $f$.
	\end{enum}

\end{teo}

\begin{prova}
	\par \noindent \textbf{\textsl{(I)} $\Rightarrow$ \textsl{(II)}:} Because the codimension of $f$ is reduced by $\bar n$ at the $i$th coordinate, there is a totally geodesic submanifold $\O_{k_i}^{m_i} \subset \o{i}$, with $n_i - m_i = \bar n$, and such that $f(M) \subset \o{1}\x\cdots\x\O_{k_i}^{m_i}\x\cdots \x\o{\ell}$.
	
	Hence, there is a isometric immersion $\tilde{f} \colon M^m \to \o{1}\x\cdots\x\O_{k_i}^{m_i}\x\cdots \x\o{\ell}$ such that $f = (\id \x \cdots \x\jmath_i\x \cdots \x \id) \circ \tilde{f}$, where $\jmath_i \colon \O_{k_i}^{m_i} \hookrightarrow \o{i}$ is the totally geodesic inclusion and each $\id \colon \o{j} \to \o{j}$ is the identity map.

	Let $L \subset T_f^\perp M$ be the vector sub-bundle whose fiber in $x$ is given by
	\begin{align*}
		& L(x) := \left(( \id \x \cdots \x\jmath_i \x  \cdots \x\id )_*  T_{\tilde{f}(x)}\o{1} \x\cdots\x \O_{k_i}^{m_i}\x \cdots\x\o{\ell}\right)^\perp = \\
		&= \{(0, \cdots, 0)\} \x \left( T_{\tilde{f}_i(x)}\O_{k_i}^{m_i} \right)^\perp \x \{(0,\cdots, 0)\}.
	\end{align*}
	Obviously, $\dim L = n_i - m_i = \bar n$ and $\pi_i(L) = L$.

	From \eqref{eq: tensores composta}, we know that $\BS_i^f = (\id \x \cdots \x \jmath_i \x \cdots \x \id)_*\BS_i^{\tilde f}$, thus
	\begin{align*}
		& \BS_i(TM) \subset (\id\x\cdots\x \jmath_i \x\cdots\x\id)_*T_{\tilde{f}} \o{1}\x\cdots\x \O_{k_i}^{m_i}\x\cdots\x\o{\ell} \Rightarrow \\
		& \Rightarrow L \subset \BS_i(TM)^\perp \ \Rightarrow L \subset \pi_i \left[ \BS_i(TM)^\perp \right].
	\end{align*}

	On the oder side, by Lemma \ref{lem: reducao}, $\BS_i(TM)^\perp = \U_i\op\V_i$ and $\V_i = \pi_i(\V_i) = \pi_i\left[ \BS_i(TM)^\perp \right]$. Hence $L \subset \V_i$.

	Because $(\id\x\cdots\x\jmath_i\x\cdots\x\id) \colon \o{1}\x\cdots\x\O_{k_i}^{m_i}\x\cdots\x\o{\ell} \to \hO$ is totally geodesic, $\mathcal{N}_1^f = (\id\x\cdots\x\jmath_i\x\cdots\x \id)_* \mathcal{N}_1^{\tilde{f}}$. Hence
	\begin{multline*}
		L(x) = \left[ (\id\x\cdots\x\jmath_i\x\cdots\x\id)_* T_{\tilde f(x)}\o{1}\x\cdots\x\O_{k_i}^{m_i}\x\cdots\x\o{\ell} \right]^\perp \subset \\
		\subset \left[\left(\id \x \cdots \x \jmath_i \x \cdots\x \id\right)_* \mathcal{N}_1^{\tilde f}(x) \right]^\perp = {\mathcal{N}_1^f(x)}^\perp.
	\end{multline*}
	Therefore $L \subset \V_i \cap {\mathcal{N}^f_1}^\perp$.

	Now, let $\xi \in \G(L)$ and
	\begin{multline*}
		\z = (\id \x\cdots\x\jmath_i\x\cdots\x\id)_* \tilde \z \in \\
		\in \G \left( (\id\x\cdots\x\jmath_i \x\cdots\x \id)_*T_{\tilde{f}} \o{1}\x\cdots\x\O_{k_i}^{m_i}\x\cdots\x\o{\ell} \right).
	\end{multline*}
	Thus
	\begin{multline*}
		\interno{\nperp_X \xi}{\z} = - \interno{\xi}{\nbar_X (\id \x \cdots \x\jmath_i\x\cdots\x \id)_* \tilde{\z}} = \\
		= -\interno{\xi}{(\id\x\cdots\x\jmath_i\x\cdots\x \id)_*\nbar_X \tilde{\z}} = 0.
	\end{multline*}
	Therefore $\nperp L \subset L$. \cqd

	\noindent\textbf{\textsl{(II)} $\Rightarrow$ \textsl{(I)}:} Let $\imath \colon \hO \hookrightarrow \RN$ be the canonical inclusion, $F := \imath \circ f$ and $\tilde{L} := \imath_*L$.

	\begin{afi}{$\tilde{L}$ is a constant subspace of $\RN$.}
		Let $\xi \in\G \left( L \right)$. Because $L \subset \V_i \cap \mathcal{N}_1^\perp$, then $A_\xi = 0$ and $\pi_i (\xi) = \xi$. Thus
		\begin{align*}
			& \ntil_X \imath_* \xi = \imath_* \nbar_X \xi + \ai{f_*X}{\xi} \stackrel{\text{Lemma \ref{lem: ai}}}{=} \\
			& = - F_* \cancel{A_\xi X} + \imath_* \nperp_X \xi + \sum_{j=1}^\ell\cancel{\interno{f_*X}{\pi_j \xi}}\nu_j = \imath_* \nperp_X \xi.
		\end{align*}

		But $L$ is parallel in the normal connexion of $f$, so $\ntil_X \imath_* \xi \in \tilde{L}$.
	\end{afi}

	\begin{afi}{$\pi_i \tilde{L} = \tilde{L}$.}
		
		Because $\pi_i(L) = L$ and $\pi_i \imath_* = \imath_* \pi_i$, the claim is true.
	\end{afi}

	Lets consider the projection $P_i \colon \R^{N_1}\x\cdots\x\R^{N_\ell} \to \R^{N_i}$, hence $\pi_i(x) = (0 \cdots, P_i(x), \cdots, 0)$. Lets also define $\bar L_i := P_i\left(\tilde L^\perp\right) = \left[P_i\left(\tilde L \right)\right]^\perp$ and $F_i := P_i \circ f$.

	\begin{afi}{$f(M) \subset \R^{N_1} \x \cdots \x \left( \bar L_i + F_i(x_0)\right) \x \cdots \x \R^{N_\ell}$, where $x_0 \in M$ is any fixed point.}
		It is enough to prove that $F_i(M) \subset \bar L_i + F_i(x_0)$. But $f_* X \perp \tilde L$, then, if $\xi \in \tilde L$ is a fixed vector, we have that
		\[X\interno{F_i}{P_i \xi} = \interno{P_i f_*X}{P_i \xi} = \interno{\pi_if_* X}{\pi_i \xi} = \interno{f_* X}{\pi_i \xi} = \interno{f_*X}{\xi} = 0.\]
		Thus $\interno{F_i(x)}{P_i \xi}$ is constant in $M$.
		
		Now, if $\{\xi_1, \cdots, \xi_{\bar n}\}$ is an orthogonal base of $\tilde L$, then $\{P_i\xi_1, \cdots, P_i\xi_{\bar n}\}$ is an orthogonal base of $P_i \left(\tilde L\right)$. Thus, $\interno{F_i(x)}{P_i \xi_j} = \interno{F_i(x_0)}{P_i \xi_j}$, where $x_0$ is a fixed point. Hence $F_i(x) - F_i(x_0) \perp P_i\left(\tilde L\right)$. Therefore $F_i(x) - F_i(x_0) \in \left[P_i\left(\tilde L\right)\right]^\perp = \bar L_i$.
	\end{afi}
	
	Now we have two cases to consider: the case $k_i \ne 0$ and the case $k_i =0$. If $k_i = 0$, then $\R^{N_i} = \E^{n_i}$ and $F_i(M) \subset \bar{L}_i + F_i(x_0)$. Hence, if we identify $\bar{L}_i + F_i(x_0)= \E^{m_i}$, then $F_i(M) \subset \E^{m_i}$, with $m_i = n_i - \bar n$.

	If $k_i\ne 0$, since $\nu_i = k_i (\pi_i \circ F) \perp \tilde L$, then $F_i(x_0) \perp P_i\left( \tilde L(x_0) \right) = P_i\left( \tilde L\right)$, hence $F_i(x_0) \in \bar{L}_i$. Thus $F_i(M) \subset \bar{L}_i \cap \o{i}$.
	
	If $k_i > 0$, then $\o{i} = \ss{i}$ and $\bar{L}_i \cap \o{i} = \S_{k_i}^{m_i} = \O_{k_i}^{m_i}$, where $m_i = n_i - \bar n$. If $k_i < 0$, then $F_i(x_0)$ is a timelike and $\bar L_i$ is also timelike, cause $F_i(x_0) \in \bar L_i$. Thus $\bar L_i \cap \o{i} = \bar L_i \cap \h{i} = \H_{k_i}^{m_i} = \O_{k_i}^{m_i}$, where $m_i = n_i - \bar n$.
	
	Our conclusion is: $f(M) \subset \o{1} \x \cdots \x \O_{k_i}^{m_i}\x \cdots \x \o{n} \subset \hO$.
\end{prova}

Here we want to remark that, if the codimension at the $i$th coordinate of $f \colon M \to \hO$ can be reduced by $\bar n_i$ and the codimension at the $j$th coordinate can be reduced by $\bar n_j$, then
\[f(M) \subset \o{1} \x \cdots \x \O_{k_i}^{n_i-\bar n_i}\x \cdots \x \O_{k_j}^{n_j-\bar n_j} \x \cdots \x \o{n},\]
for some totally geodesic submanifolds $\O_{k_i}^{n_i-\bar n_i} \subset \o{i}$ and $\O_{k_j}^{n_j-\bar n_j} \subset \o{j}$.

\begin{cor}\label{cor: reducao}
	Let $f\colon M^m\to \hO$ be an isometric immersion and lets suppose that $\V_i \cap \mathcal{N}_1^\perp$ be a vector sub-bundle of $T^\perp M$ with dimension $\bar n$. If $\nperp \left(\V_i \cap \mathcal{N}_1^\perp \right) \subset \mathcal{N}_1^\perp$, then the codimension of $f$ is reduced by $\bar n$ at the $i$th coordinate.
\end{cor}

\begin{proof}
	By Theorem \ref{teo: reducao}, it is enough to show that $L := \V_i \cap \mathcal{N}_1^\perp$ is parallel in the normal connexion of $f$. On the other side, we know by \textsl{(II)} of Lemma \ref{lem: reducao} that $\nperp L \subset \V_i$. But $\nperp L \subset \mathcal{N}_1^\perp$, therefore $L$ is parallel.
\end{proof}

\begin{teo} \label{teorema: reducao de codimension 2}
	Let $f\colon M^m \to \hO$ be an isometric immersion and lets suppose that $\V_i \cap \mathcal{N}_1^\perp$ be a vector sub-bundle of $T_f^\perp M$. Thus $\nperp \left(\V_i\cap \mathcal{N}_1^\perp \right) \subset \mathcal{N}_1^\perp$ if, and only if,
	\begin{enum}
		\item $\left(\nperp_Z \mathcal{R}^\perp \right)(X,Y,\xi) = 0$, for all $\xi \in\V_i \cap \mathcal{N}_1^\perp$,  and
		\item $\nperp \left(\V_i \cap \mathcal{N}_1^\perp\right) \subset \{\eta\}^\perp$, where $\eta$ is the mean curvature vector of $f$.
	\end{enum}
\end{teo}

\begin{prova}
	\par \noindent ($\Rightarrow$): Since $\mathcal{N}_1^\perp \subset \{\eta\}^\perp$ and $\nperp \left(\V_i \cap\mathcal{N}_1^\perp \right) \subset \mathcal{N}_1^\perp$, then $\nperp \left(\V_i \cap \mathcal{N}_1^\perp\right) \subset \{\eta\}^\perp$. We still have to show item \textsl{(I)}.

	By item \textsl{(I)} of Lemma \ref{lem: reducao}, $\BS_i(TM)^\perp = \U_i \op \V_i$. So, if $\xi \in \V_i \cap \mathcal{N}_1^\perp$ and $X, Y \in TM$, then, by Ricci equation \eqref{eq: Ricci},
	\begin{equation}\label{eq: R perp}
		\mathcal{R}^\perp(X,Y)\xi = \al{X}{A_\xi Y} - \al{A_\xi X}{Y} + \sum_{i=1}^\ell k_i \left( \BS_i X \wedge \BS_i Y \right) \xi = 0.
	\end{equation}
	Now, given $\xi \in \G\left( \V_i \cap \mathcal{N}_1^\perp \right)$, and because $\nperp_Z \xi \in \V_i \cap \mathcal{N}_1^\perp$, it holds that:
	\begin{multline*}
		\left(\nperp_Z \mathcal{R}^\perp \right)(X,Y,\xi) = \nperp_Z \mathcal{R}^\perp(X,Y) \xi - \mathcal{R}^\perp(\n_Z X, Y) \xi - \\
		- \mathcal{R}^\perp(X, \n_Z Y) \xi - \mathcal{R}^\perp(X, Y) \nperp_Z \xi = 0. \ \text{\cqd}
	\end{multline*}

	\noindent ($\Leftarrow$): Lets suppose now that we have items \textsl{(I)} and \textsl{(II)}. We have to show that $\nperp \left( \V_i \cap \mathcal{N}_1^\perp \right) \subset \mathcal{N}_1^\perp$.

	By the same arguments used above, equation \eqref{eq: R perp} holds. So, let $\xi \in \G \left(\V_i \cap \mathcal{N}_1^\perp\right) \subset \G \left(\{\eta\}^\perp\right)$, then
	\begin{align*}
		& 0 = \left( \nperp_Z \mathcal{R}^\perp \right)(X,Y,\xi) = \cancel{\nperp_Z \mathcal{R}^\perp (X,Y)\xi} - \cancel{\mathcal{R}^\perp\left(\n_Z X,Y\right) \xi} - \\
		& \quad - \cancel{\mathcal{R}^\perp\left(X, \n_Z Y\right)\xi} - \mathcal{R}^\perp(X,Y)\nperp_Z\xi = - \mathcal{R}^\perp(X,Y)\nperp_Z\xi.
	\end{align*}
	
	On the other side, as a result of Lemma \ref{lem: reducao}, $\nperp_X \xi \in \V_i \subset \BS_i(TM)^\perp$. As a consequence, and by Ricci's equation \eqref{eq: Ricci},
	\begin{multline*}
	0 = \mathcal{R}^\perp(X,Y)\nperp_Z\xi = \al{X}{A_{\nperp_Z \xi} Y} - \al{A_{\nperp_Z \xi}X}{Y} + \\
	+ \sum_{i=1}^\ell k_i \cancel{\left( \BS_i X \wedge \BS_i Y \right) \nperp_Z} \xi.
	\end{multline*}
	Hence, if $\z \in T^\perp M$, then
	\[0 = \interno{\al{X}{A_{\nperp_Z \xi} Y}}{\z} - \interno{\al{A_{\nperp_Z \xi}X}{Y}}{\z} = \interno{\left[A_{\nperp_Z \xi}, A_\z \right]X}{Y}.\]
	
	Therefore $\left[A_{\nperp_Z \xi}, A_{\nperp_W \xi} \right] = 0$, for all $W, Z \in \G(TM)$. Hence for each $x \in M$ there is an orthonormal base $\left\{E_1(x), \cdots, E_m(x)\right\}$ of $T_x M$ that diagonalizes all elements of the set $\set{A_{\nperp_X \xi}}{X \in T_x M}$.

	Since $\left[A_{\nperp_Z \xi}, A_{\nperp_W \xi} \right] = 0$, for all $W, Z \in \G(TM)$, then, for each $x \in M$ there is an orthonormal base $\left\{E_1(x), \cdots, E_m(x)\right\}$ of $T_x M$ that diagonalizes all elements of the set $\set{A_{\nperp_X \xi}}{X \in T_x M}$.

	Lets show that $\nperp_X \xi \in \V_i \cap \mathcal{N}_1^\perp$. For this, let $\ell_{k,i}$ be the eigenvalue of $A_{\nperp_{E_k} \xi}$ associated with the eigenvector $E_i(x)$, for each pair $i,k \in \{1, \cdots, m\}$. Thus
	\[\interno{\al{E_i}{E_j}}{\nperp_{E_k} \xi} = \interno{A_{\nperp_{E_k}\xi}E_i}{E_j} = \ell_{k,i}\interno{E_i}{E_j} = \ell_{k,i}\delta_{ij}.\]

	Since $\xi \in \Gamma \left(\V_i \cap \mathcal{N}_1^\perp \right)$ and $\V_i \subset \BS_i(TM)^\perp$, it follows from the second equation in \eqref{eq: Codazzi} and that
	\begin{align*}
		& 0 = (\n_Y A)(X,\xi) - (\n_X A)(Y, \xi) = \\
		& = \n_Y \cancel{A_\xi X} - \cancel{A_\xi \n_Y X} - A_{\nperp_Y \xi} X - \n_X \cancel{A_\xi Y} + \cancel{A_\xi \n_X Y} + A_{\nperp_X \xi} Y \Rightarrow \\
		& \Rightarrow A_{\nperp_X \xi} Y = A_{\nperp_Y \xi} X, \ \text{for any $X, Y \in \G(TM)$.}
	\end{align*}
	So, $A_{\nperp_{E_i} \xi} E_j = A_{\nperp_{E_j} \xi} E_i \Rightarrow \ell_{i,j} E_j = \ell_{j,i} E_i \Rightarrow \ell_{i,j} = 0$, if $i \ne j$. Consequently
	\[\interno{\al{E_i}{E_j}}{\nperp_{E_k} \xi} =
	\begin{cases}
		0, &\text{if $i \ne j$ or $i \ne k$,} \\
		\ell_{i,i}, &\text{if $i=j=k$}.
	\end{cases}\]

	On the other side, given that $\nperp \left(\V_i \cap \mathcal{N}_1^\perp\right) \subset \{\eta\}^\perp$, $\interno{\al{E_i}{E_i}}{\nperp_{E_i}\xi} = \sum\limits_{j=1}^{m} \interno{\al{E_j}{E_j}}{\nperp_{E_i}\xi} = m \interno{\eta}{\nperp_{E_i}\xi} = 0$. Therefore $\nperp_{E_i} \xi \in \mathcal{N}_1^\perp$.
\end{prova}

	\section{A Bonnet theorem for isometric immersions in $\XO$}
		\begin{teo}\label{Bonnet}
	Let $\tau_i : = \tau(k_i)$, $\rho := \tau_1 + \cdots + \tau_\ell$ and $m' := \sum\limits_{i=1}^\ell n_i -m$.

	\begin{enum}
		\item \textbf{Existence:} Let $M^m$ be a connected and simply connected riemannian manifold and let $\cE$ be a vector bundle with dimension $m'$ and index $\rho$ on $M$ with compatible connection $\n^\cE$, curvature tensor $\mathcal{R}^\cE$ and symmetric tensor $\alpha^\cE \colon TM \oplus TM \to \cE$. Lets also consider, for each $i \in \{1, \cdots, \ell\}$, the tensors $\BR_i \colon TM \to TM$, $\BS_i \colon TM \to \cE$ and $T_i \colon \cE \to \cE$, with $\BR_i$ and $\BT_i$ symmetric. For each $\eta\in \cE$, let $A_\eta^\cE \colon TM \to TM$ be given by $\interno{A_\eta^\cE X}{Y} = \interno{\alpha^\cE(X,Y)}{\eta}$. With these assumptions, if equations \eqref{somas} until \eqref{eq: Ricci} hold, then there is an isometric immersion $f \colon M \to \O$ and an vector bundle isometry $\Phi \colon \cE \to T^\perp M$ such that
		\begin{align*}
			& \alpha^f = \Phi \alpha^\cE,  && \nperp \Phi = \Phi \n^\cE, \\
			& \pi_i \circ f_* = f_* \BR_i + \Phi \BS_i, && \left.\pi_i\right|_{T^\perp M} = f_* \BS_i^\t \Phi^{-1} + \Phi \BT_i \Phi^{-1}.
		\end{align*}
		
		\item \textbf{Uniqueness:} Let $f, g \colon M \to \O$ be isometric immersions such that $\BR_i^f= \BR_i^g$, for every $i \in \{1, \cdots, \ell\}$. Lets suppose that there is a vector bundle isometry $\Phi \colon T_f^\perp M \to T_g^\perp M$ such that
		\begin{align*}
			& \Phi \alpha_f = \alpha_g, && \Phi {^f\nperp} = {^g\nperp} \Phi, \\
			& \Phi \BS_i^f = \BS_i^g, && \Phi\BT_i^f = \BT_i^g\Phi, \ \forall i \in\{1, \cdots, \ell\}.
		\end{align*}
		With the above conditions, there is an isometry $\varphi \colon \hO \to \hO$ such that $\varphi \circ f = g$ and $\left.\varphi_*\right|_{T_f^\perp M} = \Phi$.		
	\end{enum}
\end{teo}

\subsection{Existence}

Let $\bar \ell$ be the number of elements of the set $J = \set{i \in \{1, \cdots \ell\}}{k_i \ne 0}$ and let $\cF = \cE \oplus \Upsilon$ be the Whitney sum, where $\Upsilon$ is a semi-riemannian vector bundle on $M$, with dimension $\bar\ell$ and index $\rho$. We can choose an (local) orthogonal frame $\set{\nu_i}{ i \in J}$, of $\Upsilon$ with $\|\nu_i\|^2 = k_i$. If $k_i = 0$, that is, $i \notin J$, let $\nu_i = 0 \in \Upsilon$.

Now, we can define  a compatible connection in $\cF$ and a symmetric section $\alpha^{\mathcal F} \in \G \left( T^*M \otimes T^*M \otimes \cF \right)$ by
\begin{equation}\label{eq: F}
	\begin{aligned}
		& \n^{\mathcal F}_X \nu_i = -k_i \BS_i X, \quad \n^{\mathcal F}_X \xi = \n_X^{\mathcal E} \xi + \sum_{i=1}^\ell \interno{\BS_i X}{\xi} \nu_i \quad \text{and} \\
		& \alpha^{\mathcal F}(X,Y) = \alpha^{\mathcal E}(X,Y) + \sum_{i=1}^\ell \interno{\BR_i X}{Y} \nu_i,
	\end{aligned}
\end{equation}
for all $X,Y \in \G(TM)$ and all $\xi \in \G(\cE)$.

Now lets define, for each $\eta \in \G(\Upsilon)$, the shape operator $A^{\mathcal F}_\eta \colon TM \to TM$ by $\interno{A^{\mathcal F}_\eta X}{Y} = \interno{\alpha^{\mathcal F}(X,Y)}{\eta}$. So, if $\xi \in \G(\mathcal E)$, then
\begin{multline*}
	\interno{A^{\mathcal F}_\xi X}{Y} = \interno{\alpha^{\mathcal F}(X,Y)}{\xi} = \interno{\alpha^{\mathcal E}(X,Y) + \sum_{i=1}^\ell \interno{\BR_i X}{Y} \nu_i}{\xi} = \\
	= \interno{\alpha^{\mathcal E}(X,Y)}{\xi} = \interno{A^{\mathcal E}_\xi X}{Y}.
\end{multline*}

On the other side,
\begin{multline*}
	\interno{A^{\mathcal F}_{\nu_i} X}{Y} = \interno{\alpha^{\mathcal F}(X,Y)}{\nu_i} = \interno{\alpha^{\mathcal E}(X,Y) + \sum_{j=1}^\ell \interno{\BR_j X}{Y} \nu_j }{\nu_i} = \\
	=  \interno{\BR_i X}{Y} \|\nu_i\|^2 = k_i \interno{\BR_i X}{Y}.
\end{multline*}
Therefore
\begin{equation}\label{AF}
	A^{\mathcal F}_\xi = A^{\mathcal E}_\xi \quad \text{and} \quad A^{\mathcal F}_{\nu_i} = k_i \BR_i.
\end{equation}

\begin{afi}{$\mathcal{R}(X,Y)Z = A^{\mathcal F}_{\alpha^{\mathcal F}(Y,Z)}X - A^{\mathcal F}_{\alpha^{\mathcal F}(X,Z)}Y$, for any $X, Y, Z \in \G(TM)$.}
	\begin{align*}
		& \mathcal{R}(X,Y)Z \stackrel{\eqref{eq: Gauss}}{=} \sum_{i=1}^\ell k_i \left( \BR_i X \wedge \BR_i Y \right)Z + A^{\mathcal E}_{\alpha^{\mathcal E}(Y,Z)} X - A^{\mathcal E}_{\alpha^{\mathcal E}(X,Z)} Y =\\
%		& = \sum_{i=1}^\ell k_i (\interno{\BR_i Y}{Z}\BR_i X - \interno{\BR_i X}{Z}\BR_i Y) + A^{\mathcal F}_{\alpha^{\mathcal E}(Y,Z)} X - A^{\mathcal F}_{\alpha^{\mathcal E}(X,Z)} = \\
		& \stackrel{\eqref{eq: F}, \eqref{AF}}{=} \cancel{\sum_{i=1}^\ell k_i \interno{\BR_i Y}{Z}\BR_i X } - \bcancel{ \sum_{i=1}^\ell k_i\interno{\BR_i X}{Z}\BR_i Y } + A^{\mathcal F}_{\alpha^{\mathcal F}(Y,Z)} X - \\
		& \phantom{\stackrel{\eqref{eq: F}, \eqref{AF}}{=}} - \cancel{ \sum_{i=1}^\ell \interno{\BR_i Y}{Z} k_i \BR_i X } - A^{\mathcal F}_{\alpha^{\mathcal F}(X,Z)} Y + \bcancel{\sum_{i=1}^\ell \interno{\BR_i X}{Z} k_i \BR_i Y} = \\
		& = A^{\mathcal F}_{\alpha^{\mathcal F}(Y,Z)} X - A^{\mathcal F}_{\alpha^{\mathcal F}(X,Z)} Y.
	\end{align*}
	Therefore the claim holds.
\end{afi}

\begin{afi}{$\left(\n_X \alpha^{\mathcal F}\right)(Y,Z) - \left(\n_Y \alpha^{\mathcal F}\right)(X,Z) = 0$, for any $X,Y,Z \in \G(TM)$.}
	\begin{align*}
		& \left(\n_X \alpha^{\mathcal F}\right)(Y,Z) = \n^{\mathcal F}_X \alpha^{\mathcal F}(Y,Z) - \alpha^{\mathcal F}\left(\n_X Y,Z\right) - \alpha^{\mathcal F}\left(Y,\n_X Z\right) = \\
		& \stackrel{\eqref{eq: F}}{=} \n^{\mathcal F}_X \left[ \alpha^{\mathcal E}(Y,Z) + \sum_{i=1}^\ell \interno{\BR_i Y}{Z} \nu_i \right] - \alpha^{\mathcal E}\left(\n_X Y,Z\right) - \sum_{i=1}^\ell \interno{\BR_i \n_X Y}{Z} \nu_i - \\
			& \phantom{\stackrel{\eqref{eq: F}}{=}} - \alpha^{\mathcal E}\left(Y,\n_X Z\right)  - \sum_{i=1}^\ell \interno{\BR_i Y}{\n_X Z}\nu_i =\\
		& = \n^{\mathcal F}_X \alpha^{\mathcal E}(Y,Z) + \sum_{i=1}^\ell \left[ (\interno{\n_X \BR_i Y}{Z} + \cancel{\interno{\BR_i Y}{\n_X Z}}) \nu_i + \interno{\BR_i Y}{Z} \n^{\mathcal F}_X \nu_i \right] - \\
			& \quad - \alpha^{\mathcal E}\left(\n_X Y,Z\right) - \sum_{i=1}^\ell \interno{\BR_i \n_X Y}{Z} \nu_i - \alpha^{\mathcal E}\left(Y,\n_X Z\right) - \sum_{i=1}^\ell \cancel{\interno{\BR_i Y}{\n_X Z}}\nu_i = \\
		& \stackrel{\eqref{eq: F}}{=} \n^{\mathcal E}_X \alpha^{\mathcal E}(Y,Z) + \sum_{i=1}^\ell \interno{\BS_i X}{\alpha^{\mathcal E}(Y,Z)} \nu_i + \sum_{i=1}^\ell \interno{\left(\n_X \BR_i\right) Y}{Z} \nu_i - \\
			& \quad -\sum_{i=1}^\ell \interno{\BR_i Y}{Z} k_i \BS_i X - \alpha^{\mathcal E}\left(\n_X Y,Z\right) - \alpha^{\mathcal E}\left(Y,\n_X Z\right).
	\end{align*}

	Hence
	\begin{multline}\label{eq: *}
		\left(\n^{\mathcal F}_X \alpha^{\mathcal F}\right)(Y,Z) = \left( \n_X \alpha^{\mathcal E} \right)(Y,Z) + \\
		+ \sum_{i=1}^\ell \left(\interno{\BS_i X}{\alpha^{\mathcal E}(Y,Z)} + \interno{\left(\n_X \BR_i\right) Y}{Z} \right)\nu_i - \sum_{i=1}^\ell \interno{\BR_i Y}{Z} k_i \BS_i X.
	\end{multline}
	
	Thus
	\begin{align*}
		& \left(\n^{\mathcal F}_Y \alpha^{\mathcal F}\right)(X,Z) \stackrel{\eqref{eq: *}}{=} \left( \n_Y \alpha^{\mathcal E} \right)(X,Z) + \\
			& \quad + \sum_{i=1}^\ell \left(\interno{\BS_i Y}{\alpha^{\mathcal E}(X,Z)} + \interno{\left(\n_Y \BR_i\right) X}{Z} \right)\nu_i -\sum_{i=1}^\ell \interno{\BR_i X}{Z} k_i \BS_i Y = \\
		& \stackrel{\eqref{eq: Codazzi}}{=} \left( \n_X \alpha^{\mathcal E} \right)(Y,Z) + \sum_{i=1}^\ell k_i\left[ \cancel{\interno{\BR_i X}{Z}\BS_i Y} - \interno {\BR_i Y}{Z} \BS_i X\right] + \\
			& \quad + \sum_{i=1}^\ell \left(\interno{\BS_i Y}{\alpha^{\mathcal E}(X,Z)} + \interno{\left(\n_Y \BR_i\right) X}{Z} \right)\nu_i - \sum_{i=1}^\ell \cancel{\interno{\BR_i X}{Z} k_i \BS_i Y} = \\
%		& = \left( \n_X \alpha^{\mathcal E} \right)(Y,Z) + \sum_{i=1}^\ell \left(\interno{\BS_i Y}{\alpha^{\mathcal E}(X,Z)} + \interno{\left(\n_Y \BR_i\right) X}{Z} \right)\nu_i - \sum_{i=1}^\ell k_i \interno {\BR_i Y}{Z} \BS_i X = \\
		& \stackrel{\eqref{eq: derivada R}}{=} \left( \n_X \alpha^{\mathcal E} \right)(Y,Z) + \sum_{i=1}^\ell \left(\interno{\BS_i Y}{\alpha^{\mathcal E}(X,Z)} + \interno{A^{\mathcal E}_{\BS_i X} Y + \BS_i^\t \alpha^{\mathcal E}(Y,X)}{Z} \right)\nu_i - \\
			& \quad - \sum_{i=1}^\ell k_i \interno {\BR_i Y}{Z} \BS_i X = \\
		& = \left( \n_X \alpha^{\mathcal E} \right)(Y,Z) + \sum_{i=1}^\ell \interno{A^{\mathcal E}_{\BS_i Y}X + A^{\mathcal E}_{\BS_i X} Y + \BS_i^\t \alpha^{\mathcal E}(X,Y)}{Z} \nu_i - \\
			& \quad - \sum_{i=1}^\ell k_i \interno {\BR_i Y}{Z} \BS_i X = \\
		& \stackrel{\eqref{eq: derivada R}}{=} \left( \n_X \alpha^{\mathcal E} \right)(Y,Z) + \sum_{i=1}^\ell \left( \interno{ (\n_X \BR_i) Y}{Z} + \interno{\BS_i X}{\alpha^{\mathcal E}(Y,Z)} \right)\nu_i - \\
			& \quad - \sum_{i=1}^\ell k_i \interno {\BR_i Y}{Z} \BS_i X \stackrel{\eqref{eq: *}}{=} \left(\n^{\mathcal F}_X \alpha^{\mathcal F}\right)(Y,Z).
	\end{align*}
	Therefore the claim holds.
\end{afi}

\begin{afi}{${\mathcal{R}^\cF}(X,Y) \eta = \alpha^{\mathcal F}\left(X, A^{\mathcal F}_\eta Y\right) - \alpha^{\mathcal F}\left(A^{\mathcal F}_\eta X, Y \right)$, for all $X,Y \in \G(TM)$ and all $\eta\in \G(\mathcal F)$.}
	It is enough to prove the claim in the cases $\eta \in \G(\mathcal E)$ and $\eta = \nu_i$, for some $i \in \{1, \cdots, \ell\}$. So, let $\xi \in \G( \mathcal E)$, hence
	\begin{align*}
		& {\cal R^F}(X,Y) \xi = \n^{\mathcal F}_X \n^{\mathcal F}_Y \xi - \n^{\mathcal F}_Y \n^{\mathcal F}_X \xi - \n^{\mathcal F}_{[X,Y]} \xi = \\
		& \stackrel{\eqref{eq: F}}{=} \n^{\mathcal F}_X \left( \n^{\mathcal E}_Y \xi + \sum_{i=1}^\ell \interno{\BS_i Y}{\xi} \nu_i\right) - \n^{\mathcal F}_Y \left( \n^{\mathcal E}_X \xi + \sum_{i=1}^\ell \interno{\BS_i X}{\xi} \nu_i\right) - \\
			& \quad - \n^{\mathcal E}_{[X,Y]} \xi - \sum_{i=1}^\ell \interno{\BS_i [X,Y]}{\xi} \nu_i = \\
		& \stackrel{\eqref{eq: F}}{=} {\cal R^E}(X,Y) \xi + \sum_{i=1}^\ell \interno{\BS_i X}{\n^{\mathcal E}_Y \xi} \nu_i - \sum_{i=1}^\ell \interno{\BS_i Y}{\n^{\mathcal E}_X \xi} \nu_i - \sum_{i=1}^\ell \interno{\BS_i [X,Y]}{\xi} \nu_i + \\
			& \quad + \sum_{i=1}^\ell (X\interno{\BS_i Y}{\xi} - Y\interno{\BS_i X}{\xi}) \nu_i - \sum_{i=1}^\ell \interno{\BS_i Y}{\xi} k_i \BS_i X + \sum_{i=1}^\ell \interno{\BS_i X}{\xi} k_i \BS_i Y = \\
		& = {\cal R^E}(X,Y) \xi + \sum_{i=1}^\ell \left( \cancel{\interno{\BS_i X}{\n^{\mathcal E}_Y \xi}} - \bcancel{\interno{\BS_i Y}{\n^{\mathcal E}_X \xi}} - \interno{\BS_i \n_X Y - \BS_i \n_Y X}{\xi}\right) \nu_i + \\
			& \quad + \sum_{i=1}^\ell \left(\interno{\n^{\mathcal E}_X \BS_i Y}{\xi} + \bcancel{\interno{\BS_i Y}{\n^{\mathcal E}_X \xi}} - \interno{\n^{\mathcal E}_Y \BS_i X}{\xi} - \cancel{\interno{\BS_i X}{\n^{\mathcal E}_Y \xi}} \right) \nu_i -\\
			& \quad -  \sum_{i=1}^\ell k_i ( \BS_i X \wedge \BS_i Y) \xi = \\
		& = {\cal R^E}(X,Y) \xi - \sum_{i=1}^\ell k_i ( \BS_i X \wedge \BS_i Y) \xi + \sum_{i=1}^\ell [ \interno{(\n_X \BS_i) Y}{\xi} - \interno{(\n_Y \BS_i) X}{\xi}]\nu_i = \\
		& \stackrel{\eqref{eq: Ricci}, \eqref{eq: derivada S}}{=} \alpha^{\mathcal E}\left(X, A^{\mathcal E}_\xi Y\right) - \alpha^{\mathcal E}\left(A^{\mathcal E}_\xi X, Y\right) + \\
		& + \sum_{i=1}^\ell \interno{\cancel{\BT_i \alpha^{\mathcal E}(X,Y)} - \alpha^{\mathcal E}(X,\BR_i Y) - \cancel{\BT_i \alpha^{\mathcal E}(Y,X)} + \alpha^{\mathcal E}(Y,\BR_i X)}{\xi} \nu_i = \\
		& = \alpha^{\mathcal E}\left(X, A^{\mathcal E}_\xi Y\right) - \alpha^{\mathcal E}\left(A^{\mathcal E}_\xi X, Y\right) - \sum_{i=1}^\ell \interno{A^{\mathcal E}_\xi X}{\BR_i Y} \nu_i + \sum_{i=1}^\ell \interno{A^{\mathcal E}_\xi Y}{\BR_i X}\nu_i = \\
		& \stackrel{\eqref{eq: F},\eqref{AF}}{=} \alpha^{\mathcal F}\left(X, A^{\mathcal E}_\xi Y\right) - \alpha^\cF\left(A^\cF_\xi X, Y\right).
	\end{align*}
	
	Now, for $i \in \{1, \cdots, \ell\}$,
	\begin{align*}
		& {\cal R^F}(X,Y) \nu_i = \n^{\mathcal F}_X \n^{\mathcal F}_Y \nu_i - \n^{\mathcal F}_Y \n^{\mathcal F}_X \nu_i - \n^{\mathcal F}_{[X,Y]} \nu_i = \\
		& \stackrel{\eqref{eq: F}}{=} - \n^{\mathcal F}_X k_i \BS_i Y + \n^{\mathcal F}_Y k_i \BS_i X + k_i \BS_i [X,Y] = \\
		& \stackrel{\eqref{eq: F}}{=} - k_i \n^{\mathcal E}_X \BS_i Y - k_i \sum_{j=1}^\ell \interno{\BS_j X}{\BS_i Y} \nu_j + k_i \n^{\mathcal E}_Y \BS_i X +  k_i \sum_{j=1}^\ell \interno{\BS_j Y}{\BS_i X} \nu_j + \\
			& \quad + k_i \BS_i \n_X Y - k_i \BS_j \n_Y X = \\
		& = -k_i \left( \n_X \BS_i \right)Y + k_i (\n_Y \BS_i) X - k_i \sum_{j=1}^\ell \left( \interno{\BS_i^\t\BS_j X}{Y} - \interno{\BS_i^\t\BS_j Y}{X} \right) \nu_j = \\
		& \stackrel{\eqref{eq: derivada S}, \eqref{eq: RST}, \eqref{eq: RST2}}{=} -k_i \left[ \cancel{\BT_i \alpha^{\mathcal E}(X,Y)} - \alpha^{\mathcal E}(X, \BR_i Y) \right] + k_i \left[ \cancel{\BT_i \alpha^{\mathcal E}(Y,X)} - \alpha^{\mathcal E}(Y, \BR_i X)\right] + \\
			& \quad + k_i\sum_{\substack{j=1\\ j \ne i}}^\ell \left( \interno{\BR_i\BR_j X}{Y} - \interno{\BR_i\BR_j Y}{X} \right)\nu_j - \interno{k_i\BR_i(\id - \BR_i)X}{Y} \nu_i + \\
			& \quad + k_i \interno{\BR_i(\id - \BR_i)Y}{X} \nu_i = \\
		& = \alpha^{\mathcal E}(X, k_i\BR_i Y) - \alpha^{\mathcal E}(Y, k_i\BR_i X) + \sum_{j=1}^\ell \left( \interno{k_i\BR_i\BR_j X}{Y} - \interno{k_i \BR_i\BR_j Y}{X} \right) \nu_j -   \\
			& \quad - \cancel{k_i\interno{\BR_i X}{Y} \nu_i} - \cancel{k_i \interno{\BR_i Y}{X} \nu_i} = \\
		& = \alpha^{\mathcal E}(X, k_i\BR_i Y) + \sum_{j=1}^\ell \interno{\BR_j X}{k_i\BR_i Y} \nu_j - \\
			& \quad - \alpha^{\mathcal E}(Y, k_i\BR_i X) - \sum_{j=1}^\ell \interno{\BR_j Y}{k_i\BR_i X} \nu_j = \\
		& \stackrel{\eqref{AF}, \eqref{eq: F}}{=} \alpha^{\mathcal F}\left(X, A_{\nu_i}^{\mathcal F} Y \right) - \alpha^{\mathcal F}\left(Y, A_{\nu_i}^{\mathcal F} X\right).
	\end{align*}
	Therefore the claim holds.
\end{afi}

Now, let $n = \sum\limits_{i=1}^\ell \left(n_i + \upsilon(k_i)\right)$ and $t = \sum\limits_{i=1}^\ell \tau(k_i)$. Since Claims 1, 2 and 3 hold, and because of Bonnet Theorem for isometric immersions in $\R_t^n$ (see \cite{LTV}), there is an isometric immersion $F \colon M \to \R_t^n$ and a vector bundle isometry $\tilde\Phi \colon \cF \to T_F^\perp M$ such that $\alpha_F = \tilde\Phi \alpha^{\mathcal F}$ and $\nbarperp \tilde \Phi = \tilde \Phi \n^{\mathcal F}$, where $\nbarperp$ is the normal connection in $T_F^\perp M$.

Let $\mathcal G := TM \oplus \mathcal F$ be the Whitney sum and lets define the following connection in $\cal G$:
\begin{align*}
	&\n^{\cal G}_X Y = \n_X Y + \alpha^{\mathcal F}(X,Y), && \forall X,Y \in \G(TM); \\
	&\n^{\cal G}_X \eta = - A^{\mathcal F}_\eta X + \n^{\mathcal F}_X \eta, && \forall X \in \G(TM), \ \text{and all} \ \eta \in \G \left( \cal F \right).
\end{align*}

Besides, for each $i \in \{1, \cdots \ell\}$, let $P_i \colon \mathcal{G} \to \mathcal{G}$ be given by
\begin{align}\label{Pi}
	\left.P_i\right|_{TM} &= \BR_i + \BS_i, & \left.P_i\right|_{\mathcal E} &= \BS_i^\t + \BT_i, & P_i (\nu_j) &= \delta_{ij} \nu_i.
\end{align}

\begin{afi}{$P_iP_j = \delta_{ij} P_i$.}
	Let $X \in TM$ and $\xi \in \cE$. If $i \ne j$, then
	\begin{align*}
		P_iP_j X &= P_i (\BR_j X + \BS_j X) = P_i \BR_j X + P_i \BS_j X = \\
			&= \cancel{\BR_i\BR_j X} + \bcancel{\BS_i\BR_j X} + \cancel{\BS_i^\t\BS_j X} + \bcancel{\BT_i\BS_j X} \stackrel{\eqref{eq: RST2}}{=} 0; \\
		P_i P_j \xi &= P_i(\BS_j^\t \xi + \BT_j \xi) = \cancel{\BR_i\BS_j^\t \xi} + \bcancel{\BS_i\BS_j^\t \xi} + \cancel{\BS_i^\t\BT_j \xi} + \bcancel{\BT_i\BT_j \xi} \stackrel{\eqref{eq: RST2}}{=} 0; \\
		P_i P_j \nu_k &= P_i \delta_{jk} \nu_j = \delta_{jk} \delta_{ij} \nu_i = 0.
	\end{align*}

	If $i = j$, then
	\begin{align*}
		{P_i}^2 X &= P_i (\BR_i X + \BS_i X) = \BR_i^2 X + \BS_i\BR_i X + \BS_i^\t\BS_i X + \BT_i\BS_i X = \\
			& \stackrel{\eqref{eq: RST}}{=} \BR_i^2 X + \BS_i \BR_i X + \BR_i(\id - \BR_i) X + \BS_i(\id - \BR_i) X = \\
			&=\BR_i X + \BS_i X = P_i X; \\
		{P_i}^2 \xi &= P_i(\BS_i^\t \xi + \BT_i \xi) = \BR_i\BS_i^\t \xi + \BS_i\BS_i^\t \xi + \BS_i^\t\BT_i \xi + \BT_i^2 \xi = \\
			& \stackrel{\eqref{eq: RST}}{=} \BR_i\BS_i^\t \xi + \BT_i(\id - \BT_i) \xi + (\id - \BR_i)\BS_i^\t \xi + \BT_i^2 \xi = \\
			&= \BS_i^\t \xi + \BT_i \xi = P_i \xi; \\
		{P_i}^2 \nu_j &= P_i \delta_{ij} \nu_i = \delta_{ij} \nu_i = P_i \nu_j.
	\end{align*}
	
	Therefore the claim holds.	
\end{afi}

\begin{afi}{Each $P_i$ is a parallel tensor, that is, $\n^{\cal G}P_i = P_i \n^{\cal G}$.}
	Let $X, Y \in \G(TM)$ and $\xi \in \G(\cF)$, thus
	\begin{align*}
		&\n^{\cal G}_X P_i Y = \n^{\cal G}_X (\BR_i Y + \BS_i Y) = \\
		& = \n_X \BR_i Y + \alpha^{\mathcal F}(X, \BR_i Y) - A_{\BS_i Y}^{\mathcal F} X + \n^{\mathcal F}_X \BS_i Y = \\
		& \stackrel{\eqref{eq: F}, \eqref{AF}}{=} \n_X \BR_i Y + \alpha^{\mathcal E}(X, \BR_i Y) + \sum_{j=1}^\ell \interno{\BR_j X}{\BR_i Y} \nu_j - A_{\BS_i Y}^{\mathcal E} X + \\
			& \quad +  \n^{\mathcal E}_X \BS_i Y + \sum_{j=1}^\ell \interno{\BS_j X}{\BS_i Y}\nu_j = \\
%		&= \n_X \BR_i Y + \alpha^{\mathcal E}(X, \BR_i Y) + \sum_{j=1}^\ell \interno{\BR_i\BR_j X}{Y} \nu_j - A_{\BS_i Y}^{\mathcal E} X + \n^{\mathcal E}_X \BS_i Y + \\
%			& \quad + \sum_{j=1}^\ell \interno{\BS_i^\t\BS_j X}{Y}\nu_j = \\
		&\stackrel{\eqref{eq: RST}, \eqref{eq: RST2}}{=} \n_X \BR_i Y + \alpha^{\mathcal E}(X, \BR_i Y) + \sum_{\substack{j=1\\j \ne i}}^\ell \cancel{\interno{\BR_i\BR_j X}{Y} \nu_j} + \interno{\BR_i^2 X}{Y} \nu_i - \\
			& \quad - A_{\BS_i Y}^{\mathcal E} X + \n^{\mathcal E}_X \BS_i Y - \sum_{\substack{j=1\\j \ne i}}^\ell \cancel{\interno{\BR_i\BR_j X}{Y}\nu_j} + \interno{\BR_i(\id - \BR_i) X}{Y} \nu_i = \\
		&= \n_X \BR_i Y + \alpha^{\mathcal E}(X, \BR_i Y) - A_{\BS_i Y}^{\mathcal E} X + \n^{\mathcal E}_X \BS_i Y + \interno{\BR_i X}{Y} \nu_i = \\
		& \stackrel{\eqref{eq: derivada R}, \eqref{eq: derivada S}}{=} \BR_i \n_X Y + \cancel{A^{\mathcal E}_{\BS_i Y} X} + \BS_i^\t\alpha^{\mathcal E}(X, Y) + \bcancel{\alpha^{\mathcal E}(X, \BR_i Y)} - \cancel{A_{\BS_i Y}^{\mathcal E} X} +  \\
			& \quad + \BS_i \n_X Y + \BT_i \alpha^{\mathcal E}(X,Y) - \bcancel{\alpha^{\mathcal E}(X, \BR_i Y)} + \interno{\BR_i X}{Y} \nu_i = \\
		& = \BR_i \n_X Y + \BS_i^\t\alpha^{\mathcal E}(X, Y) + \BS_i \n_X Y + \BT_i \alpha^{\mathcal E}(X,Y) + \interno{\BR_i X}{Y} \nu_i = \\
		&= P_i \left( \n_X Y + \alpha^{\mathcal E}(X,Y) + \sum_{j=1}^\ell \interno{\BR_j X}{Y} \nu_j \right) = P_i \n^{\cal G}_X Y.
	\end{align*}
		
	\begin{align*}
		&\n^{\cal G}_X P_i \xi = \n^{\cal G}_ X \left( \BS_i^\t \xi + \BT_i \xi \right) = \n_X \BS_i^\t \xi + \alpha^{\mathcal F}\left(X, \BS_i^\t \xi\right) - A^{\mathcal F}_{\BT_i \xi} X + \n^{\mathcal F}_X \BT_i\xi = \\
		&\stackrel{\eqref{eq: F},\eqref{AF}}{=} \n_X \BS_i^\t \xi + \alpha^{\mathcal E}\left(X, \BS_i^\t \xi \right) + \sum_{j=1}^\ell \interno{\BR_j X}{\BS_i^\t \xi}\nu_j - A^{\mathcal E}_{\BT_i \xi} X + \n^{\mathcal E}_X \BT_i\xi + \\
			&\quad + \sum_{j=1}^\ell \interno{\BS_j X}{\BT_i\xi}\nu_j = \\
		&= \n_X \BS_i^\t \xi + \alpha^{\mathcal E}\left(X, \BS_i^\t \xi \right) +  \sum_{\substack{j=1\\j \ne i}}^\ell \cancel{\interno{\BS_i\BR_j X}{\xi}\nu_j} + \interno{\BS_i\BR_i X}{\xi} \nu_i - A^{\mathcal E}_{\BT_i \xi} X + \\
			& \quad + \n^{\mathcal E}_X \BT_i\xi + \sum_{\substack{j=1\\j \ne i}}^\ell \cancel{\interno{\BT_i\BS_j X}{\xi}\nu_j} + \interno{\BT_i\BS_i X}{\xi} \nu_i = \\
		& \stackrel{\eqref{eq: RST}, \eqref{eq: RST2}}{=} \n_X \BS_i^\t \xi + \alpha^{\mathcal E}\left(X, \BS_i^\t \xi \right) - A^{\mathcal E}_{\BT_i \xi} X + \n^{\mathcal E}_X \BT_i\xi + \interno{\BS_i X}{\xi} \nu_i = \\
		& \stackrel{\eqref{eq: derivada S},\eqref{eq: derivada T}}{=} \BS_i^\t \n^{\mathcal E}_X \xi + \cancel{A^{\mathcal E}_{\BT_i \xi} X} - \BR_i A^{\mathcal E}_\xi X + \bcancel{\alpha^{\mathcal E}\left(X, \BS_i^\t \xi \right)} - \cancel{A^{\mathcal E}_{\BT_i \xi} X} + \BT_i \n^{\mathcal E}_X \xi -  \\
			& \quad - \BS_i A_\xi X - \bcancel{\alpha^{\mathcal E}\left(X,\BS_i^\t\xi\right)} + \interno{\BS_i X}{\xi} \nu_i = \\
		& = P_i \left(-A_\xi^{\mathcal E} X + \n_X^{\mathcal E} \xi + \sum_{j=1}^\ell \interno{\BS_j X}{\xi} \nu_j \right) = P_i \left( \n^{\cal G}_X \xi\right).
	\end{align*}
		
	\begin{align*}
		\n_X^{{\cal G}} P_i \nu_j & = \n_X^{{\cal G}} \delta_{ij} \nu_i = - \delta_{ij} A_{\nu_i}^{\mathcal F} X + \delta_{ij} \n^{\mathcal F}_X \nu_i =\\
			&\stackrel{\eqref{eq: F},\eqref{AF}}{=} -\delta_{ij}k_i \BR_i X - \delta_{ij} k_i \BS_i X = -\delta_{ij}k_i P_i(X). \\
		P_i \n_X^{{\cal G}} \nu_j &= P_i(-k_j A_{\nu_j}^{\mathcal F} X + \n^{\mathcal F}_X \nu_j) = P_i (-k_j \BR_j X - k_j \BS_j X) = \\
			& = P_i (-k_j P_j (X)) = -\delta_{ij}k_j P_i(X).
	\end{align*}
		
	Since $\n^{\cal G}_X P_i Y = P_i \n^{\cal G}_X Y$, $\n^{\cal G}_X P_i \xi = P_i \n^{\cal G}_X \xi$ and $\n_X^{{\cal G}} P_i \nu_j = P_i \n_X^{{\cal G}} \nu_j$, for all $X, Y \in \G(TM)$, all $\xi \in \G(\cE)$ and all $i,j \in \{1, \cdots, \ell\}$, then $P_i$ is a parallel tensor.
\end{afi}
	
\begin{afi}{$P_i = P_i^\t$.}
	It follows from straightforward calculations.
\end{afi}

Lets consider $\hat\Phi \colon \mathcal{G} \to F^*T \R_t^N$ given by $\hat\Phi|_{TM} = F_* \colon TM \to F_*TM$ and $\hat\Phi|_{\mathcal F} = \tilde \Phi \colon \cF \to T_F^\perp M$.

\begin{afi}{$\hat\Phi$ is parallel vector bundle isometry, that is, $\hat\Phi \n^{\cal G} = \ntil \hat\Phi$, where $\ntil$ is the connection in $\R_t^N$.} \label{Pi parallel}
	Let $X,Y \in \G(TM)$ and $\eta \in \Gamma(\cF)$. Thus,
	\begin{multline*}
		\interno{A^{\mathcal F}_\eta X}{Y} = \interno{\alpha^{\mathcal F}(X, Y)}{\eta} = \interno{\tilde\Phi \alpha^{\mathcal F}(X, Y)}{\tilde\Phi \eta} = \\
		= \interno{\alpha_F(X,Y)}{\tilde\Phi \eta} = \interno{A^F_{\tilde\Phi \eta} X}{Y}.
	\end{multline*}
	Hence $A^{\mathcal F}_\eta = A^F_{\tilde\Phi\eta}$ and
	\begin{align*}
		\hat\Phi \n^{\cal G}_X Y &= \hat\Phi\left( \n_X Y + \alpha^{\mathcal F}(X,Y) \right) = F_* \n_X Y + \tilde\Phi \left(\alpha^{\mathcal F}(X,Y) \right) = \\
		&= F_* \n_X Y + \alpha_F(X,Y) = \ntil_X F_* Y = \ntil_X \hat\Phi Y. \\
		\hat\Phi \n^{\cal G}_X \eta &= \hat\Phi \left( -A^{\mathcal F}_\eta X + \n^{\mathcal F}_X \eta \right) = -F_*A^{\mathcal F}_\eta X + \tilde\Phi \n^{\mathcal F}_X \eta = \\
		&= -F_*A^F_{\tilde\Phi\eta} X +  \nperp_X \tilde\Phi \eta = \ntil_X \tilde\Phi\eta = \ntil_X \hat\Phi\eta.
	\end{align*}
	
	Thus $\hat \Phi$ is parallel.
\end{afi}

Since each $P_i$ is parallel, $P_i(\mathcal{G})$ is a parallel vector sub-bundle of $\cal G$. Besides, $\interno{P_i X}{P_j Y} = \interno{P_jP_i X}{Y} = \delta_{ij} \interno{P_i X}{Y}$. Then $P_i(\mathcal{G})$ and $P_j(\mathcal{G})$ are orthogonal, if $i \ne j$. Hence each $\hat\Phi[P_i(\mathcal{G})]$ is a constant vector subspace of $\R_t^N$ and they are orthogonal. Besides, because $P_i(\nu_j) =  \delta_{ij} \nu_i$ and $\|\nu_i\|^2 = k_i$, then $\hat\Phi[P_i(\mathcal{G})] = \R^{N_i} = \R_{\tau(k_i)}^{N_i}$, where $N_i$ is the dimension of $P_i(\mathcal{G})$.

\begin{afi}{$\R_t^N = \R^{N_1} \op \cdots \op \R^{N_\ell}$}
	Since $\RN = F_* TM \op T_F^\perp M$ and $\hat\Phi \colon \mathcal{G} \to F^*T\R_t^N$ is a vector bundle isometry, we just have to show that $\mathcal{G} = P_1(\mathcal{G}) \op \cdots \op P_\ell(\mathcal{G})$. On the other side, we know that each pair $P_i(\mathcal{G})$ and $P_j(\mathcal{G})$ are orthogonal sub-fiber bundles. So we just have to show that $\z = P_1(\z) + \cdots + P_\ell(\z)$, for any $\z \in \mathcal{G}$.
	
	Let $X \in TM$ and $\xi \in \cal E$, thus
	\begin{align*}
		& X = \id(X) + 0(x) \stackrel{\eqref{somas}}{=} \sum_{i=1}^\ell \BR_i X + \sum_{i=1}^\ell \BS_i X = \sum_{i=1}^\ell P_i(X), \\
		& \xi = 0(\xi) + \id(\xi) \stackrel{\eqref{somas}}{=} \sum_{i=1}^\ell \BS_i^\t \xi + \sum_{i=1}^\ell \BT_i \xi = \sum_{i=1}^\ell P_i(\xi), \\
		& \nu_j \stackrel{\eqref{Pi}}{=} \sum_{i=1}^\ell P_i(\nu_j).
	\end{align*}
	
	Therefore $\z = P_1(\z) + \cdots + P_\ell(\z)$, for any $\z \in \mathcal{G}$.
\end{afi}

For each $i \in \{ 1, \cdots, \ell\}$, let $\pi_i \colon \R_t^N \to \R^{N_i}$ be the orthogonal projection.

\begin{afi}{$\pi_i \circ \hat\Phi = \hat\Phi \circ P_i$.}
	If $\z \in \mathcal{G}$, then $\z = \z_i + \z_i^\perp$, with $\z_i \in P_i(\mathcal{G})$ and $\z_i^\perp \in P_i(\mathcal{G})^\perp$. Thus $\z_i = P_i \z$ and $P_i\left(\z_i^\perp\right) = 0$, therefore $\left(\hat\Phi \circ P_i\right)(\z) = \hat\Phi\left(P_i(\z_i)\right) + \hat\Phi \left(P_i\left(\z_i^\perp\right)\right) = \hat\Phi\left(P_i(\z_i)\right)$.

	On the other side, since $\hat\Phi[P_i(\mathcal{G})] = \R^{N_i}$,  $\left(\pi_i \circ \hat\Phi\right)(\z) = \pi_i \left(\hat\Phi(\z_i)\right) + \pi_i\left(\hat\Phi \left(\z_i^\perp \right) \right) = \pi_i \left(\hat\Phi\left(P_i(\z_i)\right)\right) = \hat\Phi(P_i(\xi))$.
\end{afi}

So
\begin{align}
	& \pi_i\circ F_* = \pi_i \circ \hat\Phi|_{TM} = \hat\Phi\circ \left.P_i\right|_{TM} = \hat\Phi (\BR_i + \BS_i) = F_* \BR_i + \hat\Phi \BS_i, \label{piPhi1}\\
	& \pi_i \circ \hat\Phi|_{\mathcal E} = \hat\Phi \circ \left.P_i\right|_{\mathcal E} = \hat\Phi\left(\BS_i^\t + \BT_i\right) = F_*\BS_i^\t + \hat\Phi \BT_i, \label{piPhi2}\\
	& \pi_i \hat\Phi (\nu_j) = \hat\Phi P_i (\nu_j) = \delta_{ij} \hat\Phi(\nu_j). \label{piPhi3}
\end{align}

By Claim \ref{Pi parallel},
\begin{align*}
	& \ntil_X \tilde\Phi(\nu_i) = \ntil_X \hat\Phi(\nu_i) = \hat\Phi\n_X^{\cal G} \nu_i = \hat\Phi\left(-A^{\mathcal F}_{\nu_i} X + \n^{\mathcal F}_X \nu_i \right) =\\
	& \stackrel{\eqref{AF},\eqref{eq: F}}{=} \hat\Phi\left( -k_i\BR_i X - k_i\BS_i X \right) = -k_i\left(\hat\Phi\circ P_i\right)(X) = \\
	& = - k_i\left( \pi_i \circ \hat\Phi \right)(X) = -k_i \pi_i (F_* X).
\end{align*}
Therefore $\ntil_X \tilde\Phi(\nu_i) = - k_i\pi_i(F_* X)$.

Let $\z_i := \pi_i \circ F + \frac{1}{k_i} \tilde\Phi(\nu_i)$. Because of \eqref{piPhi3}, $\z_i \in \R^{N_i}$, and since $\ntil_X \tilde\Phi(\nu_i) = - k_i\pi_i(F_* X)$, $\ntil_X \z_i = 0$. Hence each $\z_i$ is constant in $\RN$.

Let $P := \sum\limits_{i=1}^\ell \z_i = F + \sum\limits_{i=1}^\ell \frac{1}{k_i}\tilde \Phi(\nu_i)$, $\tilde F := F - P$ and $\tilde{\z}_i := \pi_i \circ \tilde F + \frac{1}{k_i} \tilde \Phi(\nu_i)$. With the same calculations made for $\z_i$, we can show that $\tilde{\z}_i$ is constant.

\begin{afi}{$\tilde{\z}_i = 0$.}
	In deed,
	\begin{multline*}
		\tilde{\z}_i = \pi_i\left(\tilde F(x_0)\right) + \frac{1}{k_i}\tilde \Phi(\nu_i(x_0)) = \\
		= \pi_i(F(x_0)) - \pi_i(P) + \frac{1}{k_i}\tilde \Phi(\nu_i(x_0)) = \z_i - \pi_i(P).
	\end{multline*}
	But $\pi_i(P) = \pi_i \left(\sum\limits_{j=1}^\ell \z_j\right) = \z_i$. Therefore $\tilde{\z}_i = 0$.
\end{afi}

\begin{afi}{Replacing $F$ by $\tilde F$, if necessary, we can assume that $F(M) \subset \XO$, with $\o{i} \subset \R^{N_i}$ and $n_i + \upsilon(k_i) = N_i$.}	
	
	Replacing $F$ by $\tilde F$, if necessary, we can assume that each $\z_i = 0$, that is,
	\[\pi_i \circ F + \frac{1}{k_i}\tilde\Phi(\nu_i) = 0 \Rightarrow \left\|\pi_i \circ F \right\|^2 = \frac{\left\| \tilde\Phi(\nu_i)\right\|^2}{k_i^2} = \frac{k_i}{k_i^2} = \frac{1}{k_i}.\]
	Therefore, if $\o{i}$ is the connected component of $\ss{i} \subset \R^{N_i}$ such that $\pi_i(F(M))\subset \o{i}$, then $(\pi_i\circ F)(M) \subset \o{i}$ and $F(M) \subset \XO$.
\end{afi}

Since $F(M) \subset \XO$, there is an isometric immersion $f \colon M \to \XO$ such that $F = \imath \circ f$, where $\imath \colon \XO \to \RN$ is the canonical inclusion. Besides, because $\pi_i \circ F + \frac{1}{k_i}\tilde\Phi(\nu_i) = 0$, $\tilde\Phi(\nu_i) = -k_i(\pi_i\circ F) = \nu_i^F$.

Let $\nperp$ be the normal connection of $f$ in $T_f^\perp M$. We must to show that $\Phi = \tilde\Phi|_{\mathcal E}$ is a vector bundle isometry such $\Phi(\cE) = T_f^\perp M$ and that
\begin{align*}
	& \alpha^f = \Phi \alpha^\cE, && \nperp \Phi = \Phi \n^\cE, \\
	& \pi_i \circ f_* = f_* \BR_i + \Phi \BS_i, && \left.\pi_i\right|_{T^\perp M} = f_* \BS_i^\t \Phi^{-1} + \Phi \BT_i \Phi^{-1}.
\end{align*}

Since $T_F^\perp M = T_f^\perp M \op \spa\set{\pi_i\circ F}{k_i \ne 0}$ and $\tilde\Phi(\nu_i) = - k_i(\pi_i\circ F)$, then $\spa\set{\pi_i\circ F}{k_i \ne 0} = \tilde\Phi(\spa[\nu_1, \cdots, \nu_\ell])$, thus $\tilde\Phi(\cE) = T_f^\perp M$.

Because $J = \set{i \in \{1, \cdots, \ell\}}{k_i \ne 0}$,
\begin{align*}
	& \af{X}{Y} = \aF{X}{Y} - \sum_{i \in J} \interno{\aF{X}{Y}}{\nu_i^F}\frac{\nu_i^F}{\left\|\nu_i^F\right\|^2} = \\
	& = \tilde\Phi\left( \alpha^{\mathcal F}(X,Y) \right) - \sum_{i\in J} \interno{\tilde\Phi\left(\alpha^{\mathcal F}(X,Y)\right)}{\tilde\Phi(\nu_i)}\frac{\tilde\Phi(\nu_i)}{k_i} = \\
	& = \tilde\Phi\left( \alpha^{\mathcal F}(X,Y) - \sum_{i=1}^\ell\interno{\alpha^{\mathcal F}(X,Y)}{\nu_i}\frac{\nu_i}{k_i} \right) \stackrel{\eqref{eq: F},\eqref{AF}}{=} \tilde\Phi\left( \alpha^{\mathcal E}(X,Y) \right). \\
	& \nperp_X \Phi \xi = \nbarperp_X \tilde\Phi\xi - \sum_{i \in J} \interno{\nbarperp_X \tilde\Phi\xi}{\nu_i^F}\frac{\nu_i^F}{\left\|\nu_i^F\right\|^2} = \\
	& = \tilde\Phi \n_X^{\mathcal F} \xi - \sum_{i \in J} \interno{\tilde\Phi \n_X^{\mathcal F} \xi}{\tilde\Phi \nu_i}\frac{\tilde\Phi \nu_i}{k_i} = \\
	& = \tilde\Phi \left( \n_X^{\mathcal F} \xi - \sum_{i \in J} \interno{\n_X^{\mathcal F} \xi}{\nu_i}\frac{\nu_i}{k_i} \right) \stackrel{\eqref{eq: F}}{=} \tilde\Phi \left( \n_X^{\mathcal E} \xi \right). \\
	& \pi_i (f_* X) = \pi_i (F_* X) = \left(\pi_i\circ\hat\Phi\right)(X) = \left(\hat\Phi\circ P_i\right)(X) = \hat\Phi (\BR_i X + \BS_i X) = \\
	& = F_*\BR_i X + \tilde\Phi(\BS_i X) = f_* \BR_i X + \Phi (\BS_i X).
\end{align*}

Last, if $\z \in T^\perp M = \Phi(\cE)$, then $\z = \Phi \xi$, for some $\xi \in \cE$. Thus
\begin{align*}
	& \pi_i(\z) = \pi_i(\Phi\xi) = \left(\pi_i\circ\hat\Phi\right)(\xi) = \left(\hat\Phi\circ P_i\right)(\xi) = \hat\Phi\left(\BS_i^\t\xi + \BT_i\xi \right) =\\
	& =  F_*\BS_i^\t\xi + \tilde\Phi\BT_i\xi = f_*\BS_i^\t \Phi^{-1}(\z) + \Phi\BT_i\Phi^{-1}(\z).
\end{align*}
Therefore $\left.\pi_i\right|_{T^\perp M} = f_*\BS_i^\t\Phi^{-1} + \Phi\BT_i\Phi^{-1}$. \hfill $\Box$
		\subsection{Uniqueness}

Let $f$, $g$, and $\Phi$ be like in \textsl{(II)} of Theorem \ref{Bonnet}. So, lets consider $F := \imath\circ f$, $G := \imath\circ g$ and $\tilde\Phi \colon T_F^\perp M \to T_G^\perp M$ given by $\tilde\Phi \imath_* \xi = \imath_*\Phi \xi$, for all $\xi \in T_f^\perp M$, and $\tilde\Phi \left(\nu_i^F\right) = \nu_i^G$, with $\nu_i^F := -k_i (\pi\circ F)$ and $\nu_i^G := -k_i (\pi_i\circ G)$. Hence
\begin{align*}
	& \tilde\Phi \left(\aF{X}{Y}\right) \stackrel{\text{Lemma \ref{lem: aF}}}{=} \tilde\Phi\left(\imath_*\af{X}{Y} + \sum_{i=1}^\ell \interno{\BR_i^f X}{Y} \nu_i\right) = \\
	& = \imath_*\Phi \left(\af{X}{Y}\right) + \sum_{i=1}^\ell \interno{\BR_i^g X}{Y}\tilde\Phi \left(\nu_i^F\right) = \\
	& = \imath_*\alpha_g(X,Y) + \sum_{i=1}^\ell \interno{\BR_i^g X}{Y} \nu_i^G \stackrel{\text{Lemma \ref{lem: aF}}}{=} \alpha_G(X,Y).
\end{align*}

On the other side, if $\xi \in \G\left(T_f^\perp M\right)$, then
\begin{align*}
	& \tilde\Phi \left(^F\nbarperp_X \imath_*\xi\right) \stackrel{\eqref{nbarperpxi}}{=} \tilde\Phi \left( \imath_* \left(^f\nperp_X \xi\right) + \sum_{i=1}^\ell \interno{\BS^f_i X}{\xi}\nu^F_i\right) = \\
	& = \imath_* \Phi \left(^f\nperp_X \xi\right) + \sum_{i=1}^\ell \interno{\BS^f_i X}{\xi}\nu^G_i = \imath_* \left(^g\nperp_X \Phi\xi\right) + \sum_{i=1}^\ell \interno{\Phi\BS^f_i X}{\Phi\xi}\nu^G_i = \\
	& = \imath_* \left(^g\nperp_X \Phi\xi\right) + \sum_{i=1}^\ell \interno{\BS_i^g X}{\Phi\xi}\nu^G_i = {^G\nbarperp_X \imath_*\Phi\xi}.
\end{align*}
Therefore $\tilde\Phi\alpha_F = \alpha_G$ and $\tilde\Phi{^F\nbarperp} = {^G\nbarperp}\tilde\Phi$.

From the uniqueness part of Bonnet Theorem for isometric immersions in $\RN = \R_t^n$, there is an isometry $\tau \colon \RN \to\RN$ such that $G = \tau \circ F$ and $\left.\tau_*\right|_{T_F^\perp M} = \tilde\Phi$. Lets denote $\tau(Z) = B(Z) + C$, where $C \in \RN$ is constant and $B = \tau_* \in \O_t(n)$ is an orthogonal transformation. Thus $B\left(\nu_i^F\right) = \tau_*\left(\nu_i^F\right) = \tilde\Phi\left(\nu_i^F\right) = \nu_i^G$, that is, $B(\pi_i \circ F) = \pi_i \circ G$, if $k_i \ne 0$.

On the other side, since $G = \tau \circ F$, then $G_* = \tau_*F_* = BF_*$.
%\begin{align*}
%	& BF_* X = B\left(F_* \BR_i^f X + \imath_* \BS_i^fX\right) = \tau_*F_*\BR_i^f X + \tau_*\imath_*\BS_i^f X = \\
%	& = G_*\BR_i^g X + \tilde\Phi\imath_*\BS_i^f X = G_*\BR_i^g X + \imath_*\Phi\BS_i^f X = \imath_*g_*\BR_i^g X + \imath_*\BS_i^g X = G_* X.
%\end{align*}
%Thus $BF_* = G_*$.

Now, if $X \in TM$ and $\xi \in T_f^\perp M$, then
\begin{align*}
	& \interno{\Phi\BS_i^f X}{\xi} = \interno{\BS_i^f X}{\Phi^{-1} \xi} = \interno{X}{{^f\BS_i^\t} \Phi^{-1} \xi}; \\
	& \interno{\Phi\BS_i^f X}{\xi} = \interno{\BS_i^g X}{\xi} = \interno{X}{{^g\BS_i^\t} \xi}.
\end{align*}
Therefore $^f\BS_i^\t = {^g\BS_i^t}\Phi$. So, if $\xi \in T_f^\perp M$, then
\begin{align*}
	& B(\imath_*\xi) = B\left( \imath_*f_*{^f\BS_i^\t}\xi + \imath_*\BT_i^f\xi \right) = \tau_*F_*{^g\BS_i^\t}(\Phi\xi) + \tau_*\imath_*\BT_i^f\xi = \\
	& = G_*{^g\BS_i^\t}(\Phi\xi) + \tilde \Phi \imath_*\BT_i^f\xi = G_*{^g\BS_i^\t}(\Phi\xi) + \imath_* \Phi \BT_i^f\xi = \\
	& = G_*{^g\BS_i^\t}(\Phi\xi) + \imath_*\BT_i^g(\Phi\xi) = \imath_*\Phi\xi.
\end{align*}
As a result, $B\circ\imath_* = \imath_*\Phi$.

Given $\z \in \RN$ and $p \in M$, we know that there are $X \in T_p^f M$ and $\xi \in {^fT_p^\perp M}$ such that $\z = F_* X + \imath_* \xi + \sum\limits_{i \in J} \interno{\z}{\nu_i^F}\frac{\nu_i^F(p)}{k_i}$. In this way,
\begin{align*}
	& B(\pi_i \z) = B\left(\pi_i\left( F_*X + \imath_*\xi + \sum_{j \in J} \interno{\z}{\nu_j^F}\frac{\nu_j^F(p)}{k_j} \right)\right) = \\
	& = B\left( F_*\BR_i^f X + F_*{^f\BS_i^\t}\xi + \imath_*\BS_i^f X + \imath_*\BT_i^f \xi \right) + \interno{\z}{\nu_i^F} \frac{B\left(\nu_i^F(p)\right)}{k_i} = \\
	& = G_*\left( \BR_i^f X + {^f\BS_i^\t} \xi\right) + \imath_*\Phi\left( \BS_i^f X + \BT_i^f \xi\right) + \interno{\z}{\nu_i^F}\frac{\nu_i^G(p)}{k_i} = \\
	& = G_*\left( \BR_i^g X + {^g\BS_i^\t} \Phi(\xi)\right) + \imath_*\left( \BS_i^g X + \BT_i^g \Phi(\xi)\right) + \interno{\z}{\nu_i^F} \frac{\nu_i^G(p)}{k_i} = \\
	& = \imath_*\left( g_* \BR_i^g X + \imath_* \BS_i^g X \right) + \imath_* \left( g_*{^g\BS_i^\t} \Phi(\xi) + \BT_i^g \Phi(\xi)\right) + \interno{\z}{\nu_i^F} \frac{\nu_i^G(p)}{k_i} = \\
	& = \pi_i G_* X + \imath_* \pi_i (\Phi\xi) + \interno{\z}{\nu_i^F}\frac{\nu_i^G{p}}{k_i} = \\
	& = \pi_i (B(F_*X)) + \pi_i(\imath_*\Phi(\xi)) + \interno{\z}{\nu_i^F}\frac{\nu_i^G{p}}{k_i} = \\
	& = \pi_i \left( B(F_*X) + B(\imath_*\xi) + \sum_{j \in J}^\ell \interno{\z}{\nu_j^F}\frac{B(\nu_j^F(p))}{k_j}\right) = \pi_i (B(\z)).
\end{align*}
We have just showed that $B \circ \pi_i = \pi_i \circ B$, if $k_i \ne 0$. If $k_i = 0$, the calculations above are simpler.

Let $C_i := \pi_i(C)$. Since $B(\pi_i \circ F) = \pi_i \circ G$, then
\begin{multline*}
	B(\pi_i(F(p))) = \pi_i(G(p)) = \pi_i (B(F(p)) + C) = \\
	= \pi_i(B(F(p))) + C_i = B(\pi_i(F(p))) + C_i \Rightarrow C_i = 0.
\end{multline*}
Thus $\tilde\Phi\left(\o{i}\right) = B(\o{i}) = \o{i}$, for each $i \in \{1, \cdots, \ell\}$. Therefore $g = \varphi \circ f$, where $\varphi = \tilde\Phi|_{\hO}$ is a isometry of $\hO$ which fix each $\o{i}$. \hfill $\Box$

\bibliographystyle{acm}
%\newpage
%\addcontentsline{toc}{section}{\protect{Bibliography}}
%\bibliography{/home/brunomrs/Dropbox/Matematica/bibliografia}
\addcontentsline{toc}{section}{References}
\bibliography{/home/bruno/Dropbox/Matematica/bibliografia}
%
%\newpage
%\addcontentsline{toc}{section}{\protect{Índice Remissivo}}
%\printindex

\end{document}